\documentclass[11pt]{article}

\usepackage[pdftex]{graphicx,color}
\usepackage{amsmath,amssymb,amsfonts,amsthm,amssymb,latexsym}

\usepackage{eepic}
\usepackage{dsfont}

\newtheorem{theorem}{Theorem}[section]

\newtheorem{lemma}[theorem]{Lemma}

\newtheorem{definition}[theorem]{Definition}

\newtheorem{assumption}[theorem]{Assumption}

\def\R{{\mathbb R}}

\def\ve{{\varepsilon}}
\def\dv{{\text{div}}}

\def\v{{\bf v}}
\def\n{{\bf n}}

\def\S{\operatorname{S}\!}

\usepackage{subfigure}

\def\XXint#1#2#3{{\setbox0=\hbox{$#1{#2#3}{\int}$}
\vcenter{\hbox{$#2#3$}}\kern-.87\wd0}}

\def\XXiint#1#2#3{{\setbox0=\hbox{$#1{#2#3}{\int}$}
\vcenter{\hbox{$#2#3$}}\kern-1.05\wd0}}

\def\XXintt#1#2#3{{\setbox0=\hbox{$#1{#2#3}{\int}$}
\vcenter{\hbox{$#2#3$}}\kern-.72\wd0}}

\def\Xinttt#1{\mathchoice
{\XXinttt\displaystyle\textstyle{#1}}%
{\XXinttt\textstyle\scriptstyle{#1}}%
{\XXinttt\scriptstyle\scriptscriptstyle{#1}}%
{\XXinttt\scriptscriptstyle\scriptscriptstyle{#1}}%
\!\int}
\def\XXinttt#1#2#3{{\setbox0=\hbox{$#1{#2#3}{\int}$}
\vcenter{\hbox{$#2#3$}}\kern-.52\wd0}}

\def\XXintttr#1#2#3{{\setbox0=\hbox{$#1{#2#3}{\int}$}
\vcenter{\hbox{$#2#3$}}\kern-.6\wd0}}

\def\XXintttt#1#2#3{{\setbox0=\hbox{$#1{#2#3}{\int}$}
\vcenter{\hbox{$#2#3$}}\kern-.78\wd0}}

\def\ddashinttt{\Xinttt-}

\begin{document}

\title{\huge Homogenization of oxygen transport in biological tissues} 

\author{{Anastasios Matzavinos$\,^{\mathrm{a}}$ and Mariya Ptashnyk$\,^{\mathrm{b}}$}\medskip\\
\small $^{\mathrm{a}}$ Division of Applied Mathematics, Brown  University, Providence, RI 02912, USA \vspace*{-0.05cm}\\
\small $^{\mathrm{b}}$ Division of  Mathematics, University of Dundee, Dundee, DD1 4HN, UK  }
\date{}


\maketitle

\begin{abstract}
In this paper, we extend previous work on the mathematical modeling of oxygen transport in biological tissues \cite{Matzavinos:2009}.  Specifically, we include in the modeling process the arterial and venous microstructure within the tissue by means of homogenization techniques. We focus on the two-layer tissue architecture investigated in \cite{Matzavinos:2009} in the context of abdominal tissue flaps that are commonly used for reconstructive surgery. We apply two-scale convergence methods and unfolding operator techniques to homogenize the developed microscopic model, which involves different unit-cell geometries in the two distinct tissue layers (skin layer and fat tissue) to account for different arterial branching patterns.

%

\end{abstract}

\section{Introduction}

Flow of blood and delivery of oxygen within a tissue is an area of intense research activity \cite{Galdi:2008}. At the larger end of the scale, flows through branching vessels have been studied extensively \cite{Bowles:2005,Smith:2003B,Smith:2003}. At the capillary scale, detailed experimental and simulation studies of flows in the microvasculature have been carried out \cite{Pries:1996,Popel:2000,Chaplain:2002,Chaplain:2006}, taking into account such factors as changes in the apparent blood viscosity with vessel diameter, and separation of red blood cells and plasma at bifurcations \cite{Karniadakis:2012}. 

A more coarse-grained approach,  pursued by several authors, has been to  treat blood flow through the vascular network as akin to fluid flow through a porous medium. On a smaller scale, this approach was used by Pozrikidis and Farrow \cite{Pozrikidis:2003} to describe fluid flow within a solid tumor. More recent work by Chapman et al. \cite{Chapman:2008} extended this approach to consider flow through a rectangular grid of capillaries within a tumor, where the interstitium was assumed to be an isotropic porous medium, and Poiseuille flow was assumed in the capillaries. Through the use of formal asymptotic expansions, it was found that on the lengthscale of the tumor (i.e., a lengthscale much longer than the typical capillary separation) the behavior of the capillary bed was also effectively that of a porous medium. A more phenomenological approach was taken by Breward et al. \cite{Breward:2003}, who developed a multiphase model describing vascular tumor growth. Here, the tumor is composed of a mixture of tumor cells, extracellular material, and blood vessels, with the model being used to investigate the impact of angiogenesis or blood vessel occlusion on tumor growth. A similar model was used by O'Dea et al. \cite{Odea:2008} to describe tissue growth in a perfusion bioreactor.

Matzavinos et al. \cite{Matzavinos:2009} adopted a similar multiphase modeling approach to investigate the transport of oxygen in  abdominal tissue flaps, commonly used for plastic and reconstructive surgery. Among existing types of abdominal tissue flaps, the deep inferior epigastric perforator (DIEP) flap is a central component in the current practice of several reconstructive surgical procedures \cite{Granzow:2006}. Nonetheless, complications such as fat necrosis and partial (or even total) tissue flap loss due to poor oxygenation still remain an important concern. Gill et al. \cite{Gill:2004} reported that in their study of 758 DIEP cases, 12.9 percent of the flaps developed fat necrosis and 5.9 percent of the  patients had to return to the operating room. 
In view of these data, Matzavinos et al.  \cite{Matzavinos:2009} investigated computationally the level of oxygenation in a tissue given its size and shape and the diameters of the perforating arteries. The approach adopted in \cite{Matzavinos:2009} considered a multiphase mixture of tissue cells, arterial blood vessels, and venous blood vessels, distributed throughout a domain of interest according to specified volume fractions. 

In this paper, we improve upon the coarse-grained description of  \cite{Matzavinos:2009} by employing a homogenization approach that takes into account the detailed microstructure of arterial and venous blood vessels. The microscopic model under consideration tracks the flow of blood in a specified geometry of arteries and veins within a tissue flap and the transport of oxygen in arteries, veins, and tissue.  A two-layer tissue architecture is adopted that involves different unit-cell geometries (accounting for different arterial branching patterns) in the two distinct tissue layers. We apply a combination of two-scale convergence methods \cite{Allaire:1992,Nguetseng:1989} and unfolding operator techniques \cite{Cioranescu:2012,Cioranescu:2008_1,Cioranescu:2008} to homogenize the microscopic model. Our main results are Theorems \ref{5-1}, \ref{6-1}, \ref{th:macro_delta} and \ref{th:macro_ox_delta}  on the macroscopic equations for the blood velocity fields and the oxygen concentrations under different scaling assumptions for the two tissue layers. Moreover, in Theorems \ref{4-4} and \ref{4-8}, we generalize to thin domains existing convergence results for the periodic unfolding method.

Derivations of the effective macroscopic equations are important for an accurate numerical simulation of the oxygen distribution in biological tissue. 
 To address different structures of tissues, we consider two different cases which correspond to different scaling regimes: (i) the depth of the skin layer is of the same order as the representative size of the microstructure and (ii) the depth of the skin layer is much larger than the size of the microstructure, but much smaller than the depth of the fat tissue.  For both cases we obtain the Darcy law as the macroscopic equation for blood flow in fat tissue. In the skin layer, we reduce the interface at the boundary of the fat tissue layer to two dimensions and obtain the Darcy law with the force term defined by inflow or outflow of blood from the fat tissue layer. We obtain reaction-diffusion-convection and reaction-diffusion equations as macroscopic models for oxygen transport in blood and tissue oxygen concentrations, respectively. The transport of oxygen between tissue and arterial blood on the surface of the blood vessels is represented by the reaction terms in the macroscopic equations. Additionally, in the macroscopic equations for the oxygen concentration in the skin layer, we obtain the source terms defined by the inflow and outflow of oxygen from the fat tissue layer.  
 
 The main difference in the results for the two cases is that the unit cell problems are distinct, hence we obtain different effective permeability tensors and diffusion matrices. Thus we obtain different flow velocity and oxygen concentration transport  equations depending on the relationship between the thickness of the skin layer and the structure of the blood vessel networks. The macroscopic equations derived from the microscopic description of the processes take into account the microscopic structure of blood vessels network and provide a more realistic model for the oxygen transport in biological tissues.

The literature on the homogenization of fluid flows in porous media is vast (see, e.g., \cite{Allaire:1989,Arbogast:2006, Hornung,Mikelic:1991, Tartar:1980} and the references therein). Some representative results in this area are as follows. The macroscopic equations  for water flow between two porous media with different porosities were first  derived  in \cite{AndroJaeger}.  A multiscale analysis of the Stokes and Navier-Stokes problems in a thin domain was conducted in \cite{Paloka:2000}, where the authors considered  applications to lower-dimensional models in fluid mechanics. Various results on the multiscale analysis of reaction-diffusion-convection equations in perforated domains with reactions on the surfaces of the microstructure can be found in \cite{Ptashnyk:2013, Hornung, JaegerHornung:1991,  JaegerHornung:1994}.
Macroscopic equations for elliptic and parabolic reaction-diffusion equations posed in domains separated by a thin perforated layer (e.g., a sieve or a  membrane) were derived in \cite{Cioranescu:2008, Jaeger:2006}. From a mathematical perspective, the novelties of this paper include (a) the analysis of the flow between a fixed-size domain (fat tissue layer) and an $\ve$-thin layer (skin layer) under an appropriate scaling of the transmission conditions, and (b) a different scaling of the reaction-diffusion-convection equations than the one commonly used in the literature (see, e.g.,  \cite{Jaeger:2006}).

\begin{figure}
\begin{center}
\includegraphics[height=8.5cm]{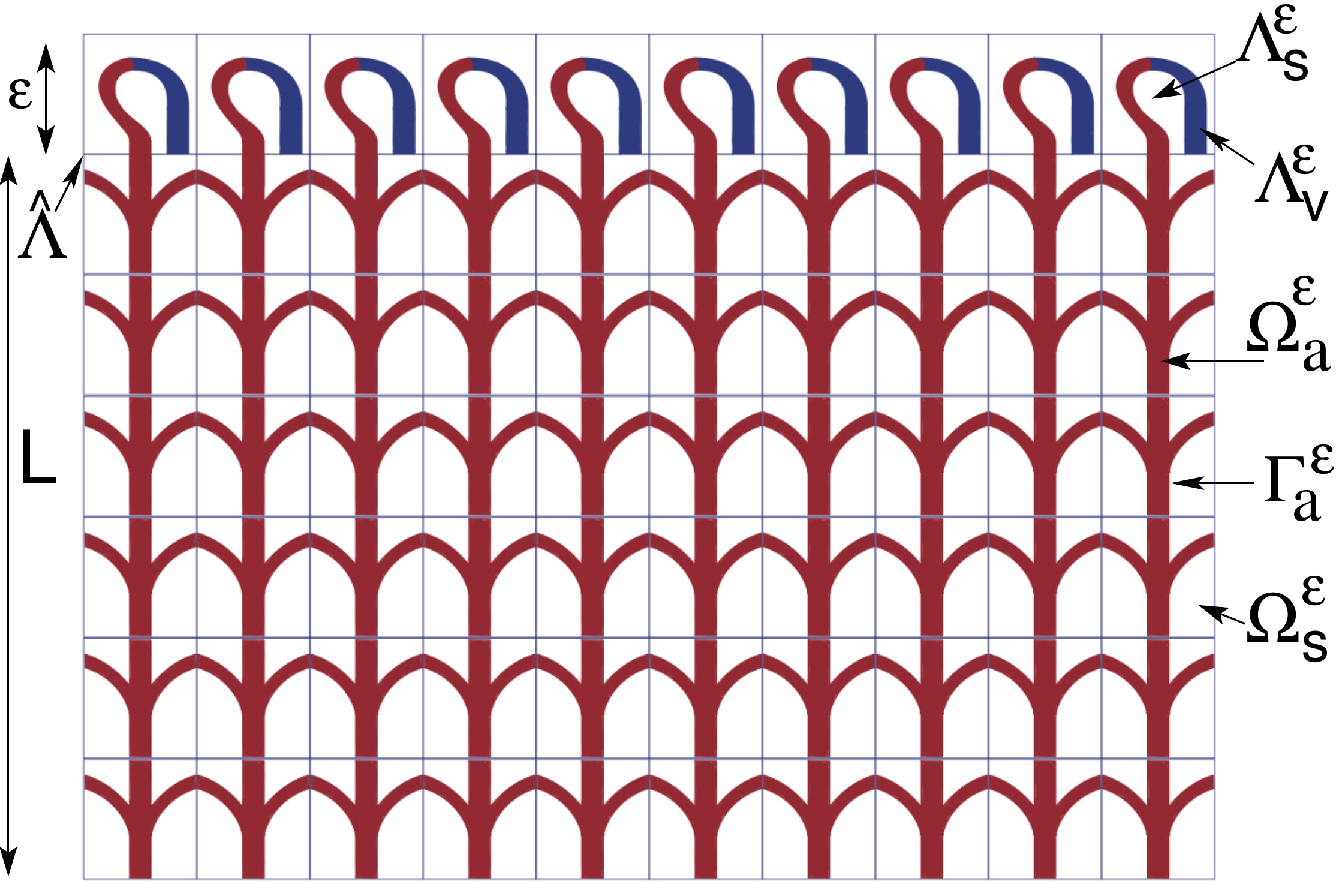}
\end{center}
\caption{Two dimensional schematic representation of a three-dimensional rectangular domain representing an abdominal tissue flap. The top layer of unit cells (denoted by $\Lambda^\ve$ in the text) corresponds to the dermic and epidermic layers of the skin, whereas the remainder of the domain (denoted by $\Omega$ in the text) corresponds to fat tissue. Only the arterial blood vessels are shown in the fat tissue layer. Arteries (in red) and veins (in blue) are shown in the skin tissue layer, which is characterized by the presence of arterial-venous connections, i.e. geometric regions where arteries and veins meet.  }\label{fig1}
\end{figure}
 
\begin{figure}
\begin{center}
\includegraphics[height=4.5cm]{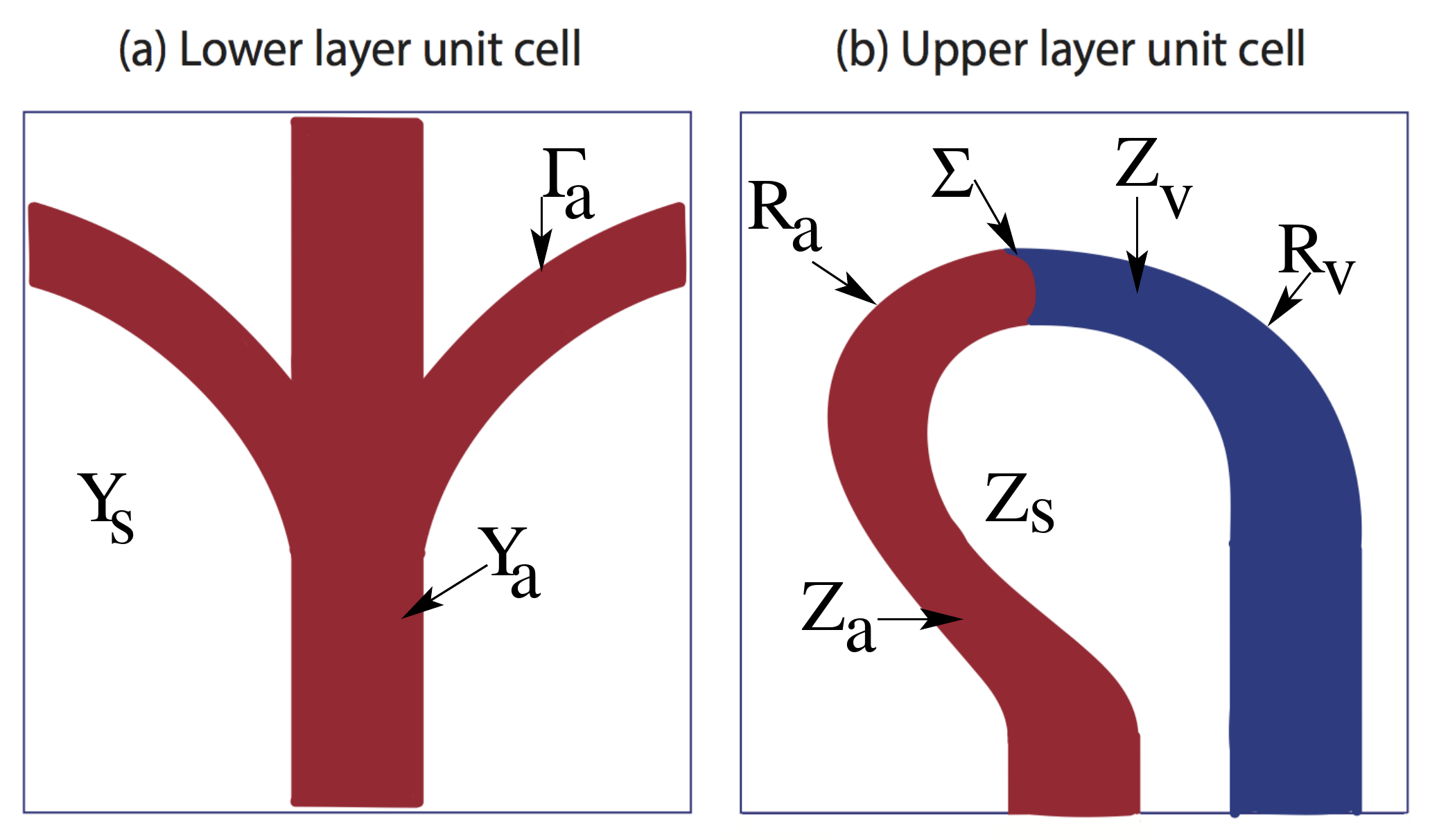}
\end{center}
\caption{Two-dimensional schematic representation of the two distinct, three-dimensional unit-cell geometries used in the microscopic model: (a) unit-cell geometry corresponding to the lower layer, i.e. the fat tissue layer; (b) unit-cell geometry corresponding to the upper layer, which represents the dermic and epidermic layers of the skin. Only the arterial blood vessels are shown in the fat tissue layer.}\label{fig2}
\end{figure}

The paper is organized as follows. In section 2, we collect the main results of the paper. In section 3, we formulate the microscopic model to be analyzed in the remainder of the paper, initially under the assumption that the depth of the top (skin) layer has the same length scale $\ve$ as the unit cell of the fat tissue layer. In section 4, we define the notion of weak solution used in the paper, and in section 5 we  provide {\it a priori} estimates for the solutions of the microscopic model and prove convergence results for the unfolding operator for functions defined in thin domains. These estimates are used in combination with an unfolding operator approach \cite{Cioranescu:2012,Cioranescu:2008_1,Cioranescu:2008} to prove the convergence of the solutions of the microscopic equations as $\ve\rightarrow 0$. In sections 6 and 7 we derive the homogenized, macroscopic equations for the blood velocity fields (in arteries and veins) and the oxygen concentrations (in arteries, veins, and tissue), respectively. Finally, in section 8, we relax some of the scaling assumptions of the previous sections, and we assume that the depth of the top (skin) layer is of a different length scale than the unit cell of the fat tissue layer.

\section{Formulation of the main results}

In this section, we collect the main results of the paper. The notation used is further explained in section \ref{mathmodel}. As discussed in the introduction, we are mainly concerned with the derivation of macroscopic equations for oxygen transport in a two-layer tissue architecture using different scaling assumptions for the distinct layers. The microscopic geometry that leads to the macroscopic models of this section is discussed in sections \ref{mathmodel} and \ref{section7}.

\subsection{Macroscopic coefficients  and  unit cell problems}\label{sec2-1}
First, we formulate the macroscopic coefficients and the unit cell problems that will be obtained in the derivation of the macroscopic equations.  We differentiate between two cases which correspond to skin tissue layers of different relative thicknesses (see section \ref{mathmodel} for an explanation of the terms involved). 

\subsection*{Case 1} If the thickness of the skin layer (see Fig. \ref{fig1}) is of the same order as the microscopic structure, then the macroscopic permeability matrices $\mathcal K_l$ and   $\hat{\mathcal K}$ for the blood flow  are defined by
\begin{align}\label{permeable}
\begin{aligned}
& \mathcal K_l^{ji}= \frac 1 {|Y|}\int_{Y_l} \omega_{l,j}^i (y) \, dy, \qquad & \hat{\mathcal K}^{jm}= \frac 1 {|\hat Z|}\int_{Z_{av}} \hat \omega_{j}^m(y)\, dy ,
\end{aligned}
\end{align}
where $\omega_{l}^i$ and  $\hat \omega^m$  are solutions of the unit cell problems 
\begin{equation}\label{eq:omega1}
\begin{cases}
-\mu \Delta_y \omega^i_l + \nabla_y \pi^i_l = \textbf{e}_i ,\quad \dv_y\, \omega^i_l =0   &\text{ in } Y_l, \quad i=1, \ldots, n,  \; \; l=a,v,   \\
\omega^i_l =0  \qquad  \text{ on } \Gamma_l, \quad  & \omega_l^i,\;   \pi_l^i \quad Y_l-\text{periodic},
\end{cases}
\end{equation}
 and 
\begin{equation}\label{eq:omega2}
\begin{cases}
-\mu \Delta_y \hat \omega^m + \nabla_y \hat \pi^m = \textbf{e}_m, \;   & 
\dv_y\, \hat \omega^m =0 \hspace{ 0.5 cm }  \text{ in }\, \,  Z_{av}, \; \;  m=1, \ldots, n-1,  \\
 (2 \mu S_y \hat \omega^m - \hat \pi^m I)\n\times \n=0, \quad & 
\hat\omega^m\cdot \n =0  \hspace{ 0.7 cm } \text{ on }  \, \,    \hat Z_{av}^0, \\
\hat \omega^m = 0 \quad \text{ on } \, \,  R_{av}\cup \hat Z_{av}^1,  &
    \hat\omega^m, \; \hat \pi^m \quad \hat Z-\text{periodic}. 
\end{cases}
\end{equation}
The macroscopic diffusion coefficients $\mathcal A_l$ and $ \mathcal{\hat A}_m$ in the limit equations for the oxygen concentration  are given by
\begin{equation}\label{macro_vv} 
\begin{aligned}
 &  \mathcal A_{l}^{ij}= \frac 1 {|Y|}\int_{Y_l} \Big[ D_l^{ij}(y) +  \sum_{k=1}^n D_l^{ik}(y)\frac{ \partial w_l^j }{\partial y_k} \Big] dy,   \\
& \hat{\mathcal A}_m^{ij}=  \frac 1 {|\hat Z|}\int_{Z_{m}} \Big[ \hat D^{ij}_m(y) +  \sum_{k=1}^n \hat D^{ik}_m(y) \frac{\partial \hat w^j_m}{\partial y_k} \Big] dy,  
\end{aligned}
 \end{equation}
 where $l=a,v,s$ and $\;m=av, s$. The functions  $w_l$ and $\hat w_m$  are solutions of the unit cell problems
 \begin{equation}\label{UnitCell_Ox_1}
\begin{cases}
-\dv_y( D_l(y) (\nabla_y w_l^j + \textbf{e}_j)) =0 & \, \text{in } Y_l,   \quad \text{ for } l=a,v,s,  \;  \; j=1, \ldots, n,  \\
D_l(y)( \nabla_y w_l^j + \textbf{e}_j)\cdot \n =0  & \, \text{on } \Gamma_l,  \quad  
w_l^j \quad  \, Y-\text{periodic}
\end{cases}
\end{equation}
and
 \begin{eqnarray}\label{unit_hat_1}
\begin{cases}
-\dv_y (\hat D_m(y) (\nabla_y \hat w^j_m +\textbf{e}_j)) = 0 &  \text{ in }  Z_{m},   \\
\hat D_m(y) (\nabla_y\hat w^j_m+ \textbf{e}_j) \cdot \n = 0 &  \text{ on } R_{av},  \text{ on }  \hat Z^0_{m} \cup \hat Z^1_{m} ,  \\
 \hat w_m^j  \quad \qquad \qquad   \hat Z-\text{periodic}, &\text{ for } m=av,s \; \text{ and } \; j=1, \ldots, n-1.
 \end{cases}
\end{eqnarray} 

\subsection*{Case 2} If the thickness of the skin layer is (a) considerably larger than the characteristic size of the microscopic  structure and (b) significantly smaller than the thickness of  the fat tissue layer, then, in the fat tissue layer,  the macroscopic permeability tensors $ \mathcal K_l$, $l=a,v$, and the macroscopic diffusion coefficients  $\mathcal{A}_\alpha$,  $\alpha=a,v,s$, are identical  to those  defined in  \eqref{permeable} and \eqref{macro_vv}. However, different permeability and diffusion coefficients are obtained for the macroscopic equations describing the blood flow and oxygen transport in the skin layer. Specifically, we obtain
\begin{eqnarray}\label{macro_coef_delta}
   \widetilde {\mathcal K}^{ji} = \frac 1{|\widetilde Z|} \int_{\widetilde Z_{av}} \widetilde \omega_j^i (y) dy, \quad   {\widetilde {\mathcal A}}^{ij}_m= \frac 1 {|\widetilde Z|}\int_{\widetilde Z_m} \Big[\hat D^{ij}_{m}(y) + \sum_{k=1}^n \hat D^{ik}_{m}(y) \partial_{y_k} \widetilde w^j_m(y)\Big] dy,
 \end{eqnarray}
where   $m=av,s$, and  $\widetilde \omega^i$ and  $\widetilde w^j_m$ are solutions of the unit cell problems 
   \begin{equation}\label{unit_cell_7_2}
  \begin{cases}
  - \mu \Delta_y \widetilde \omega^i + \nabla_y \widetilde \pi^i = \textbf{e}_i,   \qquad  &   \dv_y \, \widetilde\omega^i =0  \qquad \text{ in } \quad \widetilde Z_{av}, \\
 \phantom{-} \widetilde \omega^i = 0  \qquad \qquad  \text{ on } \quad \widetilde  R_{av}, \qquad &
 \widetilde \omega^i, \,\widetilde \pi^i \qquad  \widetilde Z-\text{periodic}.
 \end{cases}
 \end{equation}
and
\begin{equation}\label{unit_cell_2_hat}
\begin{cases}
 -\dv_y(\hat D_{m}(y) ( \nabla_y \widetilde w^j_m + \textbf{e}_j))=0 \;  \qquad &\text{ in } \widetilde Z_{m},  \qquad m=av, s, \\
 \phantom{-}\, \hat D_{m}(y) (\nabla_y \widetilde w^j_m +  \textbf{e}_j)\cdot \n=0 \qquad \quad& \text{ on } \widetilde R_{av}, \qquad  
\widetilde w^j_m \quad \widetilde Z-\text{periodic} \; .
\end{cases}
\end{equation}

\subsection{Macroscopic equations for velocity fields and oxygen concentrations} 
Given the definitions of the macroscopic coefficients and the unit cell problems in section \ref{sec2-1}, we are now in a position to state the theorems that are proved in the remainder of the paper.  We start by defining the spaces
$$
\begin{aligned}
& H({\rm div}; \Omega) = \{ v \in L^2(\Omega)^n, \; {\rm div  }\, v \in L^2(\Omega) \}, \\
& W(\Omega) = \{ w \in H^1(\Omega) , \; w = 0 \; \text{ on } \;  \Gamma_D \}.
\end{aligned}
$$
\subsection*{Case 1}
The main results of the paper under the scaling assumptions of Case 1, as discussed in section \ref{sec2-1}, are theorems \ref{5-1} and  \ref{6-1}. These provide the macroscopic equations for the blood velocity fields (in arteries and veins) and oxygen concentrations (in arteries, veins, and tissue) respectively. The notation used in the statements of the theorems is introduced in section \ref{mathmodel}.
 
\begin{theorem}\label{5-1}   The sequence of solutions of the microscopic model \eqref{Stokes_Omega}--\eqref{Stokes_ContC}
converges to functions  $\v^0_l \in H({\rm div} ; \Omega)$,  $p_l- p_l^0 \in W(\Omega)$, $\hat \v^0_{av} \in L^2(\hat \Lambda)$, and $\hat p\in H^1(\hat \Lambda)$ that satisfy the macroscopic equations 
\begin{align}\label{macro_macro_1}
\begin{aligned}
& \v^0_l=- \mathcal K_l \nabla p_l,  && \qquad  &&  \dv \,(\mathcal K_l \nabla p_l) = 0 && \text{ in } \Omega , \\
&  p_l   =  \hat p &&  \text{ on } \hat \Lambda \;,\\
& p_l=p_l^0  &&   \text{ on } \Gamma_D,  \qquad &&  \mathcal K_l \nabla p_l \cdot \n = 0 && \text{ on } \partial \hat \Omega\times(-L,0), \\
\end{aligned}
\end{align}
where $l=a,v$, and 
\begin{align}\label{macro_macro_2}
\begin{aligned}
&\hat  \v^0_{av}= - 2\hat {\mathcal K} \nabla_{\hat x} \hat p, \qquad    2\dv_{\hat x} (\hat {\mathcal K} \nabla_{\hat x} \hat p ) =  \mathcal K_a \nabla p_a \cdot \n +\mathcal K_v \nabla p_v \cdot \n && \text{ in } \hat\Lambda, \\
& \hat {\mathcal K} \nabla_{\hat x} \hat p \cdot \n =0 && \text{ on } \partial \hat \Lambda.
\end{aligned}
\end{align}
\end{theorem}

\begin{theorem}\label{6-1}
The sequence of solutions of the microscopic model \eqref{Diff_AV}--\eqref{Diff_InitC} converges to a solution of 
the macroscopic equations 
\begin{equation}\label{macro_cl}
\begin{aligned}
& \theta_l \partial_t c_l- \dv(\mathcal A_l  \nabla c_l - \v_l^0 c_l) = \;  \lambda_l \gamma_l  (c_s-c_l) && \text{ in } \Omega_T, \\
&\theta_s \partial_t c_s- \dv(\mathcal A_s   \nabla c_s  ) =\sum_{l=a,v}  \lambda_l \gamma_l (c_l-c_s)- \theta_s \ddashinttt_{Y_s}  d_s \, dy \, c_s &&
 \text{ in } \Omega_T, \\
&c_l(t,\hat x,0)  =\;  \hat c(t,\hat x), \qquad \qquad  c_s(t,\hat x,0)  = \hat c_s(t,\hat x) &&  \text{ on } \hat \Lambda_T, \\
& (\mathcal A_l \nabla c_l - \v_l^0 c_l)\cdot \n=\;  0 \hspace{ 1.6 cm }   \text{  on }\;  (\partial \Omega \setminus (\hat\Lambda\cup \Gamma_D))\times (0,T), \\
 &c_l (t,x) = \;  c_{l,D}(t,x)  &&   \text{ on } \Gamma_{D,T}, \\
 & \mathcal A_s \nabla c_s \cdot \n = 0 \hspace{3 cm }  \text{  on } \; (\partial \Omega \setminus \hat\Lambda)\times (0,T), \\
 &c_l(0,x) =\;   c_l^0(x), \qquad \qquad  c_s(0,x) =  c_s^0(x)   &&  \text{ in } \Omega, 
\end{aligned}
\end{equation}
where $\theta_l = |Y_l|/|Y|$, $\gamma_l=|\Gamma_l|/|Y|$,  $l=a, v$,  and  $\theta_s = |Y_s|/ |Y|$.  Moreover, in the domain  
 $\hat \Lambda_T$, we have 
\begin{eqnarray}\label{macro_hat_c}
\begin{aligned}
 &\hat \theta_{av}  \partial_t \hat c-   \dv_{\hat x}(  \mathcal{\hat A}_{av} \nabla_{\hat x} \hat c - \hat \v^0_{av} \, \hat c) = 
\mathcal R_{av} (\hat c_s- \hat c)-\sum_{l=a,v}  (\mathcal A_l \nabla c_l-\v_l^0 c_l)\cdot \n,   \\
 & \hat \theta_s \partial_t \hat c_s - \dv_{\hat x}(  \mathcal{\hat A}_s  \nabla_{\hat x} \hat c_s) = 
\mathcal R_{av} (\hat c- \hat c_s)-  \mathcal A_s  \nabla c_s\cdot \n - \hat \theta_s \ddashinttt_{Z_s}  \hat d_s\,  dy \,  \hat c_s, \; \; \\
& (\mathcal{\hat A}_{av} \nabla_{\hat x} \hat c - \hat\v^0_{av} \hat c)\cdot \n =0   \quad \text{ on } \; (0,T)\times \partial \hat \Lambda, \quad \;\;  \hat c(0, \hat x) = \hat c^0(\hat x) \quad \;   \text{ in }\;  \hat \Lambda, \\
& \mathcal{\hat A}_s \nabla_{\hat x} \hat c_s\cdot \n =0 \qquad \qquad\quad \text{ on } \;    (0,T)\times\partial \hat \Lambda,  \quad\; \;
 \hat c_s(0, \hat x) = \hat c^0_s(\hat x)    \quad  \text{ in } \;  \hat \Lambda, 
 \end{aligned}
\end{eqnarray}
where $\hat \theta_{av}= |Z_{av}|/|\hat Z|$, $\hat \theta_s = |Z_s|/|\hat Z|$,  and $\mathcal R_{av} = \hat \lambda_a |R_a|/|\hat Z| +  \hat \lambda_v |R_v|/|\hat Z|$.
The macroscopic  transport velocities ${ \v}^0_l$, $\hat \v^0_{av}$ are given by 
\begin{equation}\label{macro_veloc_ve}
{ \v}^0_l(x) = \frac 1 {|Y|} \int_{Y_l} \v_l (x,y) dy, \quad {\hat  \v}^0_{av}(\hat x) = \frac 1 {|\hat Z|}  \int_{Z_{av}} \hat \v_{av} (\hat x,y)  dy, \quad l = a,v. 
\end{equation} 
The solutions of equations \eqref{macro_cl}--\eqref{macro_hat_c}  satisfy  $c_l-c_{l,D} \in L^2(0,T; W(\Omega))\cap H^1(0,T; L^2(\Omega))$ for $l=a,v$, and  $c_s \in L^2(0,T; H^1(\Omega))\cap H^1(0,T; L^2(\Omega))$. Moroever,  $\hat c, \hat c_s \in L^2(0,T; H^1(\hat \Lambda))\cap H^1(0,T; L^2(\hat \Lambda))$. Finally, $\hat c$, $\hat c_s \in  L^\infty(\hat \Lambda_T)$ and  $c_l \in L^\infty(\Omega_T)$ for $l=a,v,s$.
\end{theorem}

\subsection*{Case 2}
If we consider the scaling assumptions of Case 2, then we have to introduce two parameters: a parameter  $\ve>0$ that characterizes the length scale of the microstructure and a parameter $\delta>0$ that represents the thickness of the skin tissue layer. 

We  first derive a system of ``intermediate" equations by letting $\ve\rightarrow 0 $ while keeping $\delta$ fixed, as follows.   
\begin{theorem}\label{th:macro_ve}
As $\ve \to 0$  the sequence of  solutions of the microscopic model  given by \eqref{Stokes_Omega}, \eqref{Stokes_BC_O}, \eqref{Stokes_ContC}, and \eqref{Stokes_Lambda_delta}--\eqref{Stokes_BC_L_delta} converges to functions $\overline \v^\delta_{l} \in H({\rm div}; \Omega)$, $p^\delta_l- p_l^0 \in W(\Omega)$, $\widetilde \v^\delta_{av} \in H({\rm div}; \Lambda_\delta)$,  and $\hat p^\delta \in H^1(\Lambda_\delta)$, respectively, with $l=a,v$, that satisfy the macroscopic model 
\begin{equation}\label{macro_ve_1}
\begin{aligned}
&\overline \v^\delta_{l} = - \mathcal K_l \nabla p_l^\delta, &&   \dv (\mathcal K_l \nabla p_l^\delta) = 0 \quad  && \text{ in }\;  \Omega, \\
& \widetilde \v^\delta_{av} =  - \widetilde {\mathcal K} \nabla \hat p^\delta, &&   \dv ( \widetilde {\mathcal K} \nabla \hat p^\delta) = 0  &&  \text{ in }\;  \Lambda_\delta,    \\
&   \mathcal K_v \nabla p_v^\delta \cdot \n+ \mathcal K_a \nabla p_a^\delta \cdot \n =\displaystyle\frac{1}{\delta} \widetilde {\mathcal K}  \nabla   \hat p^\delta \cdot \n,   && \qquad  p_l^\delta =  \hat p^\delta &&  \text{ on } \; \hat \Lambda \; ,\\
& \mathcal  K_l \nabla p_l^\delta\cdot \n =0 \quad   \text{ on }\;  \partial \Omega \setminus(\Gamma_D \cup \hat \Lambda) \; , && \qquad  p^\delta_l = p^0_l  &&   \text{ on } \;  \Gamma_D, \\
&\widetilde {\mathcal K}  \nabla  \hat p^\delta\cdot \n =0  \, \quad    \text{ on } \; \partial  \Lambda_\delta \setminus \hat \Lambda.
\end{aligned}
\end{equation}
\end{theorem}

\begin{theorem}\label{thm7-2}
As $\ve \to 0$ the sequence of solutions of the microscopic equations   \eqref{Diff_AV}--\eqref{Diff_InitC} with $\delta$ instead of $\ve$ in the transmission conditions  converges to functions
$c_l^\delta- c_{l,D} \in L^2(0,T; W(\Omega))$,  $c_s^\delta \in L^2(0,T; H^1(\Omega))$, $c^\delta_l \in H^1(0,T; L^2(\Omega))$,   and $\hat c^\delta_j \in L^2(0,T; H^1(\Lambda_{\delta}))\cap H^1(0,T; L^2(\Lambda_{\delta}))$  that satisfy the macroscopic problem 
\begin{equation}
\begin{aligned}\label{macro_av_ve}
&\theta_l \partial_t  c^\delta_l - \dv(\mathcal{A}_l \nabla c^\delta_l - {\overline \v}^{\delta}_l  c^\delta_l) =  \lambda_l \gamma_l(c^\delta_s-  c^\delta_l), \qquad  && \text{ in } \Omega_T, \\
&\widetilde \theta_{av} \partial_t \hat c^\delta_{av} - \dv(\mathcal{\widetilde A}_{av} \nabla \hat c^\delta_{av} - \widetilde{\v}^\delta_{av} \hat c^\delta_{av}) = \mathcal R_{av}(\hat c^\delta_s- \hat c^\delta_{av}), \qquad && \text{ in } \Lambda_{\delta,T},  \\
& c_l^\delta= \hat c^\delta_{av}, \qquad \sum_{l=a,v} (\mathcal{A}_l \nabla c_l^\delta - \overline{\v}^\delta_l  c_l^\delta)\cdot \n  = \displaystyle\frac{1}{\delta} (\mathcal{\widetilde A}_{av} \nabla \hat c^\delta_{av} - \widetilde {\v}_{av}^\delta  \hat c^\delta_{av})\cdot \n   && \text{ on } \hat \Lambda_T, \\
&  (\mathcal{A}_l \nabla c_l^\delta - {\overline \v}_l^\delta  c_l^\delta)\cdot \n  = 0  \hspace{2. cm }    \text{ on }(\partial \Omega\setminus  (\hat \Lambda\cup \Gamma_D)) \times (0,T), \\
& c_l^\delta = c_{l,D}  \hspace{4.3 cm }   \text{ on } \Gamma_{D} \times (0,T),  \\
&(\mathcal{\widetilde A}_{av} \nabla \hat c^\delta_{av} - \widetilde{\v}^\delta_{av} \hat c^\delta_{av}) \cdot \n = 0 \hspace{1.2 cm }  \text{ on } (\partial  \Lambda_\delta\setminus \hat \Lambda)\times(0,T), \\
&  c_l^\delta(0,x)= c^0_l(x) \qquad   \text{ in } \Omega,  \hspace{ 1.9 cm }    \hat c_{av}^{\delta}(0, x)= \hat c^{\delta,0}(x) && \text{ in } \Lambda_\delta,
 \end{aligned}
\end{equation}
where $l=a,v$ and $j=av, s$, and  
\begin{equation}\label{macro_s_ve}
\begin{aligned}
&\theta_s \partial_t  c^\delta_s - \dv(\mathcal{A}_s  \nabla c^\delta_s ) =\sum_{l=a,v}  \lambda_l \gamma_l(c^\delta_l- c^\delta_s)  - \theta_s \ddashinttt_{Y_s}  d_s dy\,  c_s^\delta &&  \text{ in } \Omega_T, \\
 &\widetilde \theta_s \partial_t \hat c^\delta_s - \dv(\mathcal{\widetilde A}_s  \nabla \hat c^\delta_s ) =  \mathcal R_{av}(\hat c^\delta_{av}- \hat c^\delta_s) - \widetilde \theta_s \ddashinttt_{\widetilde Z_s} \hat d_s  dy \,  \hat c_s^\delta &&  \text{ in } \Lambda_{\delta, T},  \\
& c^\delta_s = \hat c^\delta_s, \qquad \qquad 
  \mathcal{A}_s  \nabla c^\delta_s \cdot \n =\frac 1 \delta  \mathcal{\widetilde A}_s \nabla \hat c^\delta_s  \cdot \n  \; \; &&  \text{ on } \hat \Lambda_T,  \\
 &  \mathcal{A}_s \nabla c^\delta_s\cdot \n  = 0  \quad\;\;  \text{ on }(\partial \Omega\setminus  \hat \Lambda)\times(0,T), \quad  \quad \;\; c^\delta_s(0,x)= c^{0}_s(x)  && \text{ in } \Omega, \\
& \mathcal{\widetilde A}_s \nabla \hat c^\delta_s  \cdot \n = 0 \quad \;\;   \text{ on }  (\partial  \Lambda_\delta \setminus \hat \Lambda)\times(0,T), \quad \quad \hat c^\delta_s(0,x)= \hat c^{\delta, 0}_{s}(x)  && \text{ in }  \Lambda_\delta.
 \end{aligned}
\end{equation}
Here the following notation has been used:
\begin{eqnarray*} 
&&\widetilde \theta_m= \frac{|\widetilde{Z}_{m}|}{|\widetilde{Z}|},  \;  m=av, s, \;   \theta_l=\frac{ |Y_l|}{|Y|}, \; 
 \mathcal R_{av} =  \frac{\hat \lambda_v|\widetilde R_v| + \hat \lambda_a|\widetilde R_a|}{|\widetilde{Z}|}, \;  
\gamma_l =\frac{ |\Gamma_l|}{|Y|}, \;  l=a,v,s,
\end{eqnarray*}
and  the macroscopic transport velocities are defined as 
\begin{equation}\label{macro_veloc_delta}
\overline \v^\delta_{l} (x) = \frac 1 {|Y|} \int_{Y_l} \v^\delta_l (x,y) \, dy,  \quad 
  \widetilde \v^\delta_{av} (x) = \frac 1 {|\widetilde Z|} \int_{\widetilde Z_{av}} \hat \v^\delta_{av} (x,y) \, dy, \qquad l = a,v.
\end{equation}
\end{theorem}

Given these ``intermediate" results,  we derive the final macroscopic equations by letting $\delta\rightarrow0$ in \eqref{macro_ve_1}, as follows.
\begin{theorem}\label{th:macro_delta}
As $\delta \to 0$  the sequence of  solutions  of the  equations  \eqref{macro_ve_1} converges to functions  $\overline \v_l \in H({\rm div}; \Omega)$,  $p_l -p_l^0\in W(\Omega)$, $\widetilde \v_{av} \in L^2(\hat \Lambda)$, and $\hat p \in H^1(\hat \Lambda)$, respectively, with $l=a,v$,  
that satisfy the problem 
\begin{equation}\label{macro_delta_ve}
\begin{aligned}
&\overline \v_l =- \mathcal K_l \nabla p_l, \qquad \quad\;   \dv \,( \mathcal K_l \nabla p_l) = 0  && \text{ in } \; \Omega,  \\
& p_l(\hat x, 0)  =\hat p(\hat x)    \hspace{ 1 cm }  \text{ on } \; \hat \Lambda, \hspace{3 cm } p_l= p^0_l   &&\text{ on }  \;\Gamma_D, \\
&\widetilde \v_{av} = - \widetilde {\mathcal K} \nabla_{\hat x}  \hat p, \qquad  \quad  \dv_{\hat x}(\widetilde {\mathcal K} \nabla_{\hat x}  \hat p) = \mathcal K_a \nabla p_a \cdot \n + \mathcal K_v \nabla p_v \cdot \n   \qquad && \text{ on } \; \hat \Lambda, \\
&  \mathcal K_l \nabla p_l\cdot \n =0 \hspace{ 1.1 cm}   \text{ on }\;  \partial \Omega \setminus(\Gamma_D \cup \hat \Lambda), \; \;  \qquad  \widetilde {\mathcal K} \nabla_{\hat x}  \hat p\cdot \n =0   \qquad   &&  \text{ on } \; \partial \hat \Lambda.
\end{aligned}
\end{equation}
\end{theorem}

\begin{theorem}\label{th:macro_ox_delta} 
As   $\delta \to 0$ we obtain the macroscopic problem 
\begin{equation}\label{macro_delta_av}
\begin{aligned}
&\theta_l \partial_t  c_l - \dv(\mathcal{A}_l \nabla c_l - \overline{\v}_l  c_l) =  \lambda_l \gamma_l(c_s-  c_l), \; && \text{ in } \Omega_T \; , \\
&\theta_s \partial_t  c_s - \dv(\mathcal{A}_s \nabla c_s ) = \sum_{l=a,v} \lambda_l \gamma_l(c_l- c_s)
 -\theta_s \ddashinttt_{Y_s}  d_s(t,y) dy  \;  c_s   &&  \text{ in } \Omega_T, \\
& c_l(t,\hat x, 0)  = \hat c_{av} (t, \hat x) \qquad  \; \; \; c_s(t,\hat x, 0)  = \hat c_s (t, \hat x)  && \text{ on } \hat \Lambda_T, \qquad \; \\ 
& (\mathcal{A}_l \nabla c_l - \overline{\v}_l  c_l)\cdot \n  = 0  \qquad \text{ on }(\partial \Omega\setminus( \hat \Lambda\cup \Gamma_D))\times(0,T) \; , \\
& c_l(t,x)= c_{l,D} && \text{ on }  \Gamma_{D,T} , \\
 & \mathcal{A}_s \nabla c_s\cdot \n  = 0  \hspace{2 cm }  \text{ on }(\partial \Omega\setminus  \hat \Lambda)\times(0,T), \\
& c_l(0,x)= c_l^0(x) \hspace{ 2.5 cm }  c_s(0,x)= c^0_{s}(x) \qquad && \text{ in } \Omega,
 \end{aligned}
\end{equation}
where $l=a,v$, and   in  $\hat \Lambda_T$ we have 
\begin{equation}\label{macro_delta_s}
\begin{aligned}
& \widetilde \theta_{av} \partial_t \hat c_{av} - \dv_{\hat x}(\mathcal{\widetilde A}_{av} \nabla \hat c_{av} - \widetilde{\v}_{av} \hat c_{av}) = \mathcal R_{av} (\hat c_s- \hat c_{av})  - \sum_{l=a,v} (\mathcal{A}_l \nabla c_l - \overline{\v}_l  c_l)\cdot \n,  \\
 & \widetilde \theta_s \partial_t \hat c_s - \dv_{\hat x} (\mathcal{\widetilde A}_s \nabla \hat c_s ) =  \mathcal R_{av}(\hat c_{av}- \hat c_s) -\mathcal{A}_s \nabla c_s \cdot \n  -  \widetilde \theta_s    \ddashinttt_{\widetilde Z_s} \hat d_s(t,y)  dy\; \hat c_s,  \\
& (\mathcal{\widetilde A}_{av} \nabla \hat c_{av} - \widetilde {\v}_{av} \hat c_{av}) \cdot \n = 0,  \qquad  \mathcal{\widetilde A}_s \nabla \hat c_s  \cdot \n = 0 \qquad   \text{ on } \partial \hat \Lambda_T\; , \\
&\hat c_{av}(0,\hat x)= \hat c^0(\hat x) \hspace{ 3.1 cm }   \hat c_s(0,\hat x)= \hat c^0_{s}(\hat x)  \qquad  \text{ in } \hat \Lambda.
 \end{aligned}
\end{equation}
 Moreover, the solutions of equations  \eqref{macro_delta_av} and \eqref{macro_delta_s} satisfy $c_l -c_{l,D} \in L^2(0,T; W(\Omega))$, $c_s \in L^2(0,T; H^1(\Omega))\cap  H^1(0,T; L^2(\Omega))$, $c_l \in H^1(0,T; L^2(\Omega))$,  
$\hat c_j \in L^2(0,T; H^1(\hat \Lambda)) \cap H^1(0,T; L^2(\hat \Lambda))$  for
$l=a,v$, $j=av,s$.
\end{theorem}

We remark that the structure of the macroscopic equations for the blood velocity fields is the same  in both cases, i.e. in Theorem~\ref{5-1} and Theorem~\ref{th:macro_delta}. However, the permeability tensors for the flow in the skin layer  are different, since they are determined by solutions of different unit cell problems; see equations \eqref{eq:omega1}, \eqref{eq:omega2}, and \eqref{unit_cell_7_2}. These results reflect the differences in the microscopic structure and the microscopic equations for the skin layer in the two different cases. We also remark that the factor of $2$ in the macroscopic equations \eqref{macro_macro_2} is specific to Case 1.

A similar situation appears in  the macroscopic equations for oxygen transport. In both cases, we obtain the same structure for the equations; see Theorem~\ref{6-1} and Theorem~\ref{th:macro_ox_delta}. However, the macroscopic diffusion coefficients and transport velocities are different, as manifested by equations \eqref{macro_vv} and  \eqref{macro_coef_delta} for the diffusion coefficients and equations  \eqref{macro_veloc_ve} and \eqref{macro_veloc_delta} for the transport velocities. Again, these results reflect the differences in the microscopic structure of the the skin tissue layer in the two different cases.

Finally, the ``intermediate" system  obtained in Case 2, when we let $\varepsilon\rightarrow 0$ but keep $\delta$ fixed, represents the macroscopic equations for the blood flow and oxygen concentration   in the two  domains with different microscopic structures (skin layer and fat tissue layer). 

\section{The microscopic model}\label{mathmodel}

We now introduce the microscopic model that leads to the asymptotic (macroscopic) results stated in the previous section. 
As in \cite{Matzavinos:2009} we adopt a three-dimensional rectangular geometry for a DIEP tissue flap with a two-layer tissue architecture.  The approach in this paper differs from that in \cite{Matzavinos:2009} in that the geometry of the vascular microstructure is explicitly defined. A two-dimensional schematic representation of the three-dimensional geometry used is shown in Fig. \ref{fig1}. The top layer of unit cells in Fig. \ref{fig1} corresponds to the dermic and epidermic layers of the skin, whereas the remainder of the domain corresponds to fat tissue.

We denote the fat tissue layer by $\Omega=\hat \Omega\times (-L,0)$, with some  $L>0$ and $\hat \Omega\subset \R^2$. The top (skin) layer is assumed to be thin as compared to the fat tissue layer and is denoted by $\Lambda^\ve= \hat \Omega\times(0,\ve)$ with $\Lambda^1= \hat \Omega\times(0, 1)$, $\hat \Lambda = \hat\Omega\times\{0 \}$. The small positive parameter $\ve$ represents both the scale of the unit cell describing the arterial branching pattern and the depth of the skin layer (this assumption is relaxed in section 8). 

The vascular microstructure is assumed to differ in the two layers of the domain. Specifically, $\Omega$ is constructed by a periodic arrangement of a (scaled) unit cell $\overline Y= \overline Y_a\cup \overline Y_v\cup \overline Y_s$, where $Y_a$,  $Y_v$, and $Y_s$ partition $Y$ into the geometric domains of arteries, veins, and tissue, respectively. Figure \ref{fig2}(a) shows an example of such a unit cell that represents a specific arterial branching pattern for the fat tissue layer. We define the domains occupied by  arteries, veins and tissue in $\Omega$ as $\Omega_a^\ve =\text{Int} \big(\cup_{\xi \in \mathbb Z^3}  \ve (\overline Y_{a}+ \xi) \big)\cap \Omega$, 
$\Omega_v^\ve = \text{Int}\big(\cup_{\xi \in \mathbb Z^3} \ve (\overline Y_{v}+ \xi)\big)\cap \Omega$, and $\Omega_s^\ve = \text{Int}\big(\cup_{\xi \in \mathbb Z^3} \ve (\overline Y_{s}+\xi) \big)\cap \Omega$, respectively. The small parameter  $\ve$ corresponds to the size of the arterial microscopic structure. In particular,  $\ve$ is the ratio between the size of the periodically repeating unit cell and the size of the whole tissue domain.

Similarly, we define a (different) unit cell $\overline Z=\overline Z_{a}\cup \overline Z_{v}\cup \overline Z_{s}$ that describes the arterial and venous geometry in  $\Lambda^\ve$. We  define $\Lambda_a^\ve = \text{Int} \big(\cup_{\eta \in \mathbb Z^2}  \ve (\overline Z_{a} + (\eta, 0)) \big)\cap \Lambda^\ve$,
$\Lambda_v^\ve =\text{Int} \big(\cup_{\eta \in \mathbb Z^2}  \ve (\overline Z_{v} + (\eta, 0)) \big) \cap \Lambda^\ve$, and 
$\Lambda_s^\ve = \text{Int} \big(\cup_{\eta \in \mathbb Z^2}  \ve (\overline Z_{s} + (\eta, 0)) \big) \cap \Lambda^\ve$ as the domains in $\Lambda^\ve$ of arteries, veins, and tissue respectively. Figure \ref{fig2}(b) shows an example of a unit cell for $\Lambda^\ve$. Throughout the paper, it is assumed that the skin layer $\Lambda^\ve$  is characterized by the presence of arterial-venous connections that facilitate the exchange of blood between the arterial and venous systems (see, e.g., \cite{physiology,Matzavinos:2009}). A simple example of an arterial-venous connection is shown in Fig. \ref{fig2}(b)

\begin{table*}
\caption{Macroscopic domains (see text for details)\bigskip}\label{table21}
\begin{center}
\begin{tabular}{|l|l|}
\hline
Notation & Description\\
\hline
\hline
$\displaystyle \Omega = \hat \Omega\times (-L,0)$  & Fat tissue layer\\  \hline
$\displaystyle  \hat \Lambda = \hat \Omega \times \{ 0\}$ & Upper boundary of $ \Omega$\\  \hline
$\displaystyle  \Lambda^\ve = \hat \Omega \times (0, \ve)$  & Skin layer (scaling of section 3)  \\  \hline
$\displaystyle \Lambda_\delta = \hat \Omega \times (0, \delta)$ &  Skin layer (scaling of section 8) \\  \hline 
\end{tabular}
\end{center}
\end{table*}
\begin{table*}
\caption{Unit cell domains (see text for details)\bigskip}\label{table22}
\begin{center}
\begin{tabular}{|l|l|}
\hline
Notation & Description\\
\hline
\hline
$\overline Y= \overline Y_a\cup \overline Y_v\cup \overline Y_s$ & Unit cell for $ \Omega$\\ \hline
 $Y_a, Y_v, Y_s\subset Y$ & Open subsets  with Lipschitz  boundaries  $\Gamma_a$ and  $\Gamma_v$, \\
  & $Y_a \cap Y_v = \emptyset$ \\ \hline
 $\overline Z = \overline Z_{a}\cup \overline Z_{v}\cup \overline Z_{s}$ & Unit cell for  $\Lambda^\ve$\\ \hline
 $Z_a, Z_v , Z_s\subset Z$ & Open subsets  with Lipschitz boundaries  $R_a$ and   $R_v$, \\
 & $Z_a\cap Z_v = \emptyset$  \\  \hline
$\overline{\widetilde Z}  =\overline{\widetilde Z_{a}}\cup\overline{ \widetilde Z_{v}}\cup\overline{\widetilde Z_{s}}$ & Unit cell for  $\Lambda_\delta$\\ \hline
$\widetilde Z_a,  \widetilde Z_v, \widetilde Z_s \subset \widetilde Z$ & Open subsets with Lipschitz boundaries $\widetilde R_a$ and   $\widetilde R_v$,\\
& $\widetilde Z_a\cap \widetilde Z_v = \emptyset$\\ \hline   
\end{tabular}
\end{center}
\end{table*}
\begin{table*}
\caption{Microscopic domains (see text for details)\bigskip}\label{table23}
\begin{center}
\begin{tabular}{|l|l|}
\hline
Notation & Description\\
\hline
\hline
$\displaystyle \Omega_a^\ve = \text{Int}\Big(\bigcup_{\xi \in \mathbb Z^3}\ve(\overline Y_{a}+ \xi)\Big)\cap \Omega$ & Arteries in fat tissue layer \\
\hline
$\displaystyle \Omega_v^\ve =\text{Int}\Big( \bigcup_{\xi \in \mathbb Z^3}\ve(\overline Y_{v}+ \xi)\Big)\cap \Omega$  & Veins in fat tissue layer\\ \hline
$\displaystyle \Omega_s^\ve =\text{Int}\Big( \bigcup_{\xi \in \mathbb Z^3}\ve(\overline Y_{s}+ \xi)\Big)\cap \Omega$  & Tissue domain\\ \hline
$\displaystyle \Lambda_a^\ve =\text{Int}\Big( \bigcup_{\eta \in \mathbb Z^2}\ve(\overline Z_{a}+ (\eta,0))\Big)\cap \Lambda^\ve$ & Arteries in skin layer (section 3) \\ \hline
$\displaystyle \Lambda_v^\ve = \text{Int}\Big(\bigcup_{\eta \in \mathbb Z^2}\ve(\overline Z_{v}+ (\eta,0))\Big) \cap \Lambda^\ve$ & Veins in skin layer (section 3)\\ \hline
$\displaystyle \Lambda_s^\ve = \text{Int}\Big( \bigcup_{\eta \in \mathbb Z^2}\ve(\overline Z_{s}+ (\eta,0))\Big) \cap \Lambda^\ve$ & Tissue in skin layer (section 3)\\ \hline
$\displaystyle \Lambda_a^\delta = \text{Int}\Big( \bigcup_{\xi \in \mathbb Z^3}\ve(\overline{\widetilde Z_{a}}+ \xi)\Big)\cap \Lambda_\delta$ & Arteries in skin layer (section 8) \\ \hline
$\displaystyle \Lambda_v^\delta =\text{Int}\Big( \bigcup_{\xi \in \mathbb Z^3}\ve(\overline{\widetilde Z_{v}}+ \xi)\Big)\cap \Lambda_\delta$ & Veins in skin layer  (section 8)\\ \hline
$\displaystyle \Lambda_s^\delta =\text{Int}\Big( \bigcup_{\xi \in \mathbb Z^3}\ve(\overline{\widetilde Z_{s}}+ \xi)\Big)\cap \Lambda_\delta$ & Tissue domain in skin layer (section 8)\\ \hline
\end{tabular}
\end{center}
\end{table*}
\begin{table*}
\caption{Microscopic boundaries (see text for details)\bigskip}\label{table24}
\begin{center}
\begin{tabular}{|l|l|}
\hline
Notation & Description\\
\hline
\hline
$\displaystyle \Gamma_a^\ve = \bigcup_{\xi \in \mathbb Z^3}\ve(\Gamma_{a}+ \xi)\cap \Omega$  & Boundaries of  arteries in fat tissue layer \\  \hline
$\displaystyle \Gamma_v^\ve = \bigcup_{\xi \in \mathbb Z^3}\ve(\Gamma_{v}+ \xi)\cap \Omega$ & Boundaries of veins in fat tissue layer\\  \hline
$\displaystyle R_a^\ve = \bigcup_{\eta \in \mathbb Z^2}\ve(R_{a}+ (\eta, 0))\cap \Lambda^\ve$ & Boundaries of  arteries in skin layer (section 3)\\  \hline
$\displaystyle R_v^\ve = \bigcup_{\eta \in \mathbb Z^2}\ve(R_{v}+ (\eta, 0))\cap \Lambda^\ve$ & Boundaries of veins in skin layer (section 3)\\  \hline
$\displaystyle \widetilde R_a^\ve = \bigcup_{\xi \in \mathbb Z^3}\ve(\widetilde R_{a}+ \xi)\cap \Lambda_\delta$ & Boundaries of  arteries in skin layer  (section 8) \\  \hline
$\displaystyle \widetilde R_v^\ve = \bigcup_{\xi \in \mathbb Z^3}\ve(\widetilde R_{v}+ \xi)\cap \Lambda_\delta$ & Boundaries of veins in skin layer  (section 8)\\  \hline
\end{tabular}
\end{center}
\end{table*}

We first consider that the depth of the skin layer is of order $\ve$. This condition is later modified in section 8.
In the arteries and veins located in $\Omega$, blood is assumed to flow with velocities $\v_a^\ve(x)$ and $\v_v^\ve(x)$, respectively, according to the Stokes equation with zero-slip boundary conditions. Specifically, we let $p_a^\ve(x)$ and $p_v^\ve(x)$ denote the arterial and venous pressures, respectively, and we assume that $(\v_a^\ve, p_a^\ve)$ and $(\v_v^\ve, p_v^\ve)$ satisfy 
\begin{eqnarray}\label{Stokes_Omega}
\begin{cases}
 -\ve^2 \mu \,\Delta \v_l^\ve +\nabla p_l^\ve=0 \; , \qquad  
  \dv\,\v^\ve_l=0  &  \quad \text{ in } \Omega_l^\ve  ,\medskip\\
  \v_l^\ve=0  & \quad \text{ on } \Gamma_l^{\ve},
\end{cases}
\end{eqnarray}
where $l= a, v$, and $\Gamma_a^{\ve}$ and $\Gamma_v^{\ve}$ denote the outer surface of arteries and veins, respectively, in $\Omega$. As usual, the scaling in the viscosity term is such that the velocity field has a non-trivial limit as $\ve\rightarrow 0$ (see, e.g., \cite{Hornung}). Similarly, we assume that in the skin tissue layer $\Lambda^\ve$,  $(\hat \v_a^\ve, \hat p_a^\ve)$ and   $(\hat \v_v^\ve, \hat p_v^\ve)$  satisfy
\begin{eqnarray}\label{Stokes_Lambda}
\begin{cases}
 - \ve^2 \mu\,  \Delta \hat\v^\ve_l +  \nabla \hat p_l^\ve= 0 \; ,   \quad   \dv \, \hat\v_l^\ve=0  &  \quad \text{ in }  \Lambda_l^\ve, \medskip\\
  \hat\v_l^\ve=0  & \quad \text{ on }  R_l^{\ve} \; , 
  \end{cases}
\end{eqnarray}
where $l= a, v$, and $R_a^{\ve}$ and $R_v^{\ve}$ denote the outer surface of arteries and veins, respectively, in $\Lambda^\ve$.
We define  $\partial\Omega= \Gamma_D\cup(\partial \hat\Omega \times (-L, 0))\cup \hat \Lambda$, where $\Gamma_D$ denotes the lower horizontal boundary of the fat tissue layer, and  impose the  boundary conditions 
\begin{equation}\label{Stokes_BC_O}
p_l^\ve = p_l^0, \; \; \v_l^\ve \times \n =0 \;\; \text{ on } \Gamma_D\cap\partial \Omega_{l}^\ve,  
\quad \;\;
\v_l^\ve = 0 \; \; \text{ on }(\partial \hat\Omega \times (-L, 0)) \cap \partial \Omega_l^\ve, 
\end{equation}
where $l=a,v$. We consider Dirichlet boundary conditions for the blood velocities on  $\partial \Lambda^\ve= (\partial \hat\Omega\times(0,\ve))\cup \hat\Lambda \cup (\hat\Omega \times \{\ve\} )$
\begin{equation}\label{Stokes_BC_L}
\hat \v_l^\ve =0   \quad \text{ on } \partial \hat\Omega\times(0,\ve)\cap \partial \Lambda_l^\ve , \qquad
\hat \v_l^\ve =0  \quad \text{ on } \hat\Omega\times\{\ve\}\cap \partial \Lambda_l^\ve , \qquad l=a,v, 
\end{equation}
and we impose transmission conditions on $\hat \Lambda$:
\begin{eqnarray}\label{Stokes_TransC}
\begin{cases}
(-2\,\ve^2  \mu  \operatorname{S}\!\v_l^\ve+p_l^\ve I)\cdot \n= (-2\, \ve^2 \mu   \S\hat \v_l^\ve+ \hat p_l^\ve I)\cdot \n &\quad \text{on }\partial\Omega_l^\ve\cap \hat \Lambda\; , \medskip\\
\v_l^\ve =   \frac{1}{\ve}\hat\v_l^\ve &  \quad \text{on }\partial\Omega_l^\ve\cap \hat \Lambda \; ,
\end{cases}
\end{eqnarray}
where $l=a,v$, and $\S \textbf{u}$ denotes the symmetric gradient $\S \textbf{u}=1/2(\partial_{x_i} u_j + \partial_{x_j} u_i)_{ij}$. 
The $\ve^{-1}$ scaling  in the velocity boundary condition balances the blood velocity field in the skin layer with the depth of the layer.

We let $\Sigma^\ve$ denote the arterial-venous connections in $\Lambda^\ve$. In other words, $\Sigma^\ve$ denotes the $n-1$-dimensional surfaces, where arteries and veins meet in $\Lambda^\ve\subset\mathbb{R}^n$. We impose continuity conditions for blood velocities and forces on $\Sigma^\ve$, as follows.
\begin{eqnarray}\label{Stokes_ContC}
\begin{cases}
(-2\ve^2  \mu \S\hat \v_a^\ve+ \hat p_a^\ve I)\cdot \n=(-2\ve^2  \mu \S\hat \v_v^\ve+ \hat p_v^\ve I)\cdot \n & \text{ on } \Sigma^\ve, \medskip\\
\hat \v_a^\ve  = \hat \v_v^\ve  \qquad & \text{ on }\Sigma^\ve \; .
\end{cases}
\end{eqnarray}

The oxygen concentrations in the tissue and the arterial and venous blood within the fat tissue layer are denoted by $c_s^\ve(x,t)$, $c_a^\ve(x,t)$, and $c_v^\ve(x,t)$, respectively. Similarly, the corresponding concentrations in the skin tissue layer are denoted by $\hat c_s^\ve(x,t)$,  $\hat c_a^\ve(x,t)$, and $\hat c_v^\ve(x,t)$, respectively. Oxygen in the blood is transported by the flow and diffuses within the fluid. Hence, the equations describing oxygen transport in the blood are given by
\begin{eqnarray}\label{Diff_AV}
\begin{cases}
 \partial_t c_l^\ve- \dv(D_l^\ve \nabla c_l^\ve- \v_l^\ve c_l^\ve)=0  & \quad \text{ in } \Omega_l^\ve\times(0,T) \; , \medskip\\
\frac 1 \ve \partial_t \hat c_l^\ve  -\frac 1 \ve  \dv( \hat D_l^\ve \nabla\hat c_l^\ve -  \hat\v_l^\ve \hat c_l^\ve)= 0 & \quad \text{ in }  \Lambda_l^\ve\times(0,T) \; ,
\end{cases}
\end{eqnarray}
where $l=a,v$. Oxygen diffuses within the tissue with diffusion coefficient $D_s^\ve$, and it is assumed to decay and/or be consumed by the tissue cells at a rate proportional to oxygen concentration. The equations for $c_s^\ve(x,t)$ and $\hat c_s^\ve(x,t)$ are then
\begin{eqnarray}\label{Diff_Tis}
\begin{cases}
 \partial_t c_s^\ve- \dv(D_s^\ve \nabla c_s^\ve)=- d_s^\ve c_s^\ve & \quad \text{ in } \Omega_s^\ve\times(0,T)\; , \medskip \\
\frac 1 \ve \partial_t \hat c_s^\ve -\frac 1 \ve  \dv(\hat D_s^\ve \nabla \hat c_s^\ve)= - \frac 1 \ve  \hat d_s^\ve \hat c_s^\ve  & \quad \text{ in }
 \Lambda_s^\ve\times(0,T) \; . 
 \end{cases}
\end{eqnarray}

The boundary conditions on the surface of the blood vessels describe the flux of oxygen from the blood into the tissue at a rate proportional to the difference in the oxygen concentrations. 
\begin{eqnarray}\label{Diff_BC_micro_AV}
\begin{cases}
 (D_l^\ve\nabla c_l^\ve  - \v^\ve c_l^\ve )\cdot \n = - \ve\lambda_l (c_l^\ve- c_s^\ve) &\quad \text{on } \Gamma_l^{\ve}\times(0,T), \medskip\\
(\hat D_l^\ve  \nabla \hat  c_l^\ve -  \hat \v^\ve_l \hat c^\ve_l) \cdot \n = - \ve\hat \lambda_l(\hat c_l^\ve- \hat c_s^\ve) & \quad \text{on } R_l^{\ve}\times(0,T) ,
\end{cases}
\end{eqnarray}
for $l=a,v$, and
\begin{eqnarray}\label{Diff_BC_micro_Tis}
\begin{cases}
 D_s^\ve\nabla c_s^\ve \cdot \n = \ve\lambda_l(c_l^\ve- c_s^\ve)  & \quad \text{on } \Gamma_l^{\ve}\times(0,T), \medskip\\
\hat  D_s^\ve\nabla \hat c_s^\ve \cdot \n = \ve\hat \lambda_l(\hat c_l^\ve- \hat c_s^\ve) & \quad \text{on }  R_l^{\ve}\times(0,T) , 
\end{cases}
\end{eqnarray}
where the constants $\lambda_l$ and  $\hat \lambda_l$, $l=a,v$,  are the oxygen permeability coefficients of the  arterial and venous blood vessels. 

 In addition to the exchange of oxygen between blood vessels and tissue, oxygen in arterial blood is transported to the venous system through the arterial-venous connections in the upper (skin) layer of the domain. In the following, we assume continuity of concentrations and fluxes at the arterial-venous connections $\Sigma^\ve$
\begin{eqnarray}\label{Diff_ContC}
\hat c_a^\ve = \hat c_v^\ve,   \quad 
( \hat D_a^\ve\nabla \hat c_a^\ve - \hat \v_a^\ve \hat c_a^\ve)\cdot\n = (\hat D_v^\ve\nabla \hat  c_v^\ve - \hat \v_v^\ve \hat c_v^\ve)\cdot\n 
\quad \text{on } \, \, \Sigma^\ve \times(0,T)
\end{eqnarray}
We also impose transmission conditions between the fat tissue layer and the skin layer
\begin{eqnarray}\label{Diff_TransC}
\begin{cases}
c_l^\ve = \hat c_l^\ve,   \;
(D_l^\ve\nabla c_l^\ve -  \v_l^\ve c_l^\ve) \cdot\n =  \frac 1 \ve (\hat D_l^\ve\nabla \hat c_l^\ve-   \hat \v_l^\ve \hat c_l^\ve )\cdot\n  &
 \text{on }(\partial\Omega_l^\ve\cap \hat \Lambda)\times(0,T), \medskip\\
c_s^\ve =  \hat c_s^\ve,   \; 
D_s^\ve\nabla c_s^\ve \cdot\n =\frac 1 \ve \hat D_s^\ve\nabla \hat c_s^\ve \cdot\n  &
 \text{on } (\partial\Omega_s^\ve\cap \hat \Lambda)\times (0,T),\qquad
\end{cases}
\end{eqnarray}
where $l=a,v$. We remark that the $\ve^{-1}$ scaling in \eqref{Diff_TransC} balances the oxygen flux terms in the skin layer with the depth of the layer.  

At the external boundaries we consider  Dirichlet boundary conditions that define the prescribed oxygen concentration at the arterial/venous blood vessel boundaries  and zero-flux boundary conditions at the tissue boundaries: 
\begin{eqnarray}\label{Diff_BC}
\begin{cases}
c_l^\ve= c_{l, D}  & \quad \text{ on }  (\Gamma_D \cap \partial \Omega_l^\ve)\times (0,T),   \quad \text{for }  l=a,v, \medskip\\
D_l^\ve \nabla c^\ve_l \cdot \n =0 & \quad \text{ on }  \big((\partial \hat \Omega\times(-L,0))\cap \partial \Omega^\ve_l\big)\times (0,T),  \quad \text{for }  l=a,v,\medskip\\
D_s^\ve\nabla c_s^\ve \cdot \n =0 & \quad \text{ on } \big(\Gamma_D\cup (\partial \hat \Omega \times (-L,0) ) \cap \partial \Omega_s^\ve\big)\times  (0,T), \medskip\\
 \hat D_l^\ve\nabla \hat c_l^\ve  \cdot \n =0  & \quad \text{ on }  ((\hat\Omega\times\{\ve\}\cup \partial\hat \Omega\times(0,\ve))  \cap \partial \Lambda_l^\ve)\times(0,T),  \quad \text{for }  l=a,v, s. 
 \end{cases}
\end{eqnarray}
The initial conditions for the oxygen concentrations are given by 
\begin{equation}\label{Diff_InitC}
c^\ve_l(0,x) = c^0_{l}(x) \quad \text{ in } \Omega_l^\ve, \quad \;  \hat c^\ve_l(0,x) =  \hat c^{\ve,0}_{l}(x) \quad \text{ in } \Lambda^\ve_l,  \quad  \text{ where } \; \;  l=a,v,s.
\end{equation} 

In the following, we make use of the notation   $\Omega_T=\Omega\times (0,T)$,  $\Omega^\ve_{l,T}=\Omega_l^\ve\times (0,T)$, $\Gamma_{D,T} = \Gamma_D\times (0,T)$, $\partial \Omega_T= \partial\Omega\times(0,T)$, 
and $\Lambda^\ve_{l,T}=\Lambda_l^\ve\times (0,T)$  for $l=a,v,s$. We also use the notation
$\hat\Lambda_T= \hat\Lambda\times (0,T)$, $\partial \hat \Lambda_T = \partial \hat \Lambda\times (0,T)$, and $\hat Z=Z\cap\{x_n=0\}$.
The diffusion coefficients $D^\ve_l$, $\hat D^\ve_l$ and the oxygen degradation rates $d^\ve_s$, $\hat d_s^\ve$ are defined by $Y$-periodic and $\hat Z$-periodic functions $D_l$, $d_s$ and $\hat D_l$, $\hat d_s$, respectively. Specifically,  
$$D^\ve_l(x)= D_l(x/\ve),  \hat D^\ve_l(x)= \hat D_l(x/\ve),  d^\ve_s(t,x)= d_s(t,x/\ve), \mbox{and } \hat d^\ve_s(t,x)= \hat d_s(t,x/\ve),$$ 
for a.a. $t\geq 0$, $x\in \Omega$,  $x \in \Lambda^\ve$, and $l=a,v,s$. Finally, the following assumption is made throughout the paper.
\begin{assumption}\label{assumption}
The following hold:
\begin{enumerate}
\item[(i)] The diffusion coefficients $D_l\in L^\infty(Y)$, $\hat D_l \in L^\infty(Z) $ are uniformly elliptic, i.e., 
$(D_l(y) \xi, \xi)\geq D_0|\xi|^2$, \, $(\hat D_l(z) \xi, \xi)\geq \hat D_0|\xi|^2$ for all $\xi \in \mathbb R^n$ and a.a.
$y\in Y$ and $z\in Z$, where $l=a,v,s$, and $D_0 >0$, $\hat D_0>0$.

\item[(ii)] It is assumed that $d_s, \, \partial_t d_s\in L^\infty((0,T)\times Y)$ and $\hat d_s, \, \partial_t \hat d_s\in L^\infty((0,T)\times Z)$. 

\item[(iii)] With respect to the initial conditions, it is assumed that  $c_{l}^0 \in  H^2(\Omega)\cap L^\infty(\Omega)$,  
$\hat c^{\ve,0}_{l} \in H^2(\Lambda^\ve)\cap L^\infty(\Lambda^\ve)$, $c_{l}^0(x) \geq 0$  for $x \in \Omega$,   $\hat c^{\ve,0}_{l}(x) \geq 0$  for $x \in \Lambda^\ve$, $l=a,v,s$,  $\hat c^{\ve,0}_a = \hat c^{\ve,0}_v= \hat c^{\ve,0}$, and $c_{l}^0(x) = c_{l,D}(0,x)$ on $\Gamma_D$, where $l=a,v$. Moreover,  
\begin{eqnarray*}
&& \ve^{-1} \|\hat c^{\ve,0}_{l} \|^2_{H^2(\Lambda^\ve)}  \leq C,\qquad \|\hat c^{\ve,0}_{l} \|_{L^\infty(\Lambda^\ve)} \leq C, \\ 
&&c_{l}^0(x) =  \hat c^{\ve,0}_{l}(x), \quad 
D_l^\ve(x) \nabla c_{l}^0(x) \cdot \n =  \frac 1 \ve \hat D_l^\ve(x) \nabla \hat c^{\ve,0}_{l}(x)\cdot \n \, \,  \text{ on } \, \,  \partial\Omega^\ve_l\cap \hat \Lambda.
\end{eqnarray*}

\item[(iv)] It is assumed that the boundary conditions for the oxygen concentration in arteries and veins satisfy $c_{l,D} \in H^1(0,T; H^2(\Omega))\cap L^\infty(\Omega_T)$, $\partial_t c_{l,D} \in L^\infty(\Omega_T)\cap H^1(0,T; L^2(\Omega))$,  $c_{l,D}(t,x) \geq 0$ a.e. in $\Omega_T$, and $c_{l,D}(t,x) = 0$ on $\hat \Lambda_T$, for $l=a,v$.

\item[(v)] Finally, it is assumed that   $\mu >0$, $\lambda_l> 0$, $\hat \lambda_l >0$, and     $p_l^0 > 0$  for $l=a,v$.
\end{enumerate}
\end{assumption}

\section{Weak solutions and functional spaces}\label{weak_sol}

The microscopic system under consideration consists of equations \eqref{Stokes_Omega}--\eqref{Stokes_ContC} 
 for the blood velocity fields and pressures in arteries and veins, and equations
   \eqref{Diff_AV}--\eqref{Diff_InitC} for the oxygen concentrations in arteries, veins, and tissue. We now define a notion of weak solution for the system of equations \eqref{Stokes_Omega}--\eqref{Diff_InitC} and  the functional spaces that are used in this paper. We start by defining the spaces 
\begin{eqnarray*}
V(\Omega_l^\ve)&=&\big \{v\in  H^1(\Omega_l^\ve), \quad v\times \n = 0 \,  \text{ on }  \Gamma_D\cap \partial \Omega_l^\ve,  \\
&&  \hspace{2.7 cm }  v=0  \, \text{ on }  \Gamma_l^\ve\cup (\partial \hat \Omega \times (-L,0)\cap \partial\Omega_l^\ve) \big\}, \\
\hat V(\Lambda_l^\ve)&=& \{v\in  H^1(\Lambda_l^\ve), \quad v = 0 \,  \text{ on }    R_l^\ve \text{ and }  ((\partial \hat \Omega \times (0,\ve))\cup(\hat \Omega\times\{\ve\}))\cap \partial\Lambda_l^\ve  \}, \\
W(\Omega_l^\ve) &= &\{ w\in H^1(\Omega_l^\ve), \; \;  w=0 \, \text{ on } \Gamma_D\cap\partial\Omega_l^\ve \},
\\
 W(\Omega) &=& \{ w\in H^1(\Omega), \quad w=0 \, \text{ on } \Gamma_D \},\;\;\\
V_d(\Omega_l^\ve)&=& \{v\in   V(\Omega_l^\ve), \; \dv v =0 \}, \quad \qquad 
\hat V_d(\Lambda_l^\ve)=\{ v \in  \hat V(\Lambda_l^\ve), \;  \dv v =0 \},
\end{eqnarray*}
where $l=a,v$. For $\phi, \psi \in L^2((0,\sigma)\times \Omega)$ we make use of the notation
$$\langle \phi, \psi \rangle_{\Omega, \sigma} = \int_0^\sigma\int_{\Omega} \phi \psi \, dx dt.$$

In the remainder of the paper we make use of the auxiliary variable $\tilde{p}_l^\ve$ instead of  $p_l^\ve$, where
$$\tilde{p}_l^\ve(x)\,=\,p_l^\ve(x) +\frac { x_n} L p_l^0 \,\mbox{ in }\, \Omega_l^\ve,$$ $l=a,v$. The introduction of 
$\tilde{p}_l^\ve$ allows us to focus on zero Dirichlet boundary conditions for the pressure. Also, for the sake of notational simplicity, in what follows we omit the tilde $\sim$ and write $p_l^\ve$ instead of $\tilde{p}_l^\ve$.  
We remark  that the use of $\v^\ve_l\times \n=0$ on $\Gamma_D\cap \partial \Omega^\ve_l$ and $\dv\, \v^\ve_l =0 \text{ in } \,  \Omega^\ve_l$, along with the fact that $\Gamma_D$ is a flat boundary, lead to  $\partial_{x_n} \v^\ve_l\cdot \n =0$ and, hence,  $\langle \S\v^\ve_l \cdot \n, \phi_l \rangle_{\Gamma_D\cap \partial \Omega^\ve_l}=0$ for  $\v^\ve_l \in V_d(\Omega^\ve_l)$ and $\phi_l \in V(\Omega^\ve_l)$, where $l=a,v$. 

We are interested in the existence of weak solutions to the system of equations \eqref{Stokes_Omega}--\eqref{Diff_InitC}.

\begin{definition}
A weak solution of the problem \eqref{Stokes_Omega}--\eqref{Stokes_ContC} consists of  functions $\v^\ve_l \in  V_d(\Omega_l^\ve)$, $p_l^\ve \in L^2(\Omega_l^\ve)$, 
  $\hat \v^\ve_l \in \hat V_d(\Lambda_l^\ve)$, and $\hat p_l^\ve \in L^2(\Lambda_l^\ve)$, $\l=a,v$, that satisfy the equation
 \begin{equation}\label{micro_weak_flow}
 \begin{aligned} 
  &\sum_{l=a,v} \Big[\langle  2 \mu \ve^2 \S \v^\ve_l, \S \phi_l \rangle_{\Omega_l^\ve} -\langle  p_l^\ve,  \dv \,\phi_l \rangle_{\Omega_l^\ve} -  \frac 1{L} \langle  p_l^0, \phi_{l,n}  \rangle_{\Omega_l^\ve} \Big]\\
&+ \frac 1\ve\sum_{l=a,v} \Big[ \langle  2 \mu \ve^2  \S  \hat \v^\ve_l,  \S \hat \phi_l \rangle_{\Lambda^\ve_l}
-   \langle \hat p_l^\ve , \dv \, \hat\phi_l \rangle_{\Lambda^\ve_l} \Big]=0
\end{aligned}
\end{equation}
for all $\phi_l \in V(\Omega_l^\ve)$ and $\hat \phi_l \in \hat V(\Lambda_l^\ve)$ with $\phi_l= \frac 1 \ve \hat\phi_l$ on $\hat \Lambda\cap \partial \Omega_l^\ve$
  and $\hat \phi_a =\hat \phi_v$ on $\Sigma^\ve$.

A weak solution of the problem  \eqref{Diff_AV}--\eqref{Diff_InitC}  consists of  functions $c^\ve_l -  c_{l,D}\in L^2(0,T; W(\Omega_l^\ve))$,    $\partial_t c^\ve_l \in L^2(\Omega_{l,T}^\ve)$, $c_s^\ve \in L^2(0,T; H^1(\Omega_s^\ve))$,  
$\hat c_l^\ve \in L^2(0,T; H^1(\Lambda^\ve_l))\cap H^1(0,T; L^2(\Lambda_l^\ve))$,  $c^\ve_l \in L^\infty(\Omega^\ve_{l,T})$, and $\hat c^\ve_l \in L^\infty(\Lambda^\ve_{l,T})$,  $l=a,v,s$, which satisfy the equations
\begin{eqnarray}\label{micro_weak_av}
&\sum_{l=a,v} \Big[\langle \partial_t c_l^\ve,  \psi_l \rangle_{\Omega_l^\ve, T} +\langle D_l^\ve \nabla c_l^\ve - \v_l^\ve c_l^\ve, \nabla \psi_l \rangle_{\Omega_l^\ve, T} - \ve \langle \lambda_l (c_s^\ve- c_l^\ve),  \psi_l \rangle_{\Gamma_l^{\ve}, T}\Big] \\
&+  
\dfrac 1 \ve\sum_{l=a,v}  \Big[ \langle \partial_t  \hat c_l^\ve , \hat \psi_l \rangle_{\Lambda_l^\ve, T}   +
 \langle   \hat D_l^\ve\nabla  \hat c_l^\ve -   \hat\v_l^\ve   \hat c_l^\ve, \nabla \hat\psi_l \rangle_{\Lambda_l^\ve, T}   -\ve  \langle\hat \lambda_l(\hat c_s^\ve- \hat c_l^\ve), \hat \psi_l \rangle_{R_l^{\ve}, T}\Big]  =0  \nonumber
\end{eqnarray}
for all $\psi_l \in  L^2(0,T; W(\Omega_l^\ve) ) $ and $\hat \psi_l \in L^2(0,T;  H^1(\Lambda^\ve_l))$ with  $\psi_l= \hat \psi_l$ on $(\hat\Lambda\cap\partial \Omega_l^\ve)\times (0,T)$ and $\hat \psi_a=\hat \psi_v$ on $\Sigma^\ve\times (0,T)$, and
\begin{equation}\label{micro_weak_tissue}
\begin{aligned}
 & \langle\partial_t c_s^\ve,  \psi_s \rangle_{\Omega_s^\ve, T} + \langle D_s^\ve\nabla c_s^\ve,  \nabla \psi_s \rangle_{\Omega_s^\ve, T}  + \langle d_s^\ve c_s^\ve,  \psi_s \rangle_{\Omega_s^\ve, T}    \\
  +  
\frac 1 \ve & \Big[ \langle \partial_t \hat c_s^\ve,  \hat \psi_s \rangle_{\Lambda_s^\ve, T}+ \langle \hat D_s^\ve\nabla \hat c_s^\ve,  \nabla \hat \psi_s \rangle_{\Lambda_s^\ve, T}+
\langle  \hat d_s^\ve  \hat c_s^\ve, \hat \psi_s \rangle_{\Lambda_s^\ve, T}\Big]\\
=& \ve\sum_{l=a,v} \langle \lambda_l(c_l^\ve- c_s^\ve),  \psi_s  \rangle_{\Gamma_l^{\ve}, T} + 
  \sum_{l=a,v} \langle \hat \lambda_l (\hat c_l^\ve- \hat c_s^\ve), \hat \psi_s \rangle_{R_l^{\ve}, T},
\end{aligned}
\end{equation}
for all $\psi_s \in L^2(0,T; H^1(\Omega_s^\ve)) $ and $\hat \psi_s \in L^2(0,T; H^1(\Lambda^\ve_s))$  with $\psi_s=  \hat \psi_s$ on $(\hat\Lambda\cap\partial \Omega_s^\ve)\times(0,T)$, 
and  $c_l^\ve\to c^0_{l}$ in  $L^2(\Omega_l^\ve)$, $\hat c_l^\ve\to \hat c^{\ve,0}_{l}$ in  $L^2(\Lambda_l^\ve)$ as $t \to 0$, for $l=a,v,s$.
\end{definition}

\begin{theorem}
For each $\ve>0$  there exists a unique weak solution of the microscopic model  \eqref{Stokes_Omega}--\eqref{Diff_InitC}.
\end{theorem}
\begin{proof}[Sketch of proof]
A priori estimates similar to those shown below  in Lemma~\ref{Lemma:aprior}, along with well-known results on the well-posedness of the Stokes equations and parabolic systems, ensure the existence and uniqueness of a solution to the system \eqref{Stokes_Omega}--\eqref{Diff_InitC}. We remark that the  Dirichlet  boundary conditions for the pressure on the boundary $\Gamma_D$, see \eqref{Stokes_BC_O},  ensure the uniqueness of the pressure. 
\end{proof}


\section{A priori estimates and convergence results}\label{section4}

We now turn our attention to deriving \textit{a priori} estimates for the weak solutions of the microscopic model \eqref{Stokes_Omega}--\eqref{Diff_InitC}. 
The {\it a priori} estimates are then used in conjunction with the notion of two-scale convergence and an  unfolding operator approach to establish the convergence of the solutions as $\ve\rightarrow 0$.  

\begin{lemma}\label{Lemma:aprior}
Under Assumption \ref{assumption} the  solutions   of the problem \eqref{Stokes_Omega}--\eqref{Stokes_ContC}   satisfy the  a priori estimates
\begin{eqnarray}\label{apriori}
\begin{aligned}
&& \|\v^\ve_l\|_{L^2(\Omega_l^\ve)} + \ve \, \|\nabla \v^\ve_l\|_{ L^2(\Omega_l^\ve)} +
\frac 1{\sqrt{\ve}} \, \|\hat \v^\ve_l\|_{L^2(\Lambda_l^\ve)} + \sqrt{\ve} \,  \|\nabla \hat \v^\ve_l\|_{L^2(\Lambda_l^\ve)} \leq C ,
\end{aligned}
\end{eqnarray}
where $l=a,v$. Moreover, there exist extensions $P^\ve_a$, $P_v^\ve$ and $\hat P^\ve$ of $p_a^\ve$, $p_v^\ve$ and  $\hat p^\ve= \hat p_a^\ve \chi_{\Lambda_a^\ve} + \hat p_v^\ve \chi_{\Lambda_v^\ve}$ respectively, such that 
\begin{eqnarray}\label{apriori_extension}
\|P^\ve_a \|_{L^2(\Omega)} + \|P_v^\ve\|_{L^2(\Omega)} +  \frac 1{\sqrt{\ve}}\|\hat P^\ve \|_{L^2(\Lambda^\ve)} \leq C.
\end{eqnarray}
Finally, the  solutions  of the problem  \eqref{Diff_AV}--\eqref{Diff_InitC}, i.e. the oxygen concentrations in arteries, veins, and tissue, satisfy the estimates
\begin{equation}\label{apriori2}
\begin{aligned}
& \|c^\ve_l\|_{L^\infty(0,T; L^2(\Omega_l^\ve))} + \|\nabla c^\ve_l\|_{L^2((0,T)\times\Omega_l^\ve)} \leq C, \\
& \frac 1{\sqrt{\ve}}  \|\hat c^\ve_l\|_{L^\infty(0,T; L^2(\Lambda_l^\ve))} + \frac 1{\sqrt{\ve}}   \|\nabla \hat c^\ve_l\|_{L^2((0,T)\times\Lambda_l^\ve)} \leq C , \\
 & c^\ve_l(t,x) \geq 0 \text{ a.e. in } \Omega^\ve_{l, T},  \; \; \;  \hat c^\ve_l(t,x) \geq 0 \text{ a.e. in }  \Lambda^\ve_{l,T},  \; \; \;   \\
 & \| c^\ve_l \|_{L^\infty(\Omega_{l,T}^\ve)}  +  \|\hat c^\ve_l\|_{L^\infty(\Lambda_{l,T}^\ve)} \leq C ,\\
&  \|\partial_t c^\ve_l\|_{L^\infty(0,T; L^2(\Omega_l^\ve))} + \| \partial_t\nabla c^\ve_l\|_{L^2((0,T)\times\Omega_l^\ve))} \leq C, \\
&
 \frac 1 {\sqrt{\ve}} \|\partial_t \hat c^\ve_l\|_{L^\infty(0,T; L^2(\Lambda_l^\ve))} +  \frac 1 {\sqrt{\ve}}  \| \partial_t\nabla \hat c^\ve_l\|_{L^2((0,T)\times\Lambda_l^\ve))} \leq C, 
\end{aligned}
\end{equation}
where $l=a,v,s$. Here the  constant $C$ is independent of $\ve$. 
\end{lemma}

\begin{proof}
Using $\v^\ve_l=0$ on $\Gamma_l^\ve$ and $\big(\partial\hat \Omega\times(-L,0)\big)\cap \partial \Omega_l^\ve$,  and  $\hat \v^\ve_l=0$ 
on $R^\ve_l$ and  $\big(\partial\hat \Omega\times(0,\ve) \cup \hat \Omega\times\{\ve\}\big)\cap \partial \Lambda_l^\ve$, and applying Poincar\'e's and Korn's inequalities \cite{Allaire:1989,Arbogast:2006, JaegerHornung:1991,Tartar:1980}, we obtain
\begin{eqnarray}\label{Poincare}
\begin{aligned}
& \|\v^\ve \|^2_{L^2(\Omega_l^\ve)} +\ve^2 \|\nabla\v^\ve \|^2_{L^2(\Omega_l^\ve)}  \leq C \ve^2  \| \S \v^\ve \|^2_{ L^2(\Omega_l^\ve)}, \\
& \|\hat \v^\ve \|^2_{L^2(\Lambda_l^\ve)}  +\ve^2 \|\nabla\hat \v^\ve \|^2_{L^2(\Lambda_l^\ve)} \leq C \ve^2  \| \S \hat \v^\ve \|^2_{ L^2(\Lambda_l^\ve)},
\end{aligned}
\end{eqnarray}
with a constant $C$ independent of $\ve$.  Considering  $\v_l^\ve$ and $\hat \v_l^\ve$, where $l=a,v$, as test functions in the weak formulation  \eqref{micro_weak_flow}, using the divergence-free property  of the blood  velocity fields, and applying \eqref{Poincare} imply the  estimates in \eqref{apriori}.

Due to the continuity conditions on $\Sigma^\ve$ we   can define $\hat p^\ve= \hat p_a^\ve \chi_{\Lambda_a^\ve} + \hat p_v^\ve \chi_{\Lambda_v^\ve}$.  As  in \cite{Allaire:1989} we can construct a restriction operator, which is a linear continuous operator  $\mathcal R^\ve_l: H^1_0(\Omega) \to H^1_0(\Omega^\ve_l)$ such that 
\begin{itemize}
\item[(i)] $u \in H^1_0(\Omega^\ve_l)$ implies $\mathcal R^\ve_l \tilde u = u$  in $\Omega_l^\ve$,  where $\tilde u$ is an extension of $u$ by zero in~$\Omega$.
\item[(ii)]   $\text{div } u =0$ in $\Omega$ implies $\text{div} (\mathcal R^\ve_l u) =0$ in $\Omega_l^\ve$.
\item[(iii)]  For each $u \in H^1_0(\Omega)$ the following estimate holds
$$
\| \mathcal R^\ve_l u \|_{L^2(\Omega^\ve_l)} + \ve \|\nabla \mathcal R^\ve_l u \|_{L^2(\Omega^\ve_l)} \leq 
C \left[\| u \|_{L^2(\Omega)} + \ve \|\nabla u \|_{L^2(\Omega)} \right]
$$
\end{itemize}
with the constant $C$ being independent of $\ve$. 
A similar restriction operator can be defined  for $\Lambda^\ve=\hat\Omega\times(0,\ve)$ as a linear continuous operator $\mathcal{\hat R}^\ve: H^1_0(\Lambda^\ve) \to H^1_0(\Lambda^\ve_{av})$, where $\Lambda_{av}^\ve= \Lambda_a^\ve\cup \Sigma^\ve\cup \Lambda_v^\ve$.
Using the  properties of $\mathcal R^\ve_l$ and $ \mathcal{\hat R}^\ve$, where $l=a,v$,  we can extend $p_l^\ve$ from $\Omega_l^\ve$ into $\Omega$, and $\hat p^\ve$ from  $\Lambda_{av}^\ve$ into $\Lambda^\ve$. These extensions satisfy the {\it a priori} estimates in \eqref{apriori_extension} (see e.g.,  \cite{Allaire:1989}).   In particular, for the construction of the extension of $\hat p^\ve$, we consider a  linear functional $F^\ve$ in $H^{-1}(\Lambda^\ve)$ defined as 
$$
  \langle  F^\ve,  \psi \rangle_{H^{-1}, H^1_0(\Lambda^\ve)} = \langle  \nabla \hat p^\ve,  \mathcal{\hat R}^\ve  \psi \rangle_{H^{-1}, H^1_0(\Lambda^\ve_{av})} \quad \text{ for } \psi \in H^1_0(\Lambda^\ve) ,
$$ 
Using  equation \eqref{Stokes_Lambda},  the properties of the restriction operator $\mathcal{\hat R}^\ve$  and the estimates in \eqref{apriori}  we obtain 
\begin{eqnarray*}
&& \langle  F^\ve,  \psi \rangle_{H^{-1}, H^1_0(\Lambda^\ve)} =  \langle \ve^2 \mu\Delta \hat \v^\ve_{av}, \mathcal{\hat R}^\ve  \psi \rangle_{H^{-1}, H^1_0(\Lambda^\ve_{av})} 
= - \langle \ve^2 \mu\nabla \hat \v^\ve_{av}, \nabla \mathcal{\hat R}^\ve   \psi  \rangle_{\Lambda^\ve_{av}}, \\
&& \left|\langle  F^\ve,  \psi \rangle_{H^{-1}, H^1_0(\Lambda^\ve)}\right| \leq 
C_1  \sqrt{\ve}\left[\| \psi \|_{L^2(\Lambda^\ve)} + \ve \|\nabla \psi \|_{L^2(\Lambda^\ve)} \right]\leq 
C_2\ve \sqrt{\ve} \|\nabla \psi \|_{L^2(\Lambda^\ve)}, 
\end{eqnarray*}
where $\hat \v^\ve_{av} = \hat \v_a^\ve \chi_{\Lambda_a^\ve} + \hat \v_v^\ve \chi_{\Lambda_v^\ve}$.  Thus 
$$
\frac 1{\sqrt{\ve}} \| F^\ve\|_{H^{-1}(\Lambda^\ve)}  \leq C \ve. 
$$
Additionally, we have   
$\langle  F^\ve,  \psi \rangle_{H^{-1}, H^1_0(\Lambda^\ve)} =0$
for all  $\psi \in H^1_0(\Lambda^\ve)$ with $\dv \, \psi  =0$ in $\Lambda^\ve$. Hence, there exists $\hat P^\ve \in L^2(\Lambda^\ve)/ \mathbb R$ such that $ F^\ve = \nabla \hat P^\ve$ and, using the Ne{\u c}as inequality \cite{Paloka:2000},
$$
\frac 1 {\sqrt{\ve}}\| \hat P^\ve \|_{L^2(\Lambda^\ve)/ \mathbb R}\leq \frac 1{\sqrt{\ve}} \frac{C_1}{\ve} \| F^\ve \|_{H^{-1}(\Lambda^\ve)} \leq C_2.  
$$
In the same way as in \cite{Allaire:1989} we obtain that $\hat P^\ve$ is an extension of $\hat p^\ve$. The fact that $\hat p^\ve$ is uniquely  defined implies that $\hat P^\ve \in L^2(\Lambda^\ve)$.

Using that $c^\ve_l- c_{l,D}=0$
on $\Gamma_D\cap \partial\Omega_l^\ve$ and $c_{l,D}=0$ on $\hat \Lambda$, in conjunction with (a) the divergence-free property of $\v_l^\ve$ and $\hat\v_l^\ve$, (b) the zero-boundary conditions for $\v_l^\ve$ and $\hat\v_l^\ve$, and (c)
 the continuity of concentrations on $\hat \Lambda\cap \partial \Lambda_l^\ve$, we obtain 
\begin{equation}\label{eq21}
\begin{aligned}
\langle \v_l^\ve c_l^\ve, \nabla (c_l^\ve- c_{l,D}) \rangle_{\Omega^\ve_l} +
 \frac 1 \ve \langle  \hat \v_l^\ve \hat c_l^\ve, \nabla \hat c_l^\ve \rangle_{\Lambda^\ve_l} =
 \langle \v_l^\ve c_{l,D}, \nabla (c_l^\ve- c_{l,D}) \rangle_{\Omega^\ve_l}\\ +
 \frac 12 \langle \v_l^\ve \cdot \n, |c_l^\ve|^2 \rangle_{\hat\Lambda\cap \partial \Lambda_l^\ve} - \frac 1{2\ve} \langle \hat \v_l^\ve \cdot \n, |\hat c_l^\ve|^2 \rangle_{\hat\Lambda\cap \partial \Lambda_l^\ve} 
\leq\frac 1{2\sigma }\|\v_l^\ve\|^2_{L^2(\Omega^\ve_l)}\|c_{l,D}\|^2_{L^\infty(\Omega_l^\ve)} \\
+ \frac \sigma 2\left(\|\nabla c^\ve_l\|^2_{L^2(\Omega^\ve_l)} + \|\nabla c_{l,D}\|^2_{L^2(\Omega_l^\ve)}\right)   
\end{aligned}
\end{equation}
for some $\sigma>0$. Applying the trace inequality  \cite{JaegerHornung:1991,Ptashnyk:2008} we obtain
\begin{eqnarray}\label{boundary_estim}
\begin{aligned}
\ve \|w\|^2_{L^2(\Gamma^\ve_l)}\leq C\left[ \|w\|^2_{L^2(\Omega^\ve_l)} + \ve^2  \|\nabla w\|^2_{L^2(\Omega^\ve_l)}\right],\\
\ve \|w\|^2_{L^2(R^\ve_l)}\leq C\left[\|w\|^2_{L^2(\Lambda^\ve_l)} + \ve^2  \|\nabla w\|^2_{L^2(\Lambda^\ve_l)}\right] , \label{eq22}
\end{aligned}
\end{eqnarray}
where $l=a,v,s$,  $C$ is independent of $\ve$, $\Gamma_s= \Gamma_a \cup \Gamma_v$, and $R_s= R_a \cup R_v$.  
Now  considering  
$c^\ve_l- c_{l,D}$ and   $\hat c^\ve_l$ as test functions in \eqref{micro_weak_av}--\eqref{micro_weak_tissue} and applying estimates \eqref{eq21} and \eqref{eq22} we obtain the first estimates in  \eqref{apriori2}. 

In order to show the non-negativity of $c^\ve_l$ and $\hat c^\ve_l$, we consider $c^{\ve, -}_l = \min\{c^\ve_l, 0\} $  and $\hat c^{\ve, -}_l = \min\{\hat c^\ve_l, 0\} $  as test functions to derive:
\begin{eqnarray*}
&&\phantom{+}\sum_{l=a,v} \left[ \partial_t \| c^{\ve, -}_l  \|^2_{L^2(\Omega^\ve_l)} + \| \nabla c^{\ve, -}_l  \|^2_{L^2(\Omega^\ve_l)} +  \ve \| c^{\ve, -}_l  \|^2_{L^2(\Gamma^\ve_l)}   - \langle \v^\ve_l c^{\ve, -}_l, \nabla c^{\ve, -}_l \rangle_{\Omega^\ve_{l}}  \right]  \\
&&+ \sum_{l=a,v}\left[\frac 1{\ve} \partial_t \|\hat c^{\ve, -}_l  \|^2_{L^2(\Lambda^\ve_l)} +\frac 1 \ve \| \nabla \hat c^{\ve, -}_l  \|^2_{L^2(\Lambda^\ve_l)} +  \|\hat c^{\ve, -}_l  \|^2_{L^2(R^\ve_l)}
- \frac 1\ve \langle \hat \v^\ve_l \hat c^{\ve, -}_l, \nabla \hat c^{\ve, -}_l \rangle_{\Lambda^\ve_{l}}\right]
   \\
&&- \sum_{l=a,v} \big[  \hat \lambda_l \langle \hat c^{\ve, +}_s, \hat c^{\ve, -}_l \rangle_{R^\ve_{l}}+  \ve \lambda_l \langle c^{\ve, +}_s, c^{\ve, -}_l \rangle_{\Gamma^\ve_{l}}\big]
    \leq  C\sum_{l=a,v} \big[   \ve \langle c^{\ve, -}_s, c^{\ve, -}_l \rangle_{\Gamma^\ve_{l}}  + \langle \hat c^{\ve, -}_s, \hat c^{\ve, -}_l \rangle_{R^\ve_{l}}\big]. 
\end{eqnarray*}
Similarly, for the oxygen concentration in the surrounding tissue, we have 
\begin{eqnarray*}
\partial_t \| c^{\ve, -}_s  \|^2_{L^2(\Omega^\ve_s)} + \| \nabla c^{\ve, -}_s  \|^2_{L^2(\Omega^\ve_s)} +
\frac 1 \ve \partial_t \|\hat c^{\ve, -}_s  \|^2_{L^2(\Lambda^\ve_s)}
 + \frac 1\ve \| \nabla \hat c^{\ve, -}_s  \|^2_{L^2(\Lambda^\ve_s)}  \\ 
\sum_{l=a,v}    \left[\ve \| c^{\ve, -}_s  \|^2_{L^2(\Gamma^\ve_l)} +  \|\hat c^{\ve, -}_s  \|^2_{L^2(R^\ve_l)}-\ve \lambda_l \langle c^{\ve,-}_s, c^{\ve, +}_l \rangle_{\Gamma^\ve_{l}} -
  \hat \lambda_l \langle \hat c^{\ve, -}_s, \hat c^{\ve, +}_l \rangle_{R^\ve_{l}}   \right]  \\    
    \leq  C\sum_{l=a,v}\left[ \ve \langle c^{\ve, -}_s, c^{\ve, -}_l \rangle_{\Gamma^\ve_{l}}  + \langle \hat c^{\ve, -}_s, \hat c^{\ve, -}_l  \rangle_{R^\ve_{l}} \right], 
\end{eqnarray*}
where $c^{\ve,+}_l = \max \{ 0, c^{\ve}_l \}$ and $\hat c^{\ve,+}_l = \max \{ 0, \hat c^{\ve}_l \}$. Using the  boundary conditions for $\v_l^\ve$,  $\hat \v_l^\ve$, $c^\ve_l$ and $\hat c^\ve_l$, we obtain that 
\begin{eqnarray*}
- \langle \v^\ve_l c^{\ve, -}_l, \nabla c^{\ve, -}_l \rangle_{\Omega^\ve_{l}} - \frac 1{\ve} \langle \hat \v^\ve_l \hat c^{\ve, -}_l,  \nabla \hat c^{\ve, -}_l \rangle_{\Lambda^\ve_{l}}  =0 
\end{eqnarray*}
for $l=a,v$.  Combining the last two inequalities and applying estimates \eqref{boundary_estim} and the Gronwall inequality, we obtain that
$c^{\ve,-}_l(t,x)=0$ a.e. in $\Omega^\ve_{l,T}$ and $\hat c^{\ve,-}_l(t,x)=0$ a.e. in $\Lambda^\ve_{l,T}$ for $l= a,v,s$. 

To show the boundedness of $c^\ve_l$ and $\hat c^\ve_l$ we consider $(c^\ve_l - A)^{+} $ and $(\hat c^\ve_l - A)^{+}$ as test functions in \eqref{micro_weak_av}--\eqref{micro_weak_tissue}, where   $A \geq \max\limits_{l=a,v,s} \{ \sup_{\Omega_T} c_{l,D}(t,x), \sup_{\Omega} c_{l}^0(x),  \sup_{\Lambda^\ve} \hat c^{\ve,0}_{l}(x)\}$. Then, due to the prescribed boundary conditions, we have 
\begin{eqnarray*}
- \langle \v^\ve_l c^{\ve}_l, \nabla (c^{\ve}_l - A)^+ \rangle_{\Omega^\ve_{l}} - \frac 1 \ve  \langle \hat \v^\ve_l \hat c^{\ve}_l, \nabla (\hat c^{\ve}_l  - A)^+\rangle_{\Lambda^\ve_{l}}  =0 
\end{eqnarray*}
for $l=a,v$, and thus
\begin{eqnarray*}
\sum_{l=a,v,s}\Big[ \partial_t \| (c^{\ve}_l - A)^+  \|^2_{L^2(\Omega^\ve_l)} + \| \nabla (c^{\ve}_l - A)^+  \|^2_{L^2(\Omega^\ve_l)} + \ve \| (c^{\ve}_l - A)^+  \|^2_{L^2(\Gamma^\ve_l)} \\
\frac 1 {\ve} \partial_t \| (\hat c^{\ve}_l - A)^+  \|^2_{L^2(\Lambda^\ve_l)} +\frac 1 \ve \| \nabla (\hat c^{\ve}_l - A)^+  \|^2_{L^2(\Lambda^\ve_l)}    +  \| (\hat c^{\ve}_l - A)^+  \|^2_{L^2(R^\ve_l)}\Big]\\
\leq C \sum_{l=a,v} \Big[ \ve \langle (c_s^\ve - A)^+, (c_l^\ve - A)^+ \rangle_{\Gamma^\ve_l} + 
 \langle (\hat c_s^\ve - A)^+, (\hat c_l^\ve - A)^+ \rangle_{R^\ve_l} \Big].
\end{eqnarray*}
 Thus, applying estimates  \eqref{boundary_estim} together with the Gronwall inequality,  we conclude that 
$(c^\ve_l(t,x)-A)^+ =0 $  a.e. in $\Omega^\ve_{l, T}$ and  $(\hat c^\ve_l(t,x)-A)^+ =0 $ a.e. in $\Lambda^\ve_{l,T}$ with $l=a,v,s$. Therefore,  the second part of the estimates in \eqref{apriori2} follows.  

Finally, differentiating equations \eqref{Diff_AV} and \eqref{Diff_Tis} with respect to time, and using (a) $\partial_t (c^\ve_l- c_{l,D})$ and $\partial_t \hat c^\ve_l $, respectively,  as test functions, and (b) the regularity assumptions on the initial values $c_l^0$ and $\hat c_l^{\ve, 0}$, yield the estimates for the time derivatives in  \eqref{apriori2}.
\end{proof}

To derive the macroscopic equations we employ the notion of two-scale convergence  \cite{Allaire:1992,Nguetseng:1989} and the unfolding method \cite{Cioranescu:2012, Cioranescu:2008_1}. We denote by  $\mathcal T^\ast_\ve: L^p(\Omega^\ve_l) \to   L^p(\Omega\times Y_{l})$  the  unfolding operator and by  $\mathcal T^b_\ve: L^p(\Gamma^\ve_l) \to  L^p(\Omega \times \Gamma_{l})$  the boundary unfolding operator, for $p\in [1 , \infty)$ (see, e.g., \cite{Cioranescu:2012, Cioranescu:2008_1}).  As in \cite{Cioranescu:2008, Jaeger:2006}  we also define unfolding operators in the thin layer  $\Lambda^\ve_l$ and on $R^\ve_l$, where $l=a,v,s$, as follows. 
  
\begin{definition}
For a measurable function $\phi$ on $\Lambda^\ve$ we define the unfolding operator
$\mathcal T_\ve^{bl}$ as 
$$
\mathcal T_\ve^{bl}(\phi)(x,y) = \phi(\ve [(\hat x, 0)/\ve]+ \ve y)\quad  \text{ for } \; \hat x \in \hat \Lambda, \; y \in Z\; .
$$
For a measurable function $\phi$ on $\Lambda^\ve_l$ we define the unfolding operator
$\mathcal T_\ve^{\ast, bl}$ as 
$$
\mathcal T_\ve^{\ast, bl}(\phi)(x,y) = \phi(\ve [(\hat x, 0)/\ve]+ \ve y)\quad  \text{ for } \; \hat x \in \hat \Lambda, \; y \in Z_l\; .
$$
For a measurable function $\phi$  on $R^\ve_l$ we define the boundary unfolding operator
$\mathcal T_\ve^{b, bl}$ as 
$$
\mathcal T_\ve^{b, bl}(\phi)(x,y) = \phi(\ve[(\hat x, 0)/\ve]+\ve y)\quad  \text{ for } \; \hat x \in \hat \Lambda, \; y \in R_l\; .
$$
\end{definition}
 The definition of the unfolding operator implies directly (see e.g., \cite{Cioranescu:2008, Jaeger:2006}) that  
$$
\|\mathcal T_\ve^{\ast, bl}  \phi \|^p_{L^p(\hat \Lambda\times Z_l)} \leq  \ve^{-1}|\hat Z|\| \phi \|^p_{L^p( \Lambda^\ve_l)} \; \;   \text{ and } \; \;  \ve \mathcal T_\ve^{\ast, bl} (\nabla \phi ) = \nabla_y \mathcal T_\ve^{\ast, bl} ( \phi ) \quad  \text{ in } 
\hat \Lambda\times Z_l.
$$
Theorems \ref{4-3} and \ref{4-4} below are proven in the same manner as the corresponding results in  \cite{Cioranescu:2012, Cioranescu:2008_1}.  For the convenience of the reader, we provide short sketches of the proofs.
\begin{theorem}\label{4-3}
Let $\{w^\ve\} \subset W^{1,p}(\Lambda^\ve)$, where $p\in(1,\infty)$ and 
$
 \frac 1\ve \| w^\ve\|^p_{W^{1,p}(\Lambda^\ve)} \leq C.
$
Then, there exist a subsequence (denoted again by $w^\ve$) and functions   $w \in W^{1,p}(\hat \Lambda)$ and  $w_1 \in L^p(\hat \Lambda; W^{1,p}(Z))$  such that  $w_1$ is $\hat Z$--periodic and
\begin{eqnarray*}
\mathcal T^{bl}_\ve(w^\ve) &\rightharpoonup &w  \hspace{2.5 cm }  \text{ weakly in } L^p(\hat \Lambda; W^{1,p}(Z)), \\
\mathcal T^{bl}_\ve(\nabla w^\ve)& \rightharpoonup & \nabla_{\hat x} w + \nabla_y w_1 \qquad \text{ weakly in } L^p(\hat \Lambda\times Z)\; .
\end{eqnarray*}
\end{theorem}

\begin{proof}[Sketch of proof]
By rescaling  $\tilde w^\ve(\hat x,y)= w^\ve(\hat x, \ve y)$ and using the assumptions on $\{w^\ve\}$ we obtain that  there exists a function $w\in W^{1,p}(\hat \Lambda)$ with 
$\tilde w^\ve \to w$ in $L^p(\Lambda^1)$ and $\nabla_{\hat x} \tilde w^\ve   \rightharpoonup   \nabla_{\hat x} w$ in $L^p(\Lambda^1)$.
Also, the assumptions on $\{w^\ve\}$ ensure that  $\mathcal T^{bl}_\ve(w^\ve)$,  $\mathcal T^{bl}_\ve(\nabla w^\ve)$, and 
$\nabla_y\mathcal T^{bl}_\ve(w^\ve)$  are bounded in $L^p(\hat \Lambda\times Z)$.  Hence,  $\mathcal T^{bl}_\ve(w^\ve)  \rightharpoonup  w$ in $L^p(\hat \Lambda; W^{1,p}(Z))$.
We now define 
$$
V^\ve = \frac 1 \ve (\mathcal T^{bl}_\ve(w^\ve) - \mathcal M^{bl}_\ve(w^\ve)) ,
\quad \text{ where  } \quad 
\mathcal M^{bl}_\ve(w^\ve) = \frac 1{|Z|} \int_Z \mathcal T^{bl}_\ve(w^\ve)(\,\cdot\,, y) dy.
$$
Using the assumptions on $w^\ve$ and applying Poincar\'e's inequality, we have that 
\begin{eqnarray*}
&& \|\nabla_y V^\ve\|_{L^p(\hat \Lambda\times Z)} = \|\mathcal T^{bl}_\ve(\nabla w^\ve)\|_{L^p(\hat \Lambda\times Z)} \leq C_1, \\
&& \| V^\ve - \hat y^c \cdot \nabla_{\hat x} w\|_{L^p(\hat \Lambda\times Z)}\leq C_2 \|\nabla_y V^\ve - \nabla_{\hat x} w\|_{L^p(\hat \Lambda\times Z)}\leq C_3, 
\end{eqnarray*}
where $\hat y^c=(y_1 - a_1/2, \ldots, y_{n-1} - a_{n-1}/2)$.
Then, there exists a function $w_1 \in L^p(\hat\Lambda; W^{1,p}(Z))$ such that,  up to a subsequence,
\begin{eqnarray*}
V^\ve - \hat y^c \cdot \nabla_{\hat x} w  \rightharpoonup w_1 \quad \text{ in } \quad   L^p(\hat\Lambda; W^{1,p}(Z)).
\end{eqnarray*}
Hence, we have the second convergence result  stated in the theorem. 

The proof of $\hat Z$-periodicity of  $w_1$ follows the same lines as in the case of $\mathcal T_\ve$, see e.g.  \cite{Cioranescu:2008_1}. Specifically,  one considers the differences $V^\ve(\hat x, y^1_j) - V^\ve(\hat x, y^0_j)$  and $ \hat y^{c,1}_j \cdot \nabla_{\hat x} w - \hat y^{c,0}_j \cdot \nabla_{\hat x} w$, and shows that 
$w_1(\hat x,  y^1_j) = w_1(\hat x,  y^0_j)$  in the weak sense for $j=1, \ldots, n-1$, 
where 
$y^1_j= (y_1, \ldots, y_{j-1}, a_j, y_{j+1}, \ldots, y_n)$,  $y^0_j= (y_1, \ldots, y_{j-1}, 0, y_{j+1}, \ldots, y_n)$, and $\hat Z=(0,a_1)\times \ldots\times (0,a_{n-1})$.
\end{proof}

\begin{theorem}\label{4-4}
Let $\{w^\ve\} \subset W^{1,p}( \Lambda^\ve_l) $ , where $p\in (1,\infty)$ and $l=a,v,s$,  with 
$$
\ve^{-1} \| w^\ve\|^p_{L^p(\Lambda^\ve_l)} \leq C, \qquad  \ve^{p-1} \| \nabla w^\ve\|^p_{L^p( \Lambda_l^\ve)} \leq C.
$$
Then, there exist a subsequence (denoted again by $w^\ve$) and  a $\hat Z$-periodic function $\hat w \in L^p (\hat \Lambda;  W^{1,p}(Z_l))$,  such that
\begin{eqnarray*}
\mathcal T^{\ast, bl}_\ve(w^\ve)& \rightharpoonup& \hat w \quad \; \; \; \; \text{ weakly in }  L^p (\hat \Lambda;  W^{1,p} (Z_l)), \\
\ve \mathcal T^{\ast, bl}_\ve(\nabla w^\ve)& \rightharpoonup &\nabla_y  \hat w \quad \text{ weakly in }  L^p (\hat \Lambda \times Z_l)\; .
\end{eqnarray*}
\end{theorem}
\begin{proof}
Due to the assumptions on $\{w^\ve\}$, we obtain that $\mathcal T^{\ast, bl}_\ve(w^\ve)$ is bounded in $L^p(\hat \Lambda; W^{1,p}(Z_l))$.  Thus, there exists a function $\hat w$ such that the stated  convergences are satisfied. 
The $\hat Z$-periodicity follows by the fact that  for $\psi \in C_0(\hat\Lambda\times Z)$,
\begin{eqnarray*}
&& \int_{\hat \Lambda\times Z_l}\left[ \mathcal T^{\ast, bl}_\ve(w^\ve)(\hat x, y + (\hat e_j, 0)) -  \mathcal T^{\ast, bl}_\ve(w^\ve)(\hat x, y)\right] \psi(\hat x,y) d\hat x dy \\
&&  = \int_{\hat \Lambda\times Z_l} \mathcal T^{\ast, bl}_\ve(w^\ve)(\hat x, y) ( \psi(\hat x - \ve \hat e_j,y) - \psi(\hat x, y) ) d\hat x dy \to 0  \quad \text{ as } \ve \to 0,
\end{eqnarray*}
where $\hat e_j$ are standard basis vectors for $j=1, \ldots, n-1$.
\end{proof}

To prove convergence results for the unfolding operator in the perforated  thin layer  $\Lambda^\ve_l$, with $l=a,v,s$, we define an interpolation operator $\mathcal Q^{\ast, bl}_\ve$.
First, we introduce the  notation:
  \begin{equation*}
  \begin{aligned}
& \mathcal Y= \text{Int}\bigcup_{k\in \{0,1\}^{d-1}} (\overline Z + (k,0)), \;  \hat \Lambda^\ve_{\mathcal Y}=  \text{Int}\bigcup_{\xi \in  \Xi^\ve_{\mathcal Y} } \ve(\overline{\hat Z} + \xi), \;   \Lambda^\ve_{\mathcal Y, l}= \text{Int}\bigcup_{\xi \in  \Xi^\ve_{\mathcal Y} } \ve(\overline Z_l + (\xi,0)),\\
& \Xi^\ve_{\mathcal Y} = \{\xi \in \mathbb Z^{n-1} :  \ve(\mathcal Y+ (\xi,0)) \subset \Lambda^\ve \}, \; \; 
\hat \Xi^\ve= \{ \xi \in \mathbb Z^{n-1}: \ve(Z+ (\xi,0)) \subset \Lambda^\ve \}.
\end{aligned}
\end{equation*}
 Then, the definition of $\mathcal Q^{\ast, bl}_\ve$ is similar to the one for perforated domains in \cite{Cioranescu:2012}.
 \begin{definition} The  operator 
$
\mathcal Q^{\ast, bl}_\ve: L^p(\Lambda^\ve_{l,T}) \to L^p(0,T;W^{1,\infty}( \hat \Lambda^\ve_{\mathcal Y}\times (0,\ve)))
$
for $p \in [1, + \infty]$ is defined by
$$
\mathcal Q^{\ast,bl}_\ve(\phi)(t,\ve \xi) = \frac{1}{|Z_l|} \int_{Z_l} \phi(t,\ve (\xi, 0) + \ve y) dy \quad \text{ for } \xi \in \hat \Xi^\ve, \; \text{a.a.} \;  t \in (0,T). 
$$
For $x \in  \hat \Lambda^\ve_{\mathcal Y} \times (0,\ve) $,   $\mathcal Q^{\ast,bl}_\ve(\phi)(t,x)$ is defined as the
$Q_1$- interpolant of $\mathcal Q^{\ast,bl}_\ve(\phi)(t,\ve \xi)$ at the vertices of the cell $\ve ( [\hat x/\ve] + \hat Z)$ with respect to $x_1, \ldots, x_{n-1}$ and constant in $x_n$,  for a.a. $t \in (0,T)$.  
\end{definition}

We remark that  $\partial_t\mathcal Q^{\ast, bl}_\ve( \phi) = \mathcal Q^{\ast,bl}_\ve(\partial_t \phi)$ 
and $\partial_t\mathcal R^{\ast,bl}_\ve(\phi)= \partial_t(\phi - \mathcal Q^{\ast,bl}_\ve( \phi) )  = \mathcal R^{\ast,bl}_\ve(\partial_t \phi)$.
Lemma \ref{4-5} and Theorem \ref{th_22} below are proven in a similar manner as the corresponding results in  \cite{Cioranescu:2012}.
\begin{lemma}\label{4-5}
For all $\phi \in W^{1,p}(\Lambda^\ve_{l,T})$, where $p \in (1, + \infty)$, the following estimates hold
\begin{eqnarray*}
\|\mathcal Q^{\ast,bl}_\ve(\phi) \|_{L^p((0,T)\times\hat \Lambda^\ve_{\mathcal Y}\times (0,\ve))} &\leq & C \| \phi \|_{L^p(\Lambda^\ve_{l,T})}, \\
\|\nabla_{\hat x} \mathcal Q^{\ast,bl}_\ve(\phi) \|_{L^p((0,T)\times\hat \Lambda^\ve_{\mathcal Y}\times (0,\ve))}& \leq & C \|\nabla \phi \|_{L^p(\Lambda^\ve_{l,T})},   \\
\|\mathcal R^{\ast,bl}_\ve(\phi)\|_{L^p((0,T)\times \Lambda^\ve_{\mathcal Y, l})}& \leq & C\ve  \| \nabla\phi \|_{L^p(\Lambda^\ve_{l,T})}, \\
\|\nabla\mathcal R^{\ast,bl}_\ve(\phi)\|_{L^p((0,T)\times \Lambda^\ve_{\mathcal Y, l})} & \leq & C  \|\nabla \phi \|_{L^p(\Lambda^\ve_{l,T})}, \\
 \|\partial_t \mathcal Q^{\ast,bl}_\ve(\phi) \|_{L^p((0,T)\times \hat \Lambda^\ve_{\mathcal Y}\times (0,\ve))} &\leq & C \|\partial_t \phi \|_{L^p(\Lambda^\ve_{l,T})}, \\ 
\|\partial_t \mathcal R^{\ast,bl}_\ve(\phi)\|_{L^p((0,T)\times \Lambda^\ve_{\mathcal Y, l})} &\leq & C \ve \|\partial_t \phi \|_{L^p(\Lambda^\ve_{l,T})},
\end{eqnarray*}
where the constant $C$ is independent of $\ve$. 
\end{lemma}
\begin{theorem}\label{th_22}
Assume that the sequence $\{w^\ve\} \subset L^p(0,T; W^{1,p}(\Lambda^\ve_l))\cap W^{1,p}(0,T; L^p(\Lambda^\ve_l))$, with $p \in (1, + \infty)$, satisfies  $\ve^{-1}\|w^\ve\|^p_{L^p(0,T; W^{1,p}(\Lambda^\ve_l))} +\ve^{-1}\|\partial_t w^\ve\|^p_{L^p((0,T)\times \Lambda^\ve_l)} \leq C$. Then, there exists a function $w \in L^p(0,T; W^{1,p}(\hat \Lambda))$ such that 
\begin{eqnarray}\label{conver_th_22}
\begin{aligned}
&\mathcal T^{bl}_\ve (\mathcal Q^{\ast,bl}_\ve(w^\ve)^\sim) \rightharpoonup  w &&   \text{ weakly in } L^p(\hat \Lambda_T; W^{1,p}(Z)), \\
&\mathcal T^{bl}_\ve (  \mathcal Q^{\ast,bl}_\ve(w^\ve)^\sim) \to  w   && \text{ strongly in } L^p(0,T; L^p_{\text{loc}}(\hat \Lambda; W^{1,p}(Z))), \\
&\mathcal T^{bl}_\ve ( \nabla_{\hat x} \mathcal Q^{\ast,bl}_\ve(w^\ve)^\sim) \rightharpoonup \nabla_{\hat x} w &&   \text{ weakly in } L^p(\hat \Lambda_T\times Z),
\end{aligned}
\end{eqnarray}
where ${\mathcal Q^{\ast,bl}_\ve}(w^\ve)^\sim$ is the extension by zero  of ${\mathcal Q^{\ast,bl}_\ve}(w^\ve)$  from $(0,T)\times \hat\Lambda^\ve_{\mathcal Y} \times (0,\ve)$ into $\Lambda^\ve_T$.
\end{theorem}
\begin{proof}[Sketch of proof]
The assumptions on $w^\ve$, the estimates in Lemma~\ref{4-5}, and the definition of $\mathcal Q^{\ast,bl}_\ve$ ensure the boundedness of $\mathcal Q^{\ast,bl}_\ve(w^\ve)^\sim$, its time derivative,  and  $\nabla_{\hat x} \mathcal Q^{\ast,bl}_\ve(w^\ve)^\sim$ in $L^p(\hat \Lambda_T)$.  Hence, there exists a function $w \in L^p(0,T; W^{1,p}(\hat \Lambda))$ such that 
$\mathcal Q^{\ast,bl}_\ve(w^\ve)^\sim \to w$ weakly in $L^p(\hat \Lambda_T)$ and strongly in $L^p(0,T; L^p_{\text{loc}}(\hat \Lambda))$, 
and $\nabla_{\hat x} \mathcal Q^{\ast,bl}_\ve(w^\ve)^\sim \rightharpoonup \nabla_{\hat x} w$  weakly in $L^p(\hat \Lambda_T)$.  Then, by the properties of $\mathcal T_\ve^{bl}$ (see e.g.,  \cite{Cioranescu:2008, Jaeger:2006}), and using the fact that  $\mathcal Q^{\ast,bl}_\ve(w^\ve)$ is constant in $x_n$,  we obtain the first two convergence results in~\eqref{conver_th_22}.

 Lemma~\ref{4-5} and the definition of   $\mathcal Q^{\ast,bl}_\ve$
ensure the boundedness of $\mathcal Q^{\ast,bl}_\ve(w^\ve)|_{\hat K\times(0,\ve)}$ in $L^p(0,T; W^{1,p}(\hat K\times (0,\ve)))$, where $\hat K \subset \hat \Lambda$ is a relatively compact open set and  $\mathcal Q^{\ast,bl}_\ve(w^\ve)|_{\hat K\times (0,\ve)}$ is constant with respect to $x_n$.
Then, using Theorem~\ref{4-3}, we obtain  the existence of a function $w_{1, \hat K} \in L^p(\hat K_T; W^{1,p}(Z))$, which is constant in $y_n$ and $\hat Z$-periodic,  such that 
$$
\mathcal T_\ve^{bl} ( \nabla_{\hat x} \mathcal Q^{\ast,bl}_\ve(w^\ve)|_{\hat K}) \rightharpoonup \nabla_{\hat x} w + \nabla_{\hat y} w_{1, \hat K} \quad \text{ weakly in } L^p(\hat K_T\times   Z). 
$$
Due to the fact that $w_{1,K}$ is a polynomial of degree less or equal to one in each $y_j$, $j=1, \ldots, n-1$, and it is  constant with respect to $y_n$ and $\hat Z$-periodic, it follows that $w_{1,K}$ is constant in $y$. Then, since  $ \nabla_{\hat x} \mathcal Q^{\ast,bl}_\ve(w^\ve)^\sim$ is bounded in $L^p(\Lambda^\ve\times (0,T))$, and hence  $\mathcal T^{bl}_\ve ( \nabla_{\hat x} \mathcal Q^{\ast,bl}_\ve(w^\ve)^\sim)$ is bounded in $L^p(\hat \Lambda_T\times Z)$, we obtain the  last convergence in \eqref{conver_th_22}. 
\end{proof}

The estimates for $\mathcal R^{\ast,bl}_\ve(w^\ve)$ along with the convergence of $\mathcal T_\ve^{\ast, bl}(\ve^{-1}\mathcal R^{\ast,bl}_\ve(w^\ve))$, given by Theorem~\ref{4-4},  (and by using  Theorem~\ref{th_22})  imply the following result.
\begin{theorem}\label{4-8}
Let $\{w^\ve\} \subset L^p(0,T; W^{1,p}(\Lambda^\ve_l))\cap W^{1,p}(0,T;  L^p(\Lambda^\ve_l))$, $p \in (1, + \infty)$, with 
$\frac 1 \ve \| w^\ve\|^p_{L^p(0,T;W^{1,p}(\Lambda^\ve_l))} + \frac 1 \ve \| \partial_t w^\ve\|^p_{L^p((0,T)\times\Lambda^\ve_l)} \leq C. $
Then there exist a subsequence (denoted again by $\{w^\ve\}$) and  functions  $w \in L^p(0,T;W^{1,p}(\hat \Lambda))$ and  $w_1 \in L^p(\hat \Lambda_T; W^{1,p}(Z_l))$  such that  $w_1$ is $\hat Z-$periodic and
\begin{eqnarray*}
&\mathcal  T^{\ast, bl}_\ve(w^\ve) \rightharpoonup w& \text{ weakly in } L^p(\hat \Lambda_T; W^{1,p}(Z_l)), \\
&\mathcal  T^{\ast, bl}_\ve(w^\ve) \to w & \text{ strongly in } L^p(0,T; L^p_{\text{loc}}(\hat \Lambda; W^{1,p}(Z_l))), \\
&\mathcal  T^{\ast, bl}_\ve(\nabla w^\ve) \rightharpoonup \nabla_{\hat x} w + \nabla_y w_1 &\text{ weakly in } L^p(\hat \Lambda_T\times Z_l)\; .
\end{eqnarray*}
\end{theorem}

Finally,  using the notion of two-scale convergence and the properties of the unfolding operator, we can prove the following lemma.

\begin{lemma}\label{lem:convergence}
The following hold.
\begin{itemize}
\item[ 1.] There exist subsequences of $\{ \v_l^\ve\}$, $\{p^\ve_l\}$, $\{ \hat \v_l^\ve\}$, and $\{\hat p^\ve_l\}$  (denoted again by $\{ \v_l^\ve\}$, $\{p^\ve_l\}$, $\{ \hat \v_l^\ve\}$, and $\{\hat p^\ve_l\}$) and functions 
$\v_l \in L^2(\Omega; H^1_\text{per}(Y_l))$, $p_l \in L^2(\Omega\times Y_l)$,
$\hat \v_l \in L^2(\hat \Lambda; H^1(Z_l))$,  and $\hat p \in L^2(\hat\Lambda \times Z)$ such that $\hat \v_l$ is $\hat Z-$periodic$, \hat p_l=\hat p|_{\hat \Lambda \times Z_l}$,
  and as $\ve \to 0$
\begin{eqnarray*}
\begin{aligned}
&\v_l^\ve \to \v_l,\; \;  \ve \nabla \v^\ve_l \to \nabla_y \v_l, \quad && p_l^\ve=P_l^\ve\chi_{\Omega^\ve_l} \to p_l   && \text{ two-scale},\\
&\hat \v_l^\ve \to \hat \v_l ,\; \; \ve \nabla \hat\v^\ve_l \to \nabla_y \hat\v_l, \quad  &&\hat P^\ve \to \hat p, \quad    \hat p^\ve_l=\hat P^\ve\chi_{\Lambda^\ve_l} \to \hat p_l && \text{ two-scale}.
\end{aligned}
\end{eqnarray*} 
\item[2.]
There exist subsequences of $\{c^\ve_l\}$ and $\{\hat c^\ve_j \}$ (denoted again by $\{c^\ve_l\}$, $\{\hat c^\ve_j\} $) and $c_l \in L^2(0,T; H^1(\Omega))$, $\partial_t c_l \in L^2(\Omega_T)$, 
$c_l^1 \in L^2( \Omega_T; H^1_\text{per} (Y_l))$,  
$\hat c_j \in L^2(0, T; H^1(\hat\Lambda))$, $\hat c^1_j \in L^2(\hat\Lambda_T; H^1( Z_j))$,   and  $ \partial_t \hat c_j\in L^2(\hat\Lambda_T)$ such that $\hat c^1_j$ is $\hat Z-$periodic 
 and as $\ve \to 0$ 
\begin{equation}\label{conc-conver_2}
\begin{aligned}
&\mathcal T^{\ast}_\ve(c_l^\ve) \rightharpoonup c_l \;  && \text{ weakly in } L^2( \Omega_T; H^1(Y_l)) \; ,
\\ &\mathcal T^{\ast}_\ve(c_l^\ve)\to c_l  \; \;   &&\text{ strongly in } L^2(0,T;L^2_{\text{loc}}( \Omega; H^1(Y_l))),\\
  &\partial_t \mathcal T^\ast_\ve(c_l^\ve) \rightharpoonup \partial_t c_l  && \text{ weakly in } L^2(\Omega_T\times Y_l), \;  \\ &
  \mathcal T^{\ast}_\ve(\nabla c_l^\ve) \rightharpoonup \nabla c_l + \nabla_y c_l^1 &&  \text{ weakly in } L^2(\Omega_T\times Y_l), 
   \end{aligned}
\end{equation}
  \begin{equation}\label{conc-conver_3}
\begin{aligned}
 &\mathcal T^{\ast,bl}_\ve(\hat c^\ve_j) \rightharpoonup \hat c_j  \;  &&  \text{ weakly in } L^2(\hat\Lambda_T; H^1(Z_{j})), 
   \\ & \mathcal T^{\ast,bl}_\ve(\hat c^\ve_j) \to \hat c_j  && \text{ strongly in } L^2(0,T; L^2_{\text{loc}}(\hat \Lambda; H^1(Z_{j}))),\\
    &\partial_t \mathcal T^{\ast,bl}_\ve(\hat c^\ve_j) \rightharpoonup  \partial_t \hat c_j \;   && \text{ weakly in } L^2(\hat \Lambda_T\times Z_{j}),   \\ &
  \mathcal T^{\ast,bl}_\ve(\nabla \hat c^\ve_j) \rightharpoonup \nabla \hat c_j + \nabla_y \hat c^1_j \; && \text{ weakly in } L^2(\hat \Lambda_T\times Z_{j}), 
   \end{aligned}
\end{equation}
and 
  \begin{equation}\label{conc-conver_4}
\begin{aligned}
  &  \mathcal T^b_\ve( c_l^\ve) \rightharpoonup c_l  \;  &&  \text{ weakly in }  L^2(\Omega_T\times\Gamma_l), 
   \\ &  \mathcal T^{b, bl}_\ve(\hat c^\ve_j)  \rightharpoonup \hat c_j  \;  && \text{ weakly in }  L^2(\hat\Lambda_T\times R_{av} ), 
  \end{aligned}
\end{equation}
\end{itemize}
where $l=a,v, s$ and $j=av, s$. Here,  $\hat c^\ve_{av}=\hat c^\ve_a\chi_{\Lambda_a^\ve} +\hat c^\ve_v\chi_{\Lambda_v^\ve}$, $\Gamma_s= \Gamma_a\cup \Gamma_v$, $R_{av}= R_a\cup R_v$, and $Z_{av}= {\rm Int}(  \overline Z_a \cup  \overline Z_v)$.
\end{lemma}

\begin{proof}[Sketch of proof]
 Due to the continuity of concentrations on $\Sigma^\ve$, we can define 
$\hat c^\ve_{av}=\hat c^\ve_a\chi_{\Lambda_a^\ve} +\hat c^\ve_v\chi_{\Lambda_v^\ve}$.
 The {\it a priori} estimates in \eqref{apriori}, \eqref{apriori_extension} 
and  \eqref{apriori2} along with  (a) the compactness theorem for two-scale convergence, (b) related convergence results for  unfolded sequences \cite{Allaire:1992,Cioranescu:2012, Paloka:2000, Jaeger:2006, Nguetseng:1989}, and (c) Theorem~\ref{4-8}    imply the  convergence results in the statement of the lemma.

The last two convergence results in \eqref{conc-conver_4} follow from the weak convergence of $\mathcal T^{\ast}_\ve(c_l^\ve)$ and 
$\mathcal T^{\ast,bl}_\ve(\hat c^\ve_j)$ in  $L^2(\Omega_T; H^1(Y_l))$ and $L^2(\hat \Lambda_T; H^1(Z_{j}))$, respectively, along with the trace theorem  applied in $H^1(Y_l)$ and $H^1(Z_{j})$,  where  $l=a,v,s$ and $j=av, s$. 
 \end{proof}

\section {Derivation of macroscopic equations for velocity fields}\label{macro_velocity_1}

We now derive the homogenized, macroscopic equations for the arterial and venous blood velocity fields in the two tissue layers (skin tissue layer and fat tissue layer) of the adopted tissue geometry. We start with Theorem \ref{5-1}, which is the first of the main results of the paper.  
\begin{proof}[\bf Proof of Theorem~\ref{5-1}]
We first use  the following test functions in \eqref{micro_weak_flow}: 
\begin{enumerate}
\item[(a)] $\phi_l(x)=\ve\psi_l\left(x, \frac x \ve\right)$ with
 $\psi_l \in C^\infty_0(\Omega, C^\infty_{\text{per}}(Y))$ and  $\psi_l(x,y) =0$ on   $\Omega \times \Gamma_l$,  and 
 \item[(b)] $\hat \phi_l(x)=\ve\hat \psi\left(\hat x, \frac x \ve\right)$ with  $\hat \psi \in C^\infty_0(\hat \Lambda, C^\infty_{per}(\hat Z; C^\infty_0(0,1)))$  and $\hat \psi(\hat x, y) =0$ on  $\hat\Lambda\times (R_a\cup R_v)$.  
 \end{enumerate}
  Using the derived {\it a priori} estimates  and applying the two-scale convergence  of $p_a^\ve$,  $\hat p_a^\ve$, 
$p_v^\ve$, and $\hat p_v^\ve$, established in section \ref{section4},  we obtain that
\begin{equation}\label{macro_pres_1}
|Y|^{-1} \langle  p_a,  \dv_y \psi_a  \rangle_{\Omega\times Y_a}+ 
|Y|^{-1} \langle  p_v,  \dv_y \psi_v  \rangle_{\Omega\times Y_v}+ 
|\hat Z|^{-1}\langle \hat p,  \dv_y\hat \psi \rangle_{\hat \Lambda \times Z_{av}}= 0. 
\end{equation}
The last equation implies that  
\begin{enumerate}
\item[(a)] $p_l \in L^2(\Omega; H^1(Y_l))$ with
$\nabla_y p_l=0$ a.e. in $\Omega\times Y_l$, and
\item[(b)] $\hat  p  \in L^2(\hat \Lambda; H^1(Z_{av})) $ with
 $\nabla_y \hat p =0$ a.e. in $\hat\Lambda \times Z_{av}$,
\end{enumerate}
where $l=a,v$. Thus,  
$p_a=p_a(x)$, $p_v=p_v(x)$ in $\Omega$ and  $\hat p= \hat p(\hat x)$ in $\hat\Lambda$. 

The  two-scale convergence of $\v_l^\ve$ and $\hat \v_l^\ve$ at the oscillating boundaries $\Gamma_l^\ve$, $R_l^\ve$,  and $\Lambda^\ve_l \cap \{ x_n = \ve \}$ is ensured by the {\it a priori} estimates \eqref{apriori} and the boundary estimate  \eqref{boundary_estim}. This implies that  
\begin{eqnarray} \label{BC_Dirichlet_Macro}
&\v_l (x,y)= 0 \quad \text{ on } \Omega\times \Gamma_l, \quad   & \quad 
\hat \v_l (x,y) = 0 \quad \text{ on } \hat\Lambda\times (R_l  \cup\hat Z_{av}^1), \quad l=a,v,
\end{eqnarray}
where $\hat Z^1_{av} = \partial Z_{av} \cap \{y_n =1\}$. Using  $\dv \, \v_l^\ve = 0$ in $\Omega_l^\ve$  and considering   $\psi_l \in C_0^\infty(\Omega; C^\infty_{\text{per}}(Y))$, we  obtain 
\begin{eqnarray*}
0= \langle \dv \, \v_l^\ve (x) , \psi_l(x, x/\ve) \rangle_{\Omega^\ve_l} = - \langle  \v_l^\ve (x), \nabla \psi_l(x, x/\ve) + 1/\ve \nabla_y \psi_l(x, x/\ve)\rangle_{\Omega^\ve_l}. 
\end{eqnarray*}
The two-scale convergence of $\v_l^\ve$ implies that 
\begin{equation}\label{dvi_y_v}
0= \lim\limits_{\ve\to 0} \langle \v_l^\ve (x),  \nabla_y \psi_l(x, x/\ve) \rangle_{\Omega^\ve_l}
= - |Y|^{-1} \langle \dv_y \v_l(x,y),   \psi_l(x,y) \rangle_{\Omega \times Y_l}.
\end{equation}
Similarly,  using $\dv \, \hat\v_l^\ve =0 $ in $\Lambda_l^\ve$ with  $\hat \v^\ve_a = \hat \v^\ve_v$ on $\Sigma^\ve$ and $\hat \psi \in C_0^\infty(\hat\Lambda; C^\infty_{\text{per}} (\hat Z; C^\infty_0(0,1)))$, we obtain
\begin{eqnarray*} 
0=\lim\limits_{\ve\to 0} 
\langle  \dv \, \hat\v_{av}^\ve(x), \hat \psi(\hat x, x/\ve) \rangle_{\Lambda^\ve_{av}}
 &=& -|\hat Z|^{-1} \langle  \hat \v_{av}(\hat x, y),  \nabla_y \hat \psi(\hat x, y) \rangle_{\hat\Lambda\times Z_{av}}\\
 & =&   
|\hat Z|^{-1} \langle  \dv_y \hat\v_{av}(\hat x, y),  \hat \psi(\hat x, y)  \rangle_{\hat \Lambda\times Z_{av}},  
\end{eqnarray*}
where $\Lambda^\ve_{av}=  \Lambda^\ve_{a} \cup \Sigma^\ve\cup  \Lambda^\ve_{v}$. Therefore, $\dv_y \v_l =0$  in $\Omega\times Y_l$ and $\dv_y \hat \v_{av} =0 $ in $\hat\Lambda\times Z_{av}$, where $l=a,v$.

We now  consider the normal velocity $\hat\v^\ve_l\cdot \n$ on $\hat \Lambda\cap \partial\Lambda_l^\ve$. The  transmission conditions \eqref{Stokes_TransC} yield
\begin{eqnarray*}
\langle \hat\v^\ve_l\cdot \n, \hat \psi(\hat x, \hat x/\ve, 0) \rangle_{\hat \Lambda\cap \partial\Lambda_l^\ve} 
=\ve  \langle \v^\ve_l\cdot \n, \psi(\hat x, 0, \hat x/\ve, 0) \rangle_{\hat \Lambda\cap \partial\Lambda_l^\ve} 
\\ = \ve \langle \dv \v^\ve_l, \psi(x, x/\ve) \rangle_{\Omega^\ve_l} + 
\ve \langle \v^\ve_l, \nabla \psi(x,   x/\ve) \rangle_{\Omega^\ve_l},
\end{eqnarray*}
where $\hat \psi \in C^\infty(\overline{\hat \Lambda}; C^\infty_{\text{per}}(\hat Z; C^\infty[0,1]))$, 
$\psi \in C^\infty(\overline\Omega; C^\infty_{\text{per}}(Y))$ with $\psi = 0$ on $\Gamma_D\times Y$, and 
$\hat \psi(\hat x, \hat x/\ve,0)= \psi(\hat x, 0, \hat x/\ve, 0)$ on $\hat \Lambda$.
 Then using $ \dv \v^\ve_l=0$ in $\Omega_l^\ve$ and $\dv_y \v_l =0$ in $\Omega\times Y_l$, along with the  two-scale convergence of $\v^\ve_l$ and $\hat \v^\ve_l$, implies
\begin{eqnarray*}
|\hat Z|^{-1}\langle \hat\v_l\cdot \n, \hat \psi(\hat x, \hat y, 0) \rangle_{\hat \Lambda\times \hat Z_{l}^0} = 
|Y|^{-1}\langle \v_l, \nabla_y \psi(x, y) \rangle_{\Omega\times Y_l} =0.
\end{eqnarray*}
Hence, $\hat \v_l \cdot \n = 0$ on $\hat \Lambda\times \hat Z_{l}^0$, where $\hat Z_{l}^0= \partial Z_l \cap \{y_n = 0\}$.

Using  $ \dv \, \v_l^\ve = 0$ in $\Omega_l^\ve$ and taking $\psi\in C^\infty(\overline \Omega)$  yield
\begin{eqnarray}\label{div_BC_v}
0=\lim\limits_{\ve\to 0} \langle \dv \, \v_l^\ve,    \psi \rangle_{\Omega_l^\ve} =
\lim\limits_{\ve\to 0} \Big[ -\langle    \v_l^\ve,    \nabla \psi   \rangle_{\Omega_l^\ve}
+ \langle   \v_l^\ve \cdot \n ,  \psi \rangle_{\partial \Omega_l^\ve\cap(\Gamma_D\cup \hat \Lambda) }  \Big] .
\end{eqnarray}
Applying two-scale convergence in the first term on the right-hand side of \eqref{div_BC_v}  and integrating by parts imply 
\begin{equation}\label{div_BC_2}
\begin{aligned}
- \Big\langle \dv \Big[\frac 1{|Y|} \int_{Y_l} \v_l (\cdot,y) dy \Big],   \psi \Big\rangle_{\Omega} +
\Big\langle  \frac 1{|Y|}  \int_{Y_l} \v_l(\cdot,y)  dy \cdot \n ,  \, \psi \Big\rangle_{\partial \Omega}
\\ =\lim\limits_{\ve\to 0}   \langle  \v_l^\ve \cdot \n ,  \psi \rangle_{\partial \Omega_l^\ve\cap (\Gamma_D\cup\hat \Lambda)}  .
\end{aligned}
\end{equation}
Since $C^\infty_0(\Omega)$ is dense in $L^2(\Omega)$,  the last equation yields  
\begin{equation}\label{div37}
\dv \, \left(\frac 1 {|Y|} \int_{Y_l} \v_l (x,y) dy \right)  =0 \qquad  \text{ a.e. in } \, \,   \Omega, \quad \text{ for } \, \, l=a,v.
\end{equation}
Taking  $\psi \in C^\infty(\overline \Omega)$ with $\psi(x)=0$ on $\Gamma_D\cup \hat \Lambda$ in \eqref{div_BC_2}, and using the calculations above,  we obtain 
\begin{equation}\label{zero_flux_v}
\Big(\frac 1 {|Y|} \int_{Y_l} \v_l (\cdot, y) dy \Big) \cdot \n =0  \qquad  \text{on }  \, \,  \partial \hat \Omega \times(-L,0).
\end{equation}
Similarly, taking  $\psi \in C^\infty(\overline \Omega)$ with  $\psi(x)=0$ on $\hat \Lambda$ in \eqref{div_BC_2} we obtain
\begin{equation}\label{flux_v_Gamma_D}
\lim\limits_{\ve\to 0}   \langle  \v_l^\ve \cdot \n ,  \psi \rangle_{\partial \Omega_l^\ve\cap \Gamma_D}=  \Big\langle \frac 1 {|Y|} \int_{Y_l} \v_l (\cdot, y) dy \cdot \n,    \psi \Big\rangle_{\Gamma_D}.
\end{equation}
These calculations imply that 
\begin{equation}\label{BC_v_Lambda}
\lim\limits_{\ve\to 0}   \langle  \v_l^\ve \cdot \n ,  \psi \rangle_{\partial \Omega_l^\ve\cap \hat \Lambda}= 
\Big\langle \frac1{|Y|} \int_{Y_l} \v_l(\cdot,y)  dy \cdot \n ,  \, \psi \Big\rangle_{\hat\Lambda} \qquad  \text{ for } \; \; \psi \in C^\infty(\overline \Omega).
\end{equation}

We now consider a test function  $\hat \phi\in C^\infty( \Lambda^\ve)$, such that $\hat \phi$ is constant in $x_n$ and $\hat \phi(x) = 0 \text{ on } \partial \hat \Omega\times(0,\ve)$. Applying $\dv\, \hat \v^\ve_l(x)=0$ in $\Lambda^\ve_l$ with $\hat \v^\ve_l (x)=0$ on the boundaries $R^\ve_l$, $(\partial \hat \Omega \times (0,\ve))\cap \partial \Lambda_l^\ve$,  and $(\hat \Omega \times \{ \ve\})\cap \partial \Lambda^\ve_l$,  along with $\hat \v_a^\ve = \hat \v_v^\ve$ on $\Sigma^\ve$, yields 
\begin{equation}\label{BC_lambda v_hat}
0=\lim\limits_{\ve\to 0} \frac 1\ve \langle  \dv\,  \hat \v^\ve_{av},  \hat \phi \rangle_{\Lambda^\ve_{av}} =  
\lim\limits_{\ve\to 0}\big( - \frac 1\ve \langle  \hat \v^\ve_{av},  \nabla_{\hat x}\hat \phi  \rangle_{\Lambda^\ve_{av}}+   
\frac 1\ve  \langle  \hat \v^\ve_{av} \cdot \hat \n,  \hat \phi \rangle_{\partial \Lambda^\ve_{av}\cap \hat \Lambda}  \big),
\end{equation}
where $\hat \v^\ve_{av}=\hat \v^\ve_a \chi_{\Lambda^\ve_a} + \hat \v^\ve_v \chi_{\Lambda^\ve_v}$.
 The transmission condition $\frac 1 \ve \hat \v_l^\ve \cdot \hat\n = \v_l^\ve \cdot \hat\n$ on $\hat \Lambda \cap \partial \Omega_l^\ve$ along with   the two-scale convergence of $\hat \v^\ve_l$  and the convergence in \eqref{BC_v_Lambda} imply  
\begin{eqnarray*}
|\hat Z|^{-1}\langle  \hat \v_{av},\nabla_{\hat x}  \hat \phi  \rangle_{\hat \Lambda\times Z_{av}} &=&
\langle |Y|^{-1}  \v_a \cdot \hat\n ,  \, \hat \phi \rangle_{\hat\Lambda\times Y_a} + \langle |Y|^{-1}  \v_v \cdot \hat\n ,  \, \hat \phi \rangle_{\hat\Lambda\times Y_v},  
\end{eqnarray*}
where $\hat \n$ is the external normal vector to $\partial \Lambda^\ve\cap \hat \Lambda$. Thus 
\begin{equation}\label{bc_hat_Lambda}
\dv_{\hat x} \Big(\frac 1{|\hat Z|} \int_{Z_{av}} \hat \v_{av} \,  dy\Big) = \frac1{|Y|} \int_{Y_a} \v_a  \, dy \cdot \n + \frac1{|Y|} \int_{Y_v} \v_v \,   dy \cdot \n  \; \;  \text{ on } \;  \hat \Lambda,
\end{equation}
where $\n$ is the external normal vector to $\partial \Omega\cap\hat \Lambda$, 
 and 
$$
\frac 1{|\hat Z|} \int_{Z_{av}} \hat \v_{av} (x, y) dy \cdot \n  =0 \quad \text{ for  } \; \; x \in \partial \hat \Lambda. 
$$

Considering $\v^\ve= \v^\ve_a \chi_{\Omega_a^\ve} +  \v^\ve_v \chi_{\Omega_v^\ve} +  \ve^{-1}\hat \v^\ve_a \chi_{\Lambda_a^\ve} +  \ve^{-1}\hat \v^\ve_v \chi_{\Lambda_v^\ve} $ we obtain 
\begin{eqnarray*}
0=\int_{\Omega_{av}^\ve\cup \Lambda_{av}^\ve} \dv \v^\ve dx = 
\int_{\Gamma_D\cap \partial\Omega^\ve_a} \v^\ve_a \cdot \n \, d\hat x + \int_{\Gamma_D\cap \partial\Omega^\ve_v} \v^\ve_v \cdot \n \, d\hat x,
\end{eqnarray*}
where $ \Omega^\ve_{av}=\Omega^\ve_a\cup\Omega^\ve_v$.
Then the convergence in \eqref{flux_v_Gamma_D} yields 
\begin{equation}\label{BC_overeg_Gamma}
\frac 1 {|Y|}\int_{\Gamma_D} \Big[\int_{Y_a} \v_a (\cdot, y) dy +  \int_{Y_v} \v_v (\cdot, y) dy \Big]\cdot \n \, d\hat x=0.
\end{equation}
Considering $\dv \big(\int_{Y_a} \v_a  dy +\int_{Y_v} \v_v dy \big)=0 $ in $\Omega$   and using  \eqref{BC_overeg_Gamma}  imply
\begin{equation}\label{over_flux_Lambda}
\int_{\hat \Lambda} \Big[\int_{Y_a} \v_a  (\cdot, y)\,  dy +\int_{Y_v} \v_v (\cdot, y)\,  dy \Big] \cdot \n \, d\hat x =0.
\end{equation}

 We now consider functions $\psi_l$ and $\hat \psi$ such that 
\begin{enumerate}
\item[(a)]  $\psi_l \in C^\infty(\overline\Omega; C^\infty_{\text{per}}(Y))$, $\dv_y \psi_l =0$ in $\Omega\times Y$,    $\psi_l =0 $ on  $(\partial \hat \Omega\times(-L,0)\cup \Gamma_D)\times Y$ and on  $\Omega \times \Gamma_l$, 
\item[(b)]   $\hat \psi \in  C^\infty_0(\hat\Lambda; C^\infty_{\text{per}}(\hat Z; C^\infty[0,1]))$,
  $\dv_y \hat \psi =0$ in $\hat \Lambda \times Z$,   $\hat \psi = 0$ on $\hat \Lambda\times (R_{av}\cup \hat Z^1_{av})$.
\end{enumerate}  
Then we choose  $\phi_l(x) = \psi_l(x, \frac x \ve)$  and  $\hat \phi_l(x) =  \hat \psi(x, \frac x \ve)$, $l=a,v$, as  test functions in \eqref{micro_weak_flow}. The  two-scale convergence of  $(\v_l^\ve, p^\ve_l)$ and $(\hat \v_l^\ve, \hat p^\ve_l)$, with $l=a,v$,  implies 
\begin{align}\label{macro_veloc_1}
\begin{aligned}
&\quad \, \, \frac 1{|Y|}\sum_{l=a,v}\Big(\langle  2 \mu S_y \v_l,  S_y \psi_l \rangle_{\Omega\times Y_l} -  \langle  p_l,  \dv_x \psi_l \rangle_{\Omega\times Y_l}  -  \frac 1{L} \langle p_l^0, \psi_{l,n} \rangle_{\Omega\times Y_l}  \Big) \\ & + \frac 1 {|\hat Z|}\left(\langle 2 \mu S_y \hat \v_{av},  S_y \hat \psi  \rangle_{\hat \Lambda\times Z_{av}}    -
\langle   \hat p,  \dv_{\hat x} \hat \psi \rangle_{\hat\Lambda \times Z_{av}} \right)
= 0 .
\end{aligned}
\end{align} 
We consider functions  $\psi_l$ and $\hat \psi$ such that
\begin{enumerate}
\item[(a)]  $\psi_l \in C^\infty_0(\Omega, C^\infty_{\text{per}}(Y))$  with $\dv_y \psi_l =0$,  $\psi_l = 0$ on $\Omega \times \Gamma_l$, and 
\item[(b)] $\hat \psi \in C^\infty_0(\hat \Lambda, C^\infty_{\text{per}}(\hat Z; C^\infty_0(0,1)))$ with 
$\dv_y \hat \psi =0$, $\dv_{\hat x} \langle\hat \psi, 1\rangle_{Z_{av}} =0$, and $\hat \psi(\hat x, y) = 0$ on $\hat \Lambda\times R_{av}$.
\end{enumerate}
Using the  characterization of the  orthogonal complement to the  space of divergence-free functions (see, e.g.,   \cite{Hornung}), we obtain the existence of $p_l^1 \in L^2(\Omega\times Y_l)/\mathbb R$,  $\hat p_{av}^1 \in L^2(\hat\Lambda\times Z_{av})/\mathbb R$, and $\tilde p \in H^1(\hat \Lambda)/\mathbb R$ 
 such that 
\begin{align}\label{macro_two-scale}
\begin{aligned}
& -\mu \Delta_y \v_l + \nabla_x p_l + \nabla_y p^1_l= \frac 1 L p_l^0 \, \textbf{e}_n \quad  &&\text{ in } \Omega\times Y_l, &&  \qquad l=a,v, \\
& -\mu \Delta_y \hat \v_{av} + \nabla_{\hat x} \tilde p + \nabla_y \hat p^1_{av} = 0  \quad  &&\text{ in } \hat \Lambda\times Z_{av}. 
\end{aligned}
\end{align}
Combining equations \eqref{macro_two-scale} and \eqref{macro_veloc_1},  and considering
$\hat \psi \in  C^\infty(\overline{\hat\Lambda}; C^\infty_{\text{per}}(\hat Z; C^\infty[0,1]))$ with 
 $\dv_y \hat \psi =0$ in $\hat \Lambda \times Z$, $\langle\hat \psi, 1\rangle_{Z_{av}} \cdot \n =0$ on $\partial \hat \Lambda$, and    $\hat \psi = 0$ on $\hat \Lambda\times (R_{av}\cup \hat Z^0_{av}\cup \hat Z^1_{av} )$, we obtain 
\begin{eqnarray*}    
|\hat Z|^{-1} \langle   \hat p - \tilde p,  \dv_{\hat x} \hat \psi \rangle_{\hat\Lambda \times Z_{av}} +
|Y|^{-1} \langle  p_a, \psi\cdot \n \rangle_{\hat \Lambda \times Y_a} +  |Y|^{-1} \langle  p_v, \psi\cdot \n \rangle_{\hat \Lambda \times Y_v} =0.
\end{eqnarray*}
Thus using equality \eqref{bc_hat_Lambda} we obtain   $p_a= p_v= \hat p$  and $\tilde p = 2 \hat p$ on $\hat \Lambda$. 

   Relaxing now the assumptions on $\hat \psi $ and using
   $\hat \psi \cdot \n= 0$ on $\hat \Lambda\times \hat Z^0_{av}$ imply
  $$ (2\mu \S_y  \hat \v_{av} - \hat p_{av}^1 I ) \, \n\times \n = 0 \qquad \text{ on } \hat\Lambda\times  \hat Z_{av}^0. $$
Setting $\bar p_l= p_l - p_l^0 \frac {x_n} L$ and omitting the bar for the sake of clarity, we obtain the two-scale model 
\begin{align}\label{two-scale_Omega}
\begin{aligned}
-\mu\Delta_y \v_l + \nabla_x p_l + \nabla_y p_l^1 = 0, & \qquad \dv_y \v_l =0  \quad   \text{ in } \Omega\times Y_l , \quad l=a,v\\ 
\v_l = 0 \quad   \text{ on } \Omega\times \Gamma_l ,  & \qquad 
\v_l, \, p_l^1 \quad \text{ are }  Y-\text{periodic}, \\
p_l = p_l^0  \quad     \text{ on } \Gamma_D\times Y_l  
\end{aligned}
\end{align}
and 
\begin{align}\label{two-scale_Lambda}
\begin{aligned}
&-\mu\Delta_y \hat \v_{av} +2 \nabla_{\hat x} \hat p + \nabla_y \hat p_{av}^1 =0, \quad && \dv_y \hat \v_{av} =0\quad  && \text{ in } \hat \Lambda \times Z_{av} , \\
&(2 \mu S_y \hat \v_{av} - \hat p^1_{av} I)\, \n\times \n= 0, && \hat \v_{av}\cdot \n =0 && \text{ on } \hat \Lambda \times \hat Z_{av}^0,    \\
&\hat \v_{av} =0 \qquad   \text{ on } \hat \Lambda\times (R_{av}\cup \hat Z_{av}^1),  &&
 \hat\v_{av}, \, \hat p_{av}^1 && \text{ are } \hat Z-\text{periodic}.
\end{aligned}
\end{align}

Finally, for $(x,y)\in \Omega\times Y_l$ and  $(\hat x,y) \in \hat \Lambda\times Z_{av}$, we consider the ansatz 
\begin{eqnarray}\label{ansatz_fluid}
\begin{aligned}
& \v_l(x,y)=- \sum\limits_{j=1}^n \partial_{x_j} p_l(x) \omega^j_l(y), \quad &&
p_l^1(x,y)= -\sum\limits_{j=1}^n \partial_{x_j} p_l(x) \pi^j_l(y),  \\
& \hat \v_{av}(\hat x, y)=- 2\sum\limits_{j=1}^{n-1} \partial_{x_j} \hat p(\hat x)  \hat \omega^j (y), \quad && 
\hat p_{av}^1(\hat x, y)= -2\sum\limits_{j=1}^{n-1} \partial_{x_j} \hat p(\hat x) \hat \pi^j(y) ,  
\end{aligned}
\end{eqnarray}
where $l=a,v$, and  $(\omega^j_l, \pi_l^j)$, $(\hat \omega^j, \hat \pi^j)$ are solutions of 
the unit cell problems \eqref{eq:omega1} and  \eqref{eq:omega2} respectively. 
Applying the ansatz \eqref{ansatz_fluid} to  equations  \eqref{two-scale_Omega} and \eqref{two-scale_Lambda}, and using equations \eqref{div37} and \eqref{bc_hat_Lambda}, yields the macroscopic equations  \eqref{macro_macro_1} and \eqref{macro_macro_2}  for $\v^0_l(\cdot) = \frac 1 {|Y|} \int_{Y_l} \v_l(\cdot,y)  dy$,  $p_l$,   $\hat \v^0_{av} (\cdot) = \frac 1 {|\hat Z|} \int_{Z_{av}} \hat \v_{av}(\cdot,y)  dy$, and $\hat p$. The integral condition in \eqref{over_flux_Lambda} ensures the well-posedness of the macroscopic model 
 \eqref{macro_macro_2}. Considering the differences of two solutions  $p_l^1- p_l^2$ and $\hat p^1- \hat p^2$
of \eqref{macro_macro_1} and  \eqref{macro_macro_2}, and using the Dirichlet boundary conditions on $\Gamma_D$ and the continuity conditions on $\hat \Lambda$, we obtain the uniqueness of the solution of the macroscopic model. 
\end{proof}

\section{Derivation of macroscopic equations for oxygen concentrations}

In this section, we continue our derivation of the homogenized equations for the microscopic system \eqref{Stokes_Omega}--\eqref{Diff_InitC} by turning our attention to the oxygen concentrations in arterial blood, venous blood, and tissue. Theorem~\ref{6-1} provides the macroscopic equations dictating the dynamics of the various oxygen concentrations as $\ve\rightarrow 0$, and it complements Theorem~\ref{5-1} that was proven in the previous section. For the remainder of this section, we define $\hat \v_{av}(\hat x,y)=\hat \v_a(\hat x,y)\chi_{Z_a}(y) +\hat \v_v(\hat x,y) \chi_{Z_v}(y) $ for a.a. $(\hat x,y) \in \hat \Lambda \times Z_{av}$.

\begin{proof}[\bf Proof of Theorem \ref{6-1}]
We consider $\psi_l(t,x) = \phi^1_l(t,x) + \ve \phi^2_l(t,x, \frac x \ve)$, for $l=a,v$, and  $\hat \psi(t,x) = \hat \phi_1(t,\hat x) + \ve\hat \phi_2(t,\hat x, \frac x \ve) $ as test functions in \eqref{micro_weak_av},  where 
\begin{enumerate}
\item[(a)] $ \phi^1_l \in C^\infty(\overline \Omega_T)\cap  L^2(0,T;W(\Omega))$ with $\phi^1_l(t, \hat x, 0) =\hat \phi_1(t,\hat x)$ in $\hat \Lambda_T$, and 
$\phi_l^2 \in C^\infty_0(\Omega_T; C_{per} (Y))$
\item[(b)]  $\hat \phi_1 \in C^\infty(\overline{\hat \Lambda_T})$ and $\hat \phi_2 \in C^\infty_0(\hat \Lambda_T; C_{per}^\infty (\hat Z; C_0^\infty(0,1)))$.
\end{enumerate}
Considering  $\Omega^\delta= \{ x\in \Omega : \text{dist}(x,\partial \Omega) > \delta\} $ and $\tilde \Omega^{\ve,\delta}_l= \{ x\in \Omega^\ve_l : \text{dist}(x,\partial \Omega^\ve_l) > \delta\} $
we can write 
\begin{eqnarray*}
 \langle \v_l^\ve c_l^\ve, \nabla \psi_l \rangle_{\Omega^\ve_{l,T}}
= \frac 1{|Y|} \langle  \mathcal T^\ast_\ve(\v_l^\ve) \mathcal T^\ast_\ve(c_l^\ve),  \mathcal T^\ast_\ve(\nabla \psi_l)\rangle_{\Omega^\delta_T\times Y_l} +  \langle \v_l^\ve c_l^\ve, \nabla \psi_l \rangle_{\tilde\Omega^{\ve,\delta}_{l, T}}\;.
\end{eqnarray*}
Due to the boundedness of $c_l^\ve$ and  the {\it a priori} estimates for $\v_l^\ve$,  we obtain 
\begin{eqnarray*}
| \langle \v_l^\ve c_l^\ve, \nabla \psi_l \rangle_{\tilde\Omega^{\ve,\delta}_{l,T}}| \leq 
C \|\v^\ve_l\|_{L^2(\tilde\Omega^{\ve,\delta}_l)}\left[ \|\nabla \phi_l^1 \|_{L^2(\tilde\Omega^\delta_T)}+ \ve  \|\nabla \phi_l^2 \|_{L^2(\tilde\Omega^\delta_T\times Y_l)}\qquad \qquad\right.\\
\left.+ \|\nabla_y \phi_l^2 \|_{L^2(\tilde\Omega^\delta_{T}\times Y_l)} \right]\to 0 \; \;  \text{ as } \;  \delta \to 0\; , 
\end{eqnarray*}
where $\tilde \Omega^{\delta}= \{ x\in \Omega : \text{dist}(x,\partial \Omega) < \delta\} $.  Applying the weak convergence of $\mathcal T^\ast_\ve(\v_l^\ve)$,  the strong convergence of 
$  \mathcal T^\ast_\ve(\nabla \psi_l)$, the
 local  strong convergence of 
$ \mathcal T^\ast_\ve(c_l^\ve)$,   and letting   $\ve \to 0$ and  $\delta \to 0$ in that order,  we obtain  
$$
\langle \v_l^\ve  c_l^\ve, \nabla \psi_l\rangle_{\Omega^\ve_{l, T}}
\to1/|Y| \langle \v_l  c_l, \nabla \phi_l^1+ \nabla_y \phi_l^2\rangle_{\Omega_T\times Y_l} \; .
$$
 In a similar way as for $\v^\ve_l$,  the  regularity of $\hat \psi$ and the \textit{a priori} estimates and convergence results for $\hat \v^\ve_{av}$ and $\hat c^\ve_l$ imply 
 $$\frac 1\ve \langle \hat \v^\ve_{av} \hat c^\ve_{av}, \nabla \hat \psi  \rangle_{\Lambda^\ve_{av}, T} \to  |\hat Z|^{-1} \langle \hat \v_{av} \hat c, \nabla \hat \phi_1 + \nabla_y \hat\phi_2  \rangle_{\hat\Lambda_T\times Z_{av}} \;\; \text{ as } \;\; \ve \to 0 \;   \text{ and } \;  \delta \to 0. $$

The weak convergence of $\mathcal T^\ast_\ve(c_l^\ve)$ and $\mathcal T^\ast_\ve (\nabla c_l^\ve )$, in conjunction with the strong convergence of 
 $\mathcal T^\ast_\ve(\psi_l )$ and $\mathcal T^\ast_\ve(\nabla \psi_l )$, imply the convergence of $\langle \partial_t c_l^\ve,  \psi_l \rangle_{\Omega_l^\ve, T}$ and $\langle D_l^\ve \nabla c_l^\ve , \nabla \psi_l \rangle_{\Omega_l^\ve, T}$.
  Similar arguments  pertaining to the unfolding operator $\mathcal T^{\ast, bl}_\ve$ and the convergence results for unfolded sequences  prove the convergence of  
$\dfrac 1 \ve \langle \partial_t  \hat c_l^\ve , \hat \psi \rangle_{\Lambda_l^\ve, T}  $ and 
$\dfrac 1 \ve \langle   \hat D_l^\ve\nabla  \hat c_l^\ve, \nabla \hat\psi \rangle_{\Lambda_l^\ve, T} 
$. 
The weak convergence of $\mathcal T^\ast_\ve(c_l^\ve)$  in $L^2(\Omega_T\times\Gamma_l)$ and of $\mathcal T^{\ast, bl}_\ve(\hat c_l^\ve)$ in $L^2(\hat\Lambda_T\times R_l)$  (shown in Lemma~\ref{lem:convergence}) ensure the convergence of integrals over $\Gamma_l^\ve$ and $R^\ve_l$. 

Thus, we obtain the  macroscopic equations 
\begin{eqnarray*}
&&\, \,  \, \, \frac 1 {|Y|}\sum_{l=a,v}\left[\langle\partial_t c_l,  \phi_l^1 \rangle_{\Omega_T\times Y_l}  + \langle 
 D_l(y)(\nabla c_l + \nabla_y c_l^1)- \v_l  c_l, \nabla \phi_l^1+ \nabla_y \phi_l^2\rangle_{\Omega_T\times Y_l} \right]\\ 
&& +
\frac 1 {|\hat Z|} \left[\langle \partial_t \hat c, \hat \phi_1 \rangle_{\hat \Lambda_T\times Z_{av}} +
\langle   \hat D_{av}(y)(\nabla_{\hat x} \hat c + \nabla_y \hat c^1) - \hat \v_{av}\hat c, \nabla_{\hat x} \hat\phi_1 + \nabla_y \hat \phi_2\rangle_{\hat \Lambda_T\times Z_{av}}  \right]\\
 && 
= \frac 1{|Y|}  \sum_{l=a,v}\langle  \lambda_l( c_s- c_l),  \phi^1_l \rangle_{\Omega_T\times \Gamma_l}
+
\frac 1{|\hat Z|}\sum_{l=a,v} \langle\hat  \lambda_l(\hat c_s- \hat c),\hat \phi_1\rangle_{\hat \Lambda_T\times R_{l}}. 
\end{eqnarray*}
Furthermore, setting $\phi^1_l(t,x)=0$ in $\Omega_T$, with $l=a,v$,  and $\hat \phi_1(t,\hat x)=0$ in $\hat \Lambda_T$  we obtain 
\begin{equation}\label{macro_2-ox}
\begin{aligned}
& \frac 1{|Y|} \sum_{l=a,v}\langle  D_l(y)(\nabla c_l + \nabla_y c_l^1) 
- \v_l  c_l , \nabla_y \phi_l^2 \rangle_{\Omega_T\times Y_l} \\
+ &\frac 1{|\hat Z|}\langle \hat D_{av}(y)(\nabla_{\hat x} \hat c + \nabla_y \hat c^1) - \hat \v_{av} \hat c ,  \nabla_y \hat \phi_2 \rangle_{\hat \Lambda_T\times Z_{av}}  =0.  
\end{aligned}
\end{equation}

We now employ the divergence-free property of the velocity fields in $\Omega\times Y_l$ and $\hat \Lambda\times Z_{av}$ and the zero-boundary conditions on $\Gamma_l$ and  $R_l\cup Z_{av}^0\cup Z_{av}^1$.
These, and   the fact that $c_l$ and  $\hat c_{av}$ are independent of $y$,   yield
\begin{eqnarray*}
&&  \langle \v_l  c_l,  \nabla_y \phi_l^2 \rangle_{\Omega_T\times Y_l}  = -  
 \langle  \dv_y ( \v_l  \, c_l ),   \phi_l^2\rangle_{\Omega_T\times Y_l} +
\langle  \v_l \cdot \n \,   c_l,   \phi_l^2 \rangle_{\Omega_T \times \partial Y_l} = 0, \quad l=a,v, \\
&&  \langle \hat \v_{av}  \hat c,  \nabla_y \hat \phi_2 \rangle_{\hat \Lambda_T\times Z_{av}}  =  -  
 \langle  \dv_y (\hat  \v_{av}  \, \hat c ),   \hat \phi_2\rangle_{\hat \Lambda_T\times Z_{av}} +
\langle  \hat \v_{av} \cdot \n \,  \hat  c,   \hat \phi_2 \rangle_{\hat \Lambda_T \times \partial Z_{av}} = 0.
\end{eqnarray*}

Thus, taking first $\hat \phi_2(t,\hat x,y)=0$ in $\hat \Lambda_T\times Z$   and $\phi_l^2 \in C^\infty_0(\Omega_T; C^\infty_{\text{per}}(Y))$ with $\phi_l^2(t,x,y) =0 $ for   $y \in Y\setminus  Y_l$, $(t,x) \in \Omega_T$, and then $\hat \phi_2 \in C^\infty_0(\hat \Lambda_T; C_{per}^\infty (\hat Z; C_0^\infty(0,1)))$
  in \eqref{macro_2-ox}, we have 
\begin{eqnarray*}
&& \langle  D_l(y)(\nabla c_l + \nabla_y c_l^1), \nabla_y  \phi_l^2  \rangle_{\Omega_T\times Y_l} =0 \qquad \text{ for } \; \; l=a, v, \\
&&  \langle \hat D_{av}(y)(\nabla_{\hat x} \hat c + \nabla_y \hat c^1),  \nabla_y \hat \phi_2 \rangle_{\hat \Lambda_T\times Z_{av}} =0.
\end{eqnarray*}
 Using the linearity of the equations above,  we consider the ansatz 
\begin{equation*} 
c^1_l(t,x,y) = \sum\limits_{j=1}^n \partial_{x_j} c_l(t,x) w_l^j(y)  \quad \text{ for } \; \; l=a,v,   \quad 
\hat c^1(t,\hat x, y) = \sum\limits_{j=1}^{n-1} \partial_{x_j} \hat c(t, \hat x) \hat w^j_{av}(y)  ,
\end{equation*}
where $w^j_l$ and  $\hat w_{av}^j$  are solutions of the unit cell problems \eqref{UnitCell_Ox_1} and \eqref{unit_hat_1} respectively.

Then for $\phi_l^2=0$ and $\hat \phi_2 =0$, and using the ansatz for $c_l^1$   and $\hat c^1$, we obtain 
\begin{eqnarray*}
&&\sum_{l=a,v} \int_{\Omega_T}\left(\frac {|Y_l| }{|Y|} \partial_t c_l \, \phi_l^1 +(\mathcal A_l \nabla c_l-  \v_l^0 c_l) \nabla \phi_l^1  -    \lambda_l\frac{|\Gamma_l|}{ |Y|}( c_s- c_l)  \phi_l^1 \right) dx dt   
\\
&& + 
\int_{\hat \Lambda_T} \left(\frac { |Z_{av}|} {|\hat Z|}  \partial_t \hat c \, \hat \phi_1 + ( \hat{\mathcal{A}}_{av}  \nabla_{\hat x} \hat c - \hat \v^0_{av} \hat c)  \nabla_{\hat x} \hat \phi_1  -\sum_{l=a,v}   \hat\lambda_l \frac{|R_l|}{|\hat Z|} (\hat c_s- \hat c) \hat \phi_1 \right)  d\hat x dt=0,
\end{eqnarray*}
  where $\mathcal A_l $, $\v^0_l$, 
$\hat{\mathcal{A}}_{av}$ and $\hat \v^0_{av}$ are defined in \eqref{macro_vv} and \eqref{macro_veloc_ve}.  From the continuity conditions  \eqref{Diff_TransC}, we obtain 
$c_a(t, \hat x, 0) = \hat c(t,\hat x)$, $c_v(t, \hat x, 0) = \hat c(t,\hat x)$ on $\hat \Lambda_T$.
Considering $\phi_l^1\in C^\infty_0(\Omega_T)$ and $\hat \phi_1 =0$ and integrating by parts result in  the macroscopic equations for $c_a$ and $c_v$ in \eqref{macro_cl}-\eqref{macro_hat_c}. Considering
\begin{enumerate} 
\item[(a)] $\hat \phi_1 \in C^\infty_0(\hat \Lambda_T)$,  $\phi_l^1\in C^\infty(\overline\Omega_T)$ with  $\phi_l^1(t,x) =0$ on $\Gamma_D$ and $\phi_l^1(t,\hat x,0) = \hat \phi_1(t,\hat x)$ on $\hat \Lambda_T$,   and 
\item[(b)]    $\hat \phi_1 \in C^\infty(\overline{\hat \Lambda}_T)$, $\phi_l^1\in C^\infty(\overline\Omega_T)$ with  $\phi_l^1(t,x) =0$ on $\Gamma_D$ and $\phi_l^1(t,\hat x,0) = \hat \phi_1(t,\hat x)$ on $\hat \Lambda_T$,
\end{enumerate}
in that order, and integrating by parts result in  the macroscopic equation for    $\hat c$ in \eqref{macro_cl}-\eqref{macro_hat_c}.
Similar arguments imply the macroscopic equations for $c_s$ and $\hat c_s$. 
 The assumptions on the initial conditions ensure the existence of $\hat c^0, \hat c^0_s \in H^1(\hat \Lambda)$ 
such that $\hat c^{\ve, 0}\to \hat c^0$, $\hat c^{\ve, 0}_s \to\hat c^0_s$ in the two-scale sense.
  This and the two-scale convergence of $\partial_t c^\ve_l$, $\partial_t \hat c^\ve$  and $\partial_t \hat c^\ve_s$ imply that $c_l$,  $\hat c$ and $\hat c_s$ satisfy the initial conditions, where $l=a,v,s$. Considering the equations for the difference of two solutions of the macroscopic problem \eqref{macro_cl}-\eqref{macro_hat_c} yields the uniqueness of the solutions.   Finally,  taking $c_l^{-}$, $\hat c^{-}$, $\hat c^{-}_s$,  $(c_l - A)^{+}$,  $(\hat c - A)^{+}$ and  $(\hat c_s - A)^{+}$, for some $A\geq \max_{l=a,v,s}\{ \sup_{\Omega_T} c_{l,D}(t,x), \sup_{\Omega} c^0_l(x), \sup_{\hat \Lambda}\hat c^0(\hat x), \sup_{\hat \Lambda}\hat c^0_s(\hat x)\}$,  as test functions in \eqref{macro_cl}-\eqref{macro_hat_c} we obtain the non-negativity and boundedness of the solutions of the macroscopic problem.
\end{proof}


\section{The $\delta$ scaling for the skin layer with $0<\ve << \delta<<1$}\label{section7}

In this final section, we consider an alternative scaling for the depth $\delta$ of the skin layer. Specifically, we assume that the adopted tissue  geometry is characterized by two distinct  length scales: a scale $\delta>0$ representing the depth of the skin layer and a separate length scale $\ve>0$ characterizing the distance between arteries. In the remainder of this section, we assume that $0<\ve << \delta<<1$,  and we let  first $\ve \to 0$ and then   $\delta\rightarrow 0$. Under this scaling, the skin layer has a depth of multiple unit cells (of size $\ve$), and we assume that the arterial branching pattern is such that flow of blood is permitted between neighboring unit cells in the skin layer.

\begin{figure}
\begin{center}
\includegraphics[height=4.cm]{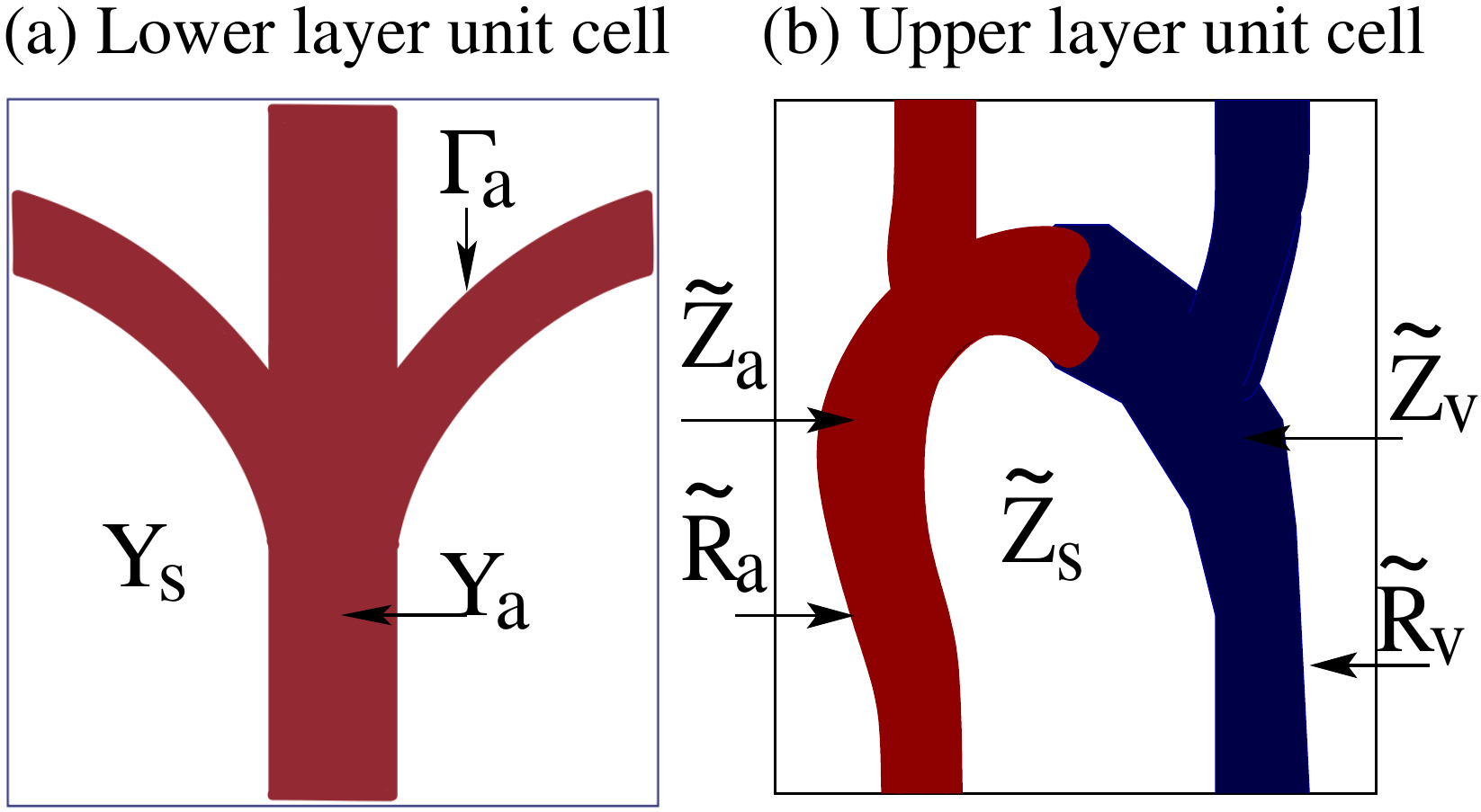}
\end{center}
\caption{Two-dimensional schematic representation of the two distinct, three-dimensional unit-cell geometries used in the microscopic model: (a) unit-cell geometry corresponding to the lower layer, i.e. the fat tissue layer; (b) unit-cell geometry corresponding to the upper layer, which represents the dermic and epidermic layers of the skin. Only the arterial blood vessels are shown in the fat tissue layer.}\label{fig3}
\end{figure}

\begin{figure}
\begin{center}
\includegraphics[height=2.5 cm]{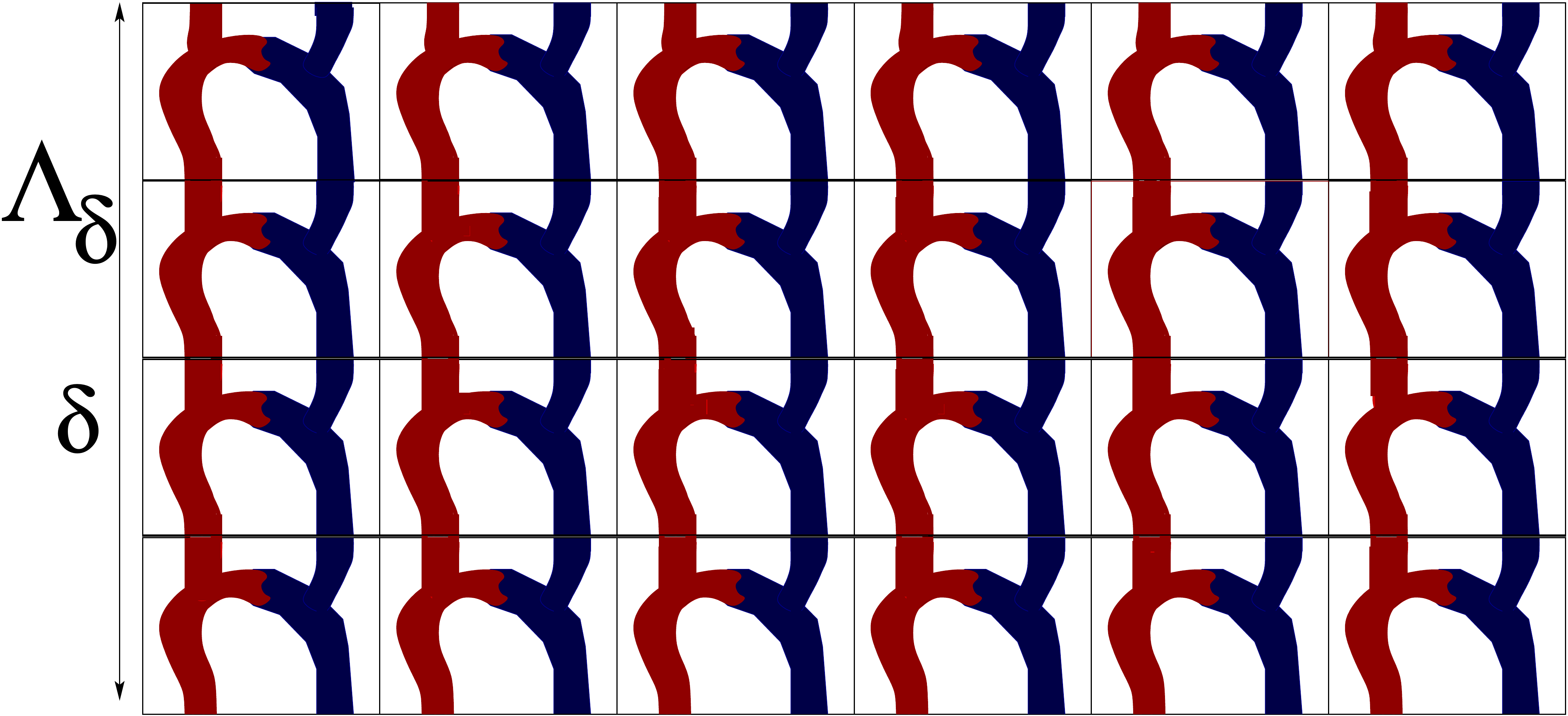}\; \; 
\includegraphics[height=3.2cm]{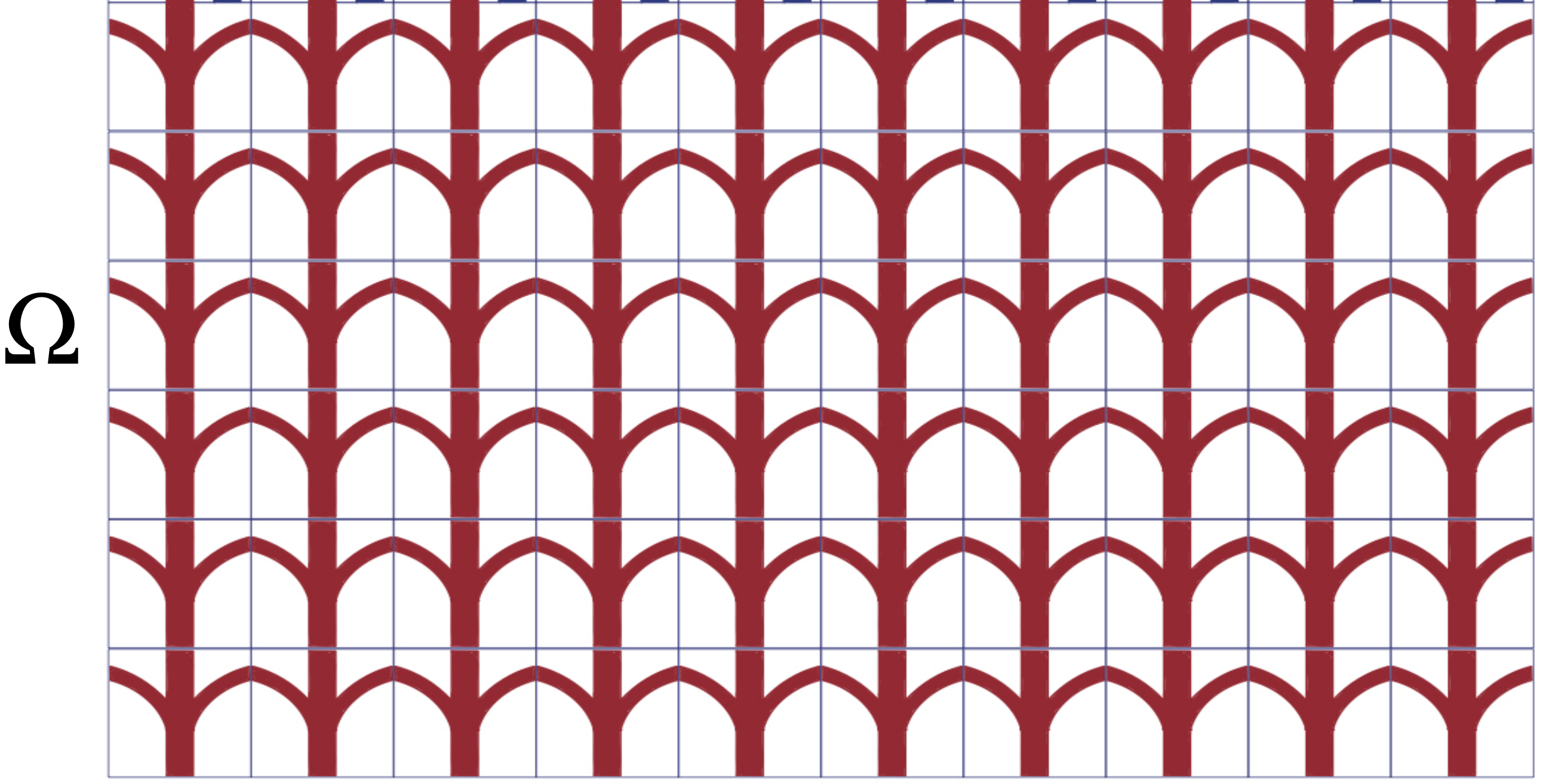}
\end{center}
\caption{Two dimensional schematic representation of the three-dimensional tissue layers discussed in the text. The domain on the left (denoted by $\Lambda^\delta$ in the text) corresponds to the dermic and epidermic layers of the skin, whereas the domain on the right (denoted by $\Omega$ in the text) corresponds to fat tissue. Only the arterial blood vessels are shown in the fat tissue layer. Arteries (in red) and veins (in blue) are shown in the skin tissue layer, which is characterized by the presence of arterial-venous connections, i.e. geometric regions where arteries and veins meet.  }\label{fig4}
\end{figure}

\subsection{Derivation of macroscopic equations for velocity fields}

We first derive the macroscopic equations for the arterial and venous blood velocity fields in the two tissue layers under the scaling assumption $0<\ve << \delta<<1$. The microscopic equations for the fluid flow  in the main tissue are as in \eqref{Stokes_Omega}. In the skin layer $\Lambda_\delta$, $(\hat \v_a^\ve, \hat p_a^\ve)$ and   $(\hat \v_v^\ve, \hat p_v^\ve)$ are  assumed to   satisfy
\begin{eqnarray}\label{Stokes_Lambda_delta}
\begin{cases}
 - \ve^2 \,\mu\,  \Delta \hat\v^\ve_l +  \nabla \hat p_l^\ve= 0 \; ,   \quad   \dv \, \hat\v_l^\ve=0  &  \quad \text{ in }  \Lambda_{l, \delta}^\ve, \medskip\\
  \hat\v_l^\ve=0  & \quad \text{ on }  \widetilde{R}_{l, \delta}^{\ve} \; , 
  \end{cases}
\end{eqnarray}
where $l= a, v$. We  impose appropriate  transmission  conditions on $\hat \Lambda$
\begin{eqnarray}\label{Stokes_TransC_delta}
\begin{cases}
(-2\,\ve^2  \mu  \operatorname{S}\!\v_l^\ve+p_l^\ve I)\cdot \n=(-2\, \ve^2 \mu   \S\hat \v_l^\ve+ \hat p_l^\ve I)\cdot \n &\quad \text{on }\partial\Omega_l^\ve\cap \hat \Lambda\; , \medskip\\
\v_l^\ve =  \displaystyle\frac{1}{\delta} \hat\v_l^\ve &  \quad \text{on }\partial\Omega_l^\ve\cap \hat \Lambda \; ,
\end{cases}
\end{eqnarray}
where $l=a,v$, 
along with boundary conditions \eqref{Stokes_BC_O} at the external boundaries  and the continuity conditions given in \eqref{Stokes_ContC}.  Moreover, 
\begin{equation}\label{Stokes_BC_L_delta}
\hat \v_l^\ve =0   \quad \text{ on } \partial \hat\Omega\times(0,\delta)\cap \partial \Lambda_l^\delta , \qquad \quad 
\hat \v_l^\ve =0  \quad \text{ on } \hat\Omega\times\{\delta\}\cap \partial \Lambda_l^\delta, \qquad l=a,v .
\end{equation}

\begin{proof}[\bf Proof of Theorem \ref{th:macro_ve}]
Similarly to Section~\ref{weak_sol}, we derive a priori estimates for $\v_l^{\ve}$ and $\hat \v_l^{\ve}$.
To derive the macroscopic equations \eqref{macro_ve_1}, 
we first consider   $\phi_l(x)=\ve\psi_l\left(x, \frac x \ve\right)$  and   $\hat \phi(x)=\ve\hat \psi\left(x, \frac x \ve\right)$ with
 $\psi_l \in C^\infty_0(\Omega, C^\infty_{\text{per}}(Y))$,  
 $\hat \psi \in C^\infty_0(\Lambda_\delta; C^\infty_{\text{per}}(\widetilde{Z}))$, $\psi_l = 0$ on $\Omega\times \Gamma_l$, and 
 $\hat \psi =0$ on $\Lambda_\delta \times \widetilde{R}_{av}$
  as test functions   for  the microscopic problem consisting of equations \eqref{Stokes_Omega}, \eqref{Stokes_BC_O}, \eqref{Stokes_ContC}, and \eqref{Stokes_Lambda_delta}--\eqref{Stokes_BC_L_delta}.
 Using the {\it a priori} estimates  and applying the two-scale limit, we obtain 
 that $p_a^\delta=p_a^\delta(x)$, $p_v^\delta=p^\delta_v(x)$ in $\Omega$ and  $\hat p^\delta= \hat p^\delta(x)$ in $\Lambda_\delta$.

Choosing now  $\phi_l(x)=\psi_l\left(x, \frac x \ve\right)$ and  $\hat \phi(x)=\hat \psi\left(x, \frac x \ve\right)$ as test functions,  where     $\psi_l \in C^\infty_0(\Omega, C^\infty_{\text{per}}(Y))$ and  $\hat \psi \in C^\infty_0(\Lambda_\delta; C^\infty_{\text{per}}(\widetilde{Z}))$ with  $\dv_y \psi_l =0$ and  $\dv_y \hat \psi =0$, as well as $\psi_l = 0$ on $\Omega\times \Gamma_l$ and 
 $\hat \psi =0$ on $\Lambda_\delta \times \widetilde{R}_{av}$,   we have 
\begin{equation}\label{macro_72}
\begin{aligned}
&\sum_{l=a,v} \frac 1{|Y|}\Big[\langle  2 \mu S_y \v^\delta_l,  S_y \psi_l \rangle_{\Omega\times Y_l} -  \langle  p^\delta_l,  \dv_x \psi_l \rangle_{\Omega\times Y_l}  - \frac 1{L} \langle p_l^0, \psi_{l,n} \rangle_{\Omega\times Y_l} \Big] \\
&\;  +
\frac 1 {\delta |\widetilde{Z}|}\left[\langle 2 \mu S_y \hat \v^\delta_{av},  S_y \hat \psi  \rangle_{\Lambda_\delta \times \widetilde{Z}_{av}}    -
\langle   \hat p^\delta,  \dv_{x} \hat \psi \rangle_{\Lambda_\delta \times \widetilde{Z}_{av}} \right].
\end{aligned}
\end{equation}
Using the divergence-free property of the velocity fields $\v^\ve_l$ and $\hat \v^\ve_l$, we obtain that  
\begin{equation}\label{div_delta_macro}
\begin{aligned}
& \dv_y \v^\delta_l =0  && \text{ in } \Omega\times Y_l,  &&
\dv\langle \v^\delta_l, 1\rangle_{Y_l}=0  && \text{ in } \Omega, \quad l=a,v, \\
& \dv_y \hat \v^\delta_{l} =0 && \text{ in }  \Lambda_\delta\times \widetilde{Z}_{l},  &&
\dv \langle \hat \v^\delta_{av}, 1  \rangle_{\widetilde{Z}_{av}}=0 && \text{ in } \Lambda_\delta.
\end{aligned} 
\end{equation}
Then considering  $\psi \in C^\infty(\overline\Omega)$ with $\psi(x)=0$ on $\partial \Omega \setminus \hat \Lambda$, and using the two-scale convergence of $\v^\ve_l$, we have 
\begin{eqnarray*}
0=-\lim\limits_{\ve \to 0 } \langle \dv \, \v_l^\ve, \psi  \rangle_{\Omega_l^\ve} =
|Y|^{-1}\langle \v_l^\delta\cdot \n,   \psi \rangle_{\hat \Lambda\times Y_l} - \lim\limits_{\ve \to 0 } \langle \v_l^\ve\cdot \n,  \psi \rangle_{\hat  \Lambda\cap \partial\Omega_l^\ve} .
\end{eqnarray*}
For $\hat \psi \in C^\infty(\overline \Lambda_\delta)$ with $\hat \psi(x) =0$ on $\partial\Lambda_\delta \setminus \hat \Lambda$,   and using $\hat \v^\ve_l =\delta \v^\ve_l$ on $\hat \Lambda\cap \partial \Omega_l^\ve$, we obtain
\begin{eqnarray*}
\lim\limits_{\ve \to 0 }  \langle \dv \, \hat \v_{av}^\ve, \hat \psi \rangle_{\Lambda_{av,\delta}^\ve} 
=  \lim\limits_{\ve \to 0 }  \left[\langle  \delta  \v_a^\ve \cdot \n,  \hat  \psi \rangle_{\partial \Omega_{a}^\ve\cap \hat \Lambda} + 
 \langle  \delta \v_v^\ve \cdot \n,  \hat \psi  \rangle_{\partial \Omega_{v}^\ve\cap \hat \Lambda}
 - \langle  \hat \v_{av}^\ve, \nabla \hat \psi \rangle_{\Lambda_{av, \delta}^\ve}\right]
\\ = 
 \langle \delta |Y|^{-1} \v_a^\delta\cdot \n,  \hat \psi  \rangle_{\hat\Lambda\times Y_a} + \langle \delta |Y|^{-1} \v_v^\delta\cdot \n,  \hat \psi  \rangle_{\hat\Lambda\times Y_v}  - \langle  |\widetilde{Z}|^{-1} \hat  \v^\delta_{av}\cdot \n,  \hat \psi  \rangle_{\hat\Lambda\times \widetilde{Z}_{av}} =0.
  \end{eqnarray*}
Considering $\psi \in C^\infty(\overline\Omega)$ and   $\hat\psi \in C^\infty(\overline\Lambda_\delta)$  with $\psi(x)=0$, $\hat \psi(x)=0$  on $\hat \Lambda$ and $\psi(x) =0$ on $\Gamma_D$,   and applying the divergence-free property of velocity fields and the boundary conditions we obtain that 
$$
 \langle  \v_l, 1 \rangle_{Y_l} \cdot \n=  0 \text{ on } \partial \hat \Omega \times (-L,0), \;\;
   \; \;  \langle \hat \v_{av}, 1\rangle_{\widetilde{Z}_{av}} \cdot \n=  0 \text{ on } \partial \hat \Omega \times (0,\delta)\cup \hat \Omega\times\{\delta\}. 
$$
 By applying integration by parts in \eqref{macro_72},  and  employing the fact  that the divergence-free space is orthogonal to the space of gradients of functions, we obtain (in the same maner as in section~\ref{macro_velocity_1}) the  macroscopic model 
\begin{equation}\label{two-scale_delta_1}
\begin{aligned}
& - \mu \Delta_y \v^\delta_l + \nabla p^\delta_l + \nabla_y p^{1, \delta}_l =0, \qquad && \dv_y \v_l^\delta =0 \; \;  && \text{ in } \Omega\times Y_l, \\
& - \mu \Delta_y \hat \v^\delta_{av} + \nabla \hat p^\delta + \nabla_y \hat p^{1, \delta}_{av} =0, \qquad && \dv_y \hat \v^\delta_{av} =0 \; \;  && \text{ in } \Lambda_\delta\times \widetilde{Z}_{av}, \\
&\phantom{-} \,  \v_l^\delta  = 0  \qquad  \quad \qquad \text{ on } \; \;  \Omega \times \Gamma_l, \qquad && \hat \v^\delta_{av} =0 \quad &&\text{ on } \Lambda_\delta \times \widetilde{R}_{av,\delta}\\
& \frac 1 {|Y|}\sum_{l=a,v} \langle \v_l^\delta, 1 \rangle_{Y_l}  \cdot \n  =  \frac 1{\delta |\widetilde{Z}|}\langle \hat  \v^\delta_{av},1 \rangle_{\widetilde{Z}_{av}} \cdot\n, &&p_l^\delta =  \hat p^\delta \; \; && \text{ on }  \hat \Lambda , \\
&  \langle \v_l^\delta \cdot \n, 1 \rangle_{Y_l} = 0 \quad \quad \text{ on } \; \; \partial \Omega \setminus (\Gamma_D \cup \hat \Lambda), &&  p_l^\delta = p_l^0 \; \;  && \text{ on }  \Gamma_D,\\
& \langle \hat \v^\delta_{av} \cdot \n, 1\rangle_{\widetilde{Z}_{av}}  = 0 \quad  \text{ on }\; \;  \partial \Lambda_\delta \setminus \hat \Lambda,\\
& \v^\delta_l, \; p^{1,\delta}_l\quad  \qquad Y-\text{periodic}, && \hat \v^\delta_{av}, \; \hat p^{1, \delta}_{av} &&  \widetilde{Z}-\text{periodic},
\end{aligned}
\end{equation}
where $p_l^{1,\delta} \in L^2(\Omega; L^2(Y_l)/\mathbb R)$, $\hat p_{av}^{1,\delta} \in L^2(\Lambda_\delta; L^2(\widetilde{Z}_{av})/\mathbb R)$, and $l=a,v$. We now consider the ansatz
\begin{eqnarray*}
&&\v_l^\delta(x,y)= - \sum\limits_{i=1}^n \partial_{x_i} p_l^\delta(x) \, \omega^i_l(y), \qquad \qquad
p^{1,\delta}_l(x,y) =  -\sum\limits_{i=1}^n \partial_{x_i} p_l^\delta(x) \,  \pi^i_l(y), \\
&&\hat \v^\delta_{av}(x,y)= - \sum\limits_{i=1}^n \partial_{x_i} \hat p^\delta(x) \, \widetilde \omega^i(y), 
\qquad \quad\;\; 
\hat p^{1,\delta}_{av}(x,y) = - \sum\limits_{i=1}^n \partial_{x_i} \hat p^\delta(x) \,  \widetilde \pi^i(y), 
\end{eqnarray*}
 where $l=a,v$, and $(\omega^i_l, \pi_l^i)$ and $(\widetilde \omega^i, \widetilde \pi^i)$ are solutions of the unit cell problems  \eqref{eq:omega1} and \eqref{unit_cell_7_2}.
Using these along  with \eqref{two-scale_delta_1} and \eqref{div_delta_macro}  we obtain the macroscopic equations in \eqref{macro_ve_1}, where  $\overline \v^\delta_{l}(\cdot) = |Y|^{-1} \int_{Y_l} \v_l^\delta(\cdot, y) dy$  and $\widetilde \v^\delta_{av}(\cdot) = |\widetilde{Z}|^{-1} \int_{\widetilde{Z}_{av}} \hat \v_{av}^\delta(\cdot, y) dy$.

We remark that similar results have been obtained in \cite{AndroJaeger}. We also note that the Dirichlet boundary conditions on $\Gamma_D$ ensure the uniqueness of the solution of problem \eqref{macro_ve_1}.
\end{proof}

\begin{proof}[\bf Proof of Theorem \ref{th:macro_delta}]
We rewrite the equations in \eqref{macro_ve_1} in weak form: 
 \begin{eqnarray}\label{delta_weak_1}
 \langle \mathcal K_a \nabla p^\delta_a, \nabla \phi_a \rangle_\Omega +  \langle \mathcal K_v \nabla p^\delta_v, \nabla \phi_v \rangle_\Omega + 
 \frac 1 {\delta}  \langle \widetilde {\mathcal K} \nabla \hat p^\delta, \nabla \hat \phi \rangle_{\Lambda_\delta} = 0
  \end{eqnarray}
   for $\phi_l \in W(\Omega)$,  $\hat \phi \in H^1(\Lambda_\delta)$ and $\phi(x) = \hat \phi(x)$ for a.a. $x\in \hat \Lambda$.  Considering $p^\delta_l + \frac{x_n} L p_l^0$ and $\hat p^\delta$ as test functions in \eqref{delta_weak_1}, and using the continuity condition $p_l^\delta= \hat p^\delta$ on $\hat \Lambda$,  we obtain  
   \begin{eqnarray*}
   \|p_l^\delta \|_{H^1(\Omega)} \leq C , \quad \frac 1\delta \|\hat p^\delta \|_{H^1(\Lambda_\delta)} \leq C.
   \end{eqnarray*}
 Hence, considering $ \tilde p^\delta (\hat x, y_n)= \hat p^\delta(\hat x, \delta y_n)$, we obtain that
  $$
  \| \tilde p^\delta \|_{L^2(\hat \Lambda \times (0,1))} \leq C , \quad  \|\nabla_{\hat x}\tilde p^\delta \|_{L^2(\hat \Lambda \times (0,1))} \leq C, \quad  \| \nabla_{y_n}  \tilde p^\delta \|_{L^2(\hat \Lambda \times (0,1))} \leq C\delta,
  $$
  and  there exist  subsequences, denoted again by $p_l^\delta$ and $ \tilde p^\delta$, and functions $p_l \in H^1(\Omega)$,   $\hat p \in H^1(\hat \Lambda\times (0,1))$, $\hat p^1\in L^2(\hat \Lambda; H^1(0,1))$, with $\hat p$ being constant in $x_n$, such that 
   \begin{eqnarray*}
   p_l^\delta  \rightharpoonup p_l\; \text{ in }\; H^1(\Omega), \;\;  \tilde p^\delta   \rightharpoonup \hat p , \;\nabla_{\hat x}  \tilde p^\delta   \rightharpoonup  \nabla_{\hat x} \hat p, \;   \delta^{-1} \partial_{y_n}\tilde p^\delta  \rightharpoonup  \partial_{y_n}\hat p^1\;  \text{ in } \;  L^2(\hat \Lambda\times(0,1)). 
   \end{eqnarray*}
   
     The continuity of pressures implies the boundary conditions for $p_a$ and $p_v$ in \eqref{macro_delta_ve}.
Considering $\phi_l \in C^\infty_0(\Omega)$ and $\hat \phi =0$ as  test functions in \eqref{delta_weak_1},  and using the weak convergence of $p_l^\delta$, where $l=a,v$, we obtain the equations 
for  $p_a$ and $p_v$ in \eqref{macro_delta_ve}.

We now consider the test functions  $\phi_l \in C^\infty(\overline\Omega)\cap W(\Omega)$  and    $\hat \phi(x)= \hat \phi_1(\hat x) +\delta \hat \phi_2(\hat x, x_n/\delta)$ with $\hat \phi_1 \in C^\infty(\overline{\hat \Lambda})$, $\hat \phi_2 \in C^\infty_0(\hat \Lambda\times (0,1))$   and  $\phi_l(x)=\hat \phi_1(\hat x)$ on $\hat \Lambda$. Using these in \eqref{delta_weak_1} and taking the limit as $\delta \to 0$
 we obtain 
   \begin{eqnarray*}
&&\sum_{l=a,v} \langle \mathcal K_l \nabla p_l \cdot \n,  \hat \phi_1 \rangle_{\hat \Lambda} + 
    \langle \widetilde {\mathcal K}( \nabla_{\hat x}  \hat p + \partial_{y_n} \hat p^1{\bf e}_n), \nabla_{\hat x} \hat  \phi_1 +
    \partial_{y_n} \hat \phi_2 {\bf e}_n \rangle_{\hat \Lambda\times(0,1)} =  0.
  \end{eqnarray*}
 Taking $\hat \phi_1 =0$ and using the fact that $\widetilde {\mathcal K}$ does not depend on $y_n$ imply that $\hat p^1$ is constant with respect to $y_n$.
Finally,  by considering first $\hat \phi_1 \in C^\infty_0({\hat \Lambda})$  and then $\hat \phi_1 \in C^\infty(\overline{\hat \Lambda})$, we  derive the macroscopic equation and boundary conditions  for $\hat p$ in \eqref{macro_delta_ve}. 
\end{proof}

\subsection{Derivation of macroscopic equations for oxygen concentrations}
We now turn our attention to the oxygen concentrations in arterial blood, venous blood, and tissue, under the scaling assumption    $0<\ve << \delta<<1$ that was delineated in section \ref{section7}. Theorem \ref{thm7-2} provides the macroscopic equations for these quantities as $\ve \to 0$ while keeping $\delta$ fixed.

 We consider the same microscopic equations as in  \eqref{Diff_AV}--\eqref{Diff_InitC}  with the scaling $1/\delta$ instead of $1/\ve$ in the transmission conditions \eqref{Diff_TransC}. Also, for the initial data, we assume that $\delta^{-1} \| \hat c^{\delta,0}_l\|^2_{H^2(\Lambda_\delta)} +  \| \hat c^{\delta,0}_l\|_{L^\infty(\Lambda_\delta)} \leq C$ instead of the corresponding assumption  on the $H^2(\Lambda^\ve)$ and $L^\infty(\Lambda^\ve)$-norms. 

\begin{proof}[\bf Proof of Theorem~\ref{thm7-2}]
 Similarly to  Lemma~\ref{Lemma:aprior} in Section~\ref{section4}  we can prove {\it a priori} estimates and convergence results for $c^\ve_l$ and $\hat c^\ve_l$, where  $l=a,v,s$.  
 We consider $\psi^\ve_l(t,x)=\phi^1_l(t,x)+ \ve \phi^2_l(t,x, x/\ve)$ and $\hat \psi^\ve(t,x)= \hat \phi_1(t,x)+ \ve \hat \phi_2(t,x, x/\ve)$, with $\phi^1_l \in C^\infty(\overline \Omega_T)\cap L^2(0,T; W(\Omega))$,  $\phi^2_l \in C^\infty_0(\Omega_T, C^\infty_{\text{per}}(Y))$, $\hat \phi_1 \in C^\infty(\overline \Lambda_{\delta, T})$,  and $\hat \phi_2 \in C^\infty_0(\Lambda_{\delta, T}, C^\infty_{\text{per}}(\widetilde{Z}))$, as test functions in \eqref{micro_weak_av} and \eqref{micro_weak_tissue}.
 Similarly to  the proof of Theorem~\ref{6-1},  using the convergence of $\mathcal T^\ast_\ve(c_l^\ve)$ and 
 $\mathcal T^\ast_\ve(\hat c_j^\ve)$,  
   along with the two-scale convergence of 
$\v^\ve_l$ and $\hat \v^\ve_l$,  and letting $\ve \to 0$ yield
\begin{eqnarray*}
\frac 1 {|Y|}\sum_{l=a,v} \langle \partial_t c_l^\delta,  \phi^1_l \rangle_{\Omega_T\times Y_l} + 
 \langle D_l(y)(\nabla c_l^\delta + \nabla_y c_l^{1,\delta}) -   \v_l^\delta  c_l^\delta, 
\nabla \phi^1_l + \nabla_y \phi_l^2 \rangle_{\Omega_T\times Y_l}  \\
+\frac 1 \delta \frac 1 {|\widetilde Z|} \Big[ \langle \partial_t \hat c^\delta_{av} , \hat \phi_1\rangle_{\Lambda_{\delta,T}\times \widetilde{Z}_{av}}  +
 \langle \hat D_{av}(y)(\nabla \hat c^\delta_{av} + \nabla_y \hat c^{1, \delta}_{av}) - \hat\v_{av}^\delta  \hat c^\delta_{av}, 
\nabla \hat \phi_1 + \nabla_y \hat \phi_2\rangle_{\Lambda_{\delta,T}\times \widetilde{Z}_{av}}  \Big]
\\=   \frac 1 {|Y|} \sum_{l=a,v}
\langle \lambda_l( c_s^\delta- c_l^\delta),  \phi^1_l  \rangle_{\Omega_T\times \Gamma_l} +
\frac 1 \delta \frac 1 {|\widetilde Z|}  \sum_{l=a,v} \langle  \lambda_l(\hat c_s^\delta- \hat c^\delta_{av}), \hat \phi_1\rangle_{\Lambda_{\delta,T}\times\widetilde{R}_l}.
 \end{eqnarray*}
In order to derive the macroscopic model \eqref{macro_av_ve} we proceed in a similar way as in the proof of Theorem~\ref{6-1}.  Choosing  first  $\phi^1_l=0$ and $\hat \phi_1=0$  and applying the divergence-free property and the boundary conditions for the velocity fields  we obtain 
\begin{eqnarray*}
 && \langle D_l(y)(\nabla c_l^\delta + \nabla_y c_l^{1,\delta}), 
 \nabla_y \phi_l^2 \rangle_{\Omega_T\times Y_l} =0, \qquad l = a,v,  \\
&& \frac 1 \delta \langle \hat D_{av}(y)(\nabla \hat c^\delta_{av} + \nabla_y \hat c^{1, \delta}_{av}),  \nabla_y \hat \phi_2\rangle_{\Lambda_{\delta,T}\times \widetilde{Z}_{av}} = 0. 
\end{eqnarray*}
Then we consider the ansatz 
\begin{eqnarray*}
c^1_l(t,x,y)= \sum_{j=1}^n \partial_{x_j} c_l(t,x) w^j_l(y)\; \;  \text{ and } \; \; \hat c^1_{av}(t,x,y)= \sum_{j=1}^n \partial_{x_j} \hat c(t,x) \widetilde w^j_{av}(y)\, ,
\end{eqnarray*}
where $w^j_l$ and $\widetilde w^j_{av}$ are solutions of the unit cell problems \eqref{UnitCell_Ox_1} and \eqref{unit_cell_2_hat}, 
and we  take $\phi_l^2 =0$ and $\hat \phi_2=0$ to arrive at 
 the macroscopic equations \eqref{macro_av_ve}.

The macroscopic equations \eqref{macro_s_ve} for the oxygen concentration  in tissue are derived in a similar manner.
Standard arguments pertaining to the difference of two solutions imply the uniqueness of the solutions of the macroscopic model consisting of equations \eqref{macro_av_ve}  and  \eqref{macro_s_ve}.
\end{proof}

\begin{proof}[\bf Proof of Theorem \ref{th:macro_ox_delta}]
Similarly to Lemma~\ref{Lemma:aprior} we can derive {\it a priori} estimates for   $c^\delta_l$ and $\hat c^\delta_m$,   
\begin{eqnarray}\label{a_priori_delta}
\begin{aligned}
&\| c^\delta_l \|_{L^\infty(0,T; H^1(\Omega))} + 
\frac 1 \delta \| \hat c^\delta_m \|_{L^\infty(0,T; H^1(\Lambda_\delta))}  \leq C, \\
&  \| \tilde c^\delta_m \|_{L^2(\Lambda^1_T)} +  \| \nabla_{\hat x}\tilde c^\delta_m \|_{L^2(\Lambda^1_T)} \leq C, \qquad 
   \|  \nabla_{y_n}\tilde c^\delta_m \|_{L^2(\Lambda^1_T)}\leq C \delta, \\
& \|\partial_t c^\delta_l \|_{L^2(\Omega_T)}   
 +\frac 1{\delta} \|\partial_t \hat c^\delta_m \|_{L^2(\Lambda_{\delta,T})} + 
  \| \partial_t \tilde c^\delta_m \|_{L^2(\Lambda^1_T )}  \leq C
   \end{aligned}
\end{eqnarray}
for $l=a,v,s$, $\,m=av, s$, where  $\tilde c^\delta_m(t, \hat x, y_n) = \hat c^\delta_m(t, \hat x, \delta y_n)$, $\Lambda^1_T=\hat \Omega\times (0,1)\times(0,T)$,  and the  constant $C$ is independent of $\delta$. 
Thus there exist functions $c_l \in L^2(0,T; H^1(\Omega)) \cap H^1(0,T; L^2(\Omega))$,   $\hat c_m \in L^2(0,T; H^1(\Lambda^1))\cap H^1(0,T; L^2(\Lambda^1))$, and $\hat c_m^1 \in L^2(\hat \Lambda_T; H^1(0,1))$,  with $\hat c_m$ being independent of $x_n$,    such that 
\begin{eqnarray}\label{convergence}
\begin{aligned}
&c^\delta_l \rightharpoonup c_l  && \text{ in } L^2(0,T; H^1(\Omega)),   
 && \partial_t  c^\delta_l \rightharpoonup \partial_t c_l   && \text{ in } L^2(\Omega_T),  \\
&\tilde c^\delta_m \rightharpoonup \hat c_m && \text{ in } L^2(0,T; H^1(\Lambda^1)),   
&& \partial_t  \tilde c^\delta_m \rightharpoonup \partial_t \hat c_m  && \text{ in } L^2(\Lambda^1_T), \quad\\
& c^\delta_l \to c_l  && \text{ in } L^2(\Omega_T),   && \tilde c^\delta_m \to \hat c_m && \text{ in } L^2(\Lambda^1_T),\\
& \delta^{-1} \partial_{y_n} \tilde c^\delta_m \rightharpoonup \partial_{y_n} \hat c^1_m && \text{ in } L^2(\Lambda^1_T), 
 \end{aligned}
\end{eqnarray}
where $l=a,v,s$ and $m=av, s$. Finally,  we use  test functions 
\begin{enumerate}
\item[(a)] $\phi_l \in C^\infty_0(\Omega_T)$  and $\hat \phi =0$, and 
\item[(b)]  $\phi_l \in C^\infty(\overline \Omega_T)$,  $\hat \phi(t,x) = \hat \phi_1(t,\hat x) +\delta \hat \phi_2(t,\hat x, x_n/\delta)$,  with $\hat \phi_1 \in C^\infty_0({\hat \Lambda}_{T})$,   $\hat \phi_2 \in C^\infty_0({\hat \Lambda}_{T}\times (0,1))$, and $\phi_l(t,x) = \hat \phi_1(t,x)$ on $\hat \Lambda_T$  
\end{enumerate}
in that order.  In the same way as in the proof of Theorem~\ref{th:macro_delta},  using the convergence results in \eqref{convergence}, along with the convergence of $\overline \v^\delta_l$ and $\widetilde \v^\delta_{av}$ (ensured by the convergence of $\nabla p^\delta_l$ and $\nabla \hat p^\delta$),   taking the limit as $\delta\to 0$,  and applying the fact that $\widetilde{\mathcal  A}_m$ are independent of $y_n$,    we obtain the limit equations in \eqref{macro_delta_av}  and \eqref{macro_delta_s}. 
The continuity conditions  for  $c^\delta_l$ and $\hat c^\delta_j$ on $\hat \Lambda_T$ ensure the continuity conditions for the limit functions $c_l$, $\hat c_j$  for $l=a,v,s$, $j = av, s$.
 The assumptions on the initial data ensure the existence of $\hat c^{0}, \hat c^0_s \in H^1(\hat \Lambda)$  such that $\hat c^{0,\delta}(\hat x, \delta y_n) \to \hat c^0(\hat x)$ and $\hat c^{0,\delta}_s(\hat x, \delta y_n) \to \hat c^0_s(\hat x)$ in $L^2(\hat \Lambda\times(0,1))$. 
Then, using the convergence of $\partial_t c_l^\delta$ and $\partial_t \tilde c_m^\delta$, we obtain that the initial conditions for $c_l$ and $\hat c_m$ are satisfied.
 Standard arguments imply the uniqueness of the solution of the macroscopic model consisting of equations \eqref{macro_delta_av}  and \eqref{macro_delta_s}. 
\end{proof}

\bibliographystyle{plain}
\bibliography{fluid} 

\end{document}